\documentclass{article}
\usepackage{amssymb,amsmath}
\usepackage{amsthm}
\usepackage{hyperref}
\usepackage{mathrsfs}
\usepackage{textcomp}
\usepackage{verbatim}

\makeatletter

\newdimen\bibspace
\renewenvironment{thebibliography}[1]{%
 \section*{\refname 
       \@mkboth{\MakeUppercase\refname}{\MakeUppercase\refname}}%
     \list{\@biblabel{\@arabic\c@enumiv}}%
          {\settowidth\labelwidth{\@biblabel{#1}}%
           \leftmargin\labelwidth
           \advance\leftmargin\labelsep
           \itemsep\bibspace
           \parsep\z@skip     %
           \@openbib@code
           \usecounter{enumiv}%
           \let\p@enumiv\@empty
           \renewcommand\theenumiv{\@arabic\c@enumiv}}%
     \sloppy\clubpenalty4000\widowpenalty4000%
     \sfcode`\.\@m}
    {\def\@noitemerr
      {\@latex@warning{Empty `thebibliography' environment}}%
     \endlist}

\makeatother

\makeatletter
\newtheorem{thm}{Theorem}[section]
\newtheorem{lem}{Lemma}[section]
\newtheorem{prop}{Proposition}[section]
\newtheorem{defn}{Definition}[section]

\newtheorem{cor}{Corollary}[section]
\newtheorem{rem}{Remark}[section]

\numberwithin{equation}{section}

\def\XXint#1#2#3{{\setbox0=\hbox{$#1{#2#3}{\int}$}
  \vcenter{\hbox{$#2#3$}}\kern-.5\wd0}}

\newcommand{\al}{\alpha}                \newcommand{\lda}{\lambda}
\newcommand{\om}{\Omega}                \newcommand{\pa}{\partial}
\newcommand{\va}{\varepsilon}           \newcommand{\ud}{\mathrm{d}}
\newcommand{\be}{\begin{equation}}      \newcommand{\ee}{\end{equation}}
                 \newcommand{\X}{\overline{X}}
              \newcommand{\B}{\mathcal{B}}
\newcommand{\R}{\mathbb{R}}              \newcommand{\Sn}{\mathbb{S}^n}
\newcommand{\D}{\dot{H}^\sigma(\R^n)}  

\newcommand{\dlim}{\displaystyle\lim}
\newcommand{\dsup}{\displaystyle\sup}
\newcommand{\dmin}{\displaystyle\min}
\newcommand{\dmax}{\displaystyle\max}
\newcommand{\dinf}{\displaystyle\inf}
\newcommand{\dsum}{\displaystyle\sum}

\begin{document}

\title{\textbf{On a fractional Nirenberg problem, part I: blow up analysis and compactness of solutions}
\bigskip}

\author{\medskip Tianling Jin\footnote{Supported in part by a University and Louis Bevier Dissertation Fellowship at Rutgers University and by Rutgers University School of Art and Science Excellence Fellowship.}, \ \
YanYan Li\footnote{Supported in part by NSF grant DMS-0701545 and DMS-1065971 and by Program for Changjiang Scholars and Innovative Research Team in University in China.}, \ \
Jingang Xiong\footnote{Supported in part by
CSC project for visiting Rutgers University and NSFC No. 11071020.}}

\date{\today}

\maketitle

\tableofcontents
\section{Introduction} 

The Nirenberg problem concerns the following: For which positive function $K$ on the standard sphere $(\Sn, g_{\Sn}), n\geq 2$, there
exists a function $w$ on $\Sn$ such that the scalar curvature (Gauss curvature in dimension $n=2$) $R_g$ of the conformal metric $g=e^{w}g_{\Sn}$ is equal to
$K$ on $\Sn$? The problem is equivalent to solving
\[
-\Delta_{g_{\Sn}}w+1=Ke^{2w}, \quad\mbox{on } \mathbb{S}^2,
\]
and
\[
-\Delta_{g_{\Sn}}v+c(n)R_0v=c(n)Kv^{\frac{n+2}{n-2}}, \quad\mbox{on } \Sn\mbox{ for }n\geq 3,
\]
where $c(n)=(n-2)/(4(n-1)), R_0=n(n-1)$ is the scalar curvature of $(\Sn, g_{\Sn})$ and $v=e^{\frac{n-2}{4}w}$.

The first work on the problem is by D. Koutroufiotis \cite{K}, where the solvability on $\mathbb{S}^2$ is established
when $K$ is assumed to be an antipodally symmetric function which is close to $1$. Moser \cite{Ms}
established the solvability on $\mathbb{S}^2$ for all antipodally symmetric functions $K$ which is positive somewhere. Without assuming any symmetry assumption on $K$, sufficient conditions were given in dimension
$n=2$ by Chang and Yang \cite{CY87} and \cite{CY88}, and in dimension $n=3$ by Bahri and Coron \cite{BC}.
Compactness of all solutions in dimensions $n=2,3$ can be found in work of Chang, Gursky and Yang \cite{CGY}, Han \cite{H1} and Schoen and Zhang \cite{SZ}. In these dimensions, a sequence of solutions
can not blow up at more than one point. Compactness and existence of solutions in higher dimensions were studied by Li in \cite{Li95} and \cite{Li96}.
The situation is very different, as far as the compactness issues are concerned: In dimension $n\geq 4$,
a sequence of solutions can blow up at more than one point, as shown in \cite{Li96}. There have been many papers on the problem and related ones, see, e.g.,
\cite{AB1, AB2, AH, BLR, B1, BA, BrN, By, CNY, CZ, CGY, CY87, CY88, CY91, CL, CD, ChLi, CL2, CX, D, EM, H1, HL1, HL2, J1, J2, Li93, LL, M, S1, S2, WY, Zhang, Zhu}.

In \cite{GJMS}, Graham, Jenne, Mason and Sparling constructed a sequence of conformally covariant elliptic operators, $\{P_{k}^g\}$, on Riemannian manifolds for all positive integers $k$ if $n$ is odd, and for $k\in \{1,\cdots, n/2\}$ if $n$ is even.
Moreover, $P_1^g$ is the conformal Laplacian $-L_g:=-\Delta_{g}+c(n)R_g$ and $P_2^g$ is the Paneitz operator. The construction in \cite{GJMS} is based on the ambient metric construction of \cite{FG}. Up to positive constants
$P_1^g(1)$ is the scalar curvature of $g$ and $P_2^g(1)$ is the $Q$-curvature. Prescribing $Q$-curvature problem on $\Sn$ was studied extensively, see, e.g., \cite{Be, DM, D1, D2, Felli, WX, WX2}.

Making use of a generalized Dirichlet to Neumann map, Graham and Zworski \cite{GZ} introduced a meromorphic family of conformally invariant operators on the conformal infinity of asymptotically hyperbolic manifolds. Recently,  Chang and Gonz\'alez \cite{CG} reconciled the way of Graham and Zworski to define conformally invariant operators $P_{\sigma}^g$ of non-integer order $\sigma\in (0,\frac{n}{2})$ and the localization method of  Caffarelli and Silvestre \cite{CS2} for factional Laplacian $(-\Delta)^{\sigma}$ on the Euclidean space $\R^n$. These lead naturally to a fractional order curvature $R^g_{\sigma}:=P_{\sigma}^g(1)$, which will be called $\sigma$-curvature in this paper. A typical example
is that standard conformal spheres $(\Sn, [g_{\Sn}])$ are the conformal infinity of Poincar\'e
disks $(\mathbb{B}^{n+1}, g_{\mathbb{B}^{n+1}})$. In this case, $\sigma$-curvature can be expressed in the following explicit way. Let $g$ be a representative in the conformal class $[g_{\Sn}]$ and write $g=v^{\frac{4}{n-2\sigma}}g_{\Sn}$,
where $v$ is positive and smooth on $\Sn$. Then the $\sigma$-curvature for $(\Sn, g)$ can be computed as
\be\label{def of sig cur}
R^{\sigma}_g=v^{-\frac{n+2\sigma}{n-2\sigma}}P_{\sigma}(v),
\ee
where $P_\sigma$ is an \emph{intertwining operator} and
\be\label{P sigma}
 P_\sigma=\frac{\Gamma(B+\frac{1}{2}+\sigma)}{\Gamma(B+\frac{1}{2}-\sigma)},\quad B=\sqrt{-\Delta_{g_{\Sn}}+\left(\frac{n-1}{2}\right)^2},
\ee
$\Gamma$ is the Gamma function and $\Delta_{g_{\Sn}}$ is the Laplace-Beltrami operator on $(\Sn, g_{\Sn})$. The operator $P_{\sigma}$ can be seen more concretely on $\R^n$ using stereographic projection. The stereographic projection from $\Sn\backslash \{N\}$ to $\R^n$ is the inverse of
\[
F: \R^n\to \Sn\setminus\{N\}, \quad x\mapsto \left(\frac{2x}{1+|x|^2}, \frac{|x|^2-1}{|x|^2+1}\right),
\]
where $N$ is the north pole of $\Sn$. Then
\be\label{eq:equ of projection}
(P_\sigma(\phi))\circ F=  |J_F|^{-\frac{n+2\sigma}{2n}}(-\Delta)^\sigma(|J_F|^{\frac{n-2\sigma}{2n}}(\phi\circ F)),\quad \mbox{for }\phi\in C^{\infty}(\Sn)
\ee
where
\[
|J_F|=\left(\frac{2}{1+|x|^2}\right)^n,
\]
and $(-\Delta)^\sigma$ is the fractional Laplacian operator (see, e.g., page 117 of \cite{S}).
When $\sigma\in (0,1)$, Pavlov and Samko \cite{PS} showed that
\be\label{description of P sigma}
P_{\sigma}(v)(\xi)=P_{\sigma}(1)v(\xi)+c_{n,-\sigma}\int_{\Sn}\frac{v(\xi)-v(\zeta)}{|\xi-\zeta|^{n+2\sigma}}\,\ud vol_{g_{\Sn}}(\zeta)
\ee
 for $v\in C^{2}(\Sn)$, where $c_{n,-\sigma}=\frac{2^{2\sigma}\sigma\Gamma(\frac{n+2\sigma}{2})}{\pi^{\frac{n}{2}}\Gamma(1-\sigma)}$ and $\int_{\Sn}$ is understood as $\lim\limits_{\va\to 0}\int_{|x-y|>\va}$.

 For the $\sigma$-curvatures
on general manifolds we refer to  \cite{GZ}, \cite{CG}, \cite{GQ} and references therein.
Corresponding to the Yamabe problem,
 fractional Yamabe problems for $\sigma$-curvatures are studied in \cite{GMS}, \cite{GQ} and \cite{QR}, and fractional Yamabe flows
on $\Sn$ are studied in \cite{JX}.

From \eqref{def of sig cur}, we consider
\be\label{main equ}
P_\sigma(v)=c(n,\sigma)K v^{\frac{n+2\sigma}{n-2\sigma}},\quad \mbox{on } \Sn,
\ee
where $c(n,\sigma)=P_\sigma(1)$, and $K>0$ is a continuous function on $\Sn$. When $K=1$,
\eqref{main equ} is the
Euler-Lagrange equation for a functional associated to the fractional Sobolev  inequality on $\Sn$ (see \cite{Be}), and all positive solutions must be of the form
\be \label{standard bubble on sn}
 v_{\xi_0, \lda}(\xi)=\left(\frac{2\lda}{2+(\lda^2-1)(1-\cos dist_{g_{\Sn}}(\xi, \xi_0))}\right)^{\frac{n-2\sigma}{2}}, \quad   \xi\in \Sn
\ee
for some $\xi_0\in \Sn$ and positive constant $\lda$. This classification can be found in \cite{Lie83}, \cite{CLO} and \cite{Li04}.
In general, \eqref{main equ} may have no positive solution, since if $v$ is a positive solution of \eqref{main equ} with $K\in C^1(\Sn)$
then it has to satisfy the Kazdan-Warner type condition
\be\label{K-W condition}
 \int_{\Sn}\langle \nabla_{g_{\Sn}} K, \nabla_{g_{\Sn}} \xi\rangle v^{\frac{2n}{n-2\sigma}}\,\ud \xi=0.
\ee
Consequently if $K(\xi)=\xi_{n+1}+2$, \eqref{main equ} has no solutions.
The proof of \eqref{K-W condition} is provided in Appendix \ref{A Kazdan-Warner identity}.

In this and a subsequent paper \cite{JLX2}, we study \eqref{main equ} with $\sigma\in (0,1)$, a fractional Nirenberg problem. Throughout the paper, we assume that $\sigma\in (0,1)$ without otherwise stated.

\begin{defn}
For $d>0$, we say that $K\in C(\Sn)$ has flatness order greater than $d$ at $\xi$ if, in some
local coordinate system $\{y_1,\cdots, y_n\}$ centered at $\xi$,
there exists a neighborhood $\mathscr{O}$ of $0$ such that $K(y)=K(0)+o(|y|^{d})$ in $\mathscr{O}$.
\end{defn}

\begin{thm}\label{K-M-E-S}
Let $n\ge 2$, and $K\in C^{1,1}(\Sn)$ be an antipodally symmetric function, i.e., $K(\xi)=K(-\xi)$ $\forall~\xi\in \Sn$, and be positive somewhere on $\Sn$.
If there exists a maximum point $\xi_0$ of $K$ at which $K$ has flatness order greater than $n-2\sigma$,
then \eqref{main equ} has at least one positive $C^{2}$ solution.
\end{thm}

For $2\leq n<2+2\sigma$, $K\in C^{1,1}(\Sn)$ has flatness order greater than $n-2\sigma$ at every maximum point.
When $\sigma=1$, the above theorem was proved by Escobar and Schoen \cite{ES} for $n\ge 3$.

\begin{thm}
 \label{main thm A} Let $n \geq 2$. Suppose that $K\in C^{1,1}(\Sn)$ is a positive function satisfying that for any critical point
$\xi_0$ of $K$, in some geodesic normal coordinates $\{y_1, \cdots, y_n\}$ centered at $\xi_0$, there exist
some small neighborhood $\mathscr{O}$ of $0$ and positive constants $\beta=\beta(\xi_0)\in (n-2\sigma,n)$, $\gamma\in (n-2\sigma, \beta]$
such that $K\in C^{[\gamma],\gamma-[\gamma]}(\mathscr{O})$ (where $[\gamma]$ is the integer part of $\gamma$) and
\[
 K(y)=K(0)+\sum_{j=1}^{n}a_{j}|y_j|^{\beta}+R(y), \quad \mbox{in } \mathscr{O},
\]
where $a_j=a_j(\xi_0)\neq 0$, $\sum_{j=1}^n a_j\neq 0$,
$R(y)\in C^{[\beta]-1,1}(\mathscr{O})$ satisfies\\ $\sum_{s=0}^{[\beta]}|\nabla^sR(y)||y|^{-\beta+s} \to 0$ as $y\to 0$.
If
\[
\sum_{\xi\in \Sn\mbox{ such that }\nabla_{g_{\Sn}}K(\xi)=0,\  \sum_{j=1}^na_j(\xi)<0}(-1)^{i(\xi)}\neq (-1)^n,
\]
where
\[
 i(\xi)=\#\{a_j(\xi): \nabla_{g_{\Sn}}K(\xi)=0,a_j(\xi)<0,1\leq j\leq n\},
\]
then \eqref{main equ} has at least one $C^2$ positive solution. Moreover, there exists a positive constant $C$ depending only on $n, \sigma$ and $K$ such that for all positive $C^2$ solutions $v$ of \eqref{main equ},
\[
1/C\leq v\leq C\quad\mbox{and}\quad\|v\|_{C^2(\Sn)}\leq C.
\]
\end{thm}

For $n=3, \sigma=1$, the existence part of the above theorem was established by Bahri and Coron \cite{BC}, and the compactness part were given in Chang, Gursky and Yang \cite{CGY} and Schoen and Zhang \cite{SZ}. For $n\geq 4, \sigma=1$, the above theorem was proved by Li \cite{Li95}.

We now consider a class of functions $K$ more general than that in Theorem \ref{main thm A}, which is modified from \cite{Li95}.

\begin{defn}\label{def*} For any real number $\beta> 1$, we say that a sequence of functions
$\{K_i\}$ satisfies condition $(*)'_{\beta}$ for some sequence of constants
$L(\beta,i)$ in some region $\om_i$, if $\{K_i\}\in C^{[\beta],\beta-[\beta]}(\om_i)$ satisfies
\[
[\nabla^{[\beta]} K_i]_{C^{\beta-[\beta]}(\om_i)}\leq L(\beta,i),
\]
and, if $\beta\geq 2$, that
\[
|\nabla^sK_i(y)|\leq L(\beta,i)|\nabla K_i(y)|^{(\beta-s)/(\beta-1)},
\]
for all $2\leq s\leq [\beta]$, $y\in \om_i$, $\nabla K_i(y)\neq 0$.
\end{defn}

Note that the function $K$ in Theorem \ref{main thm A} satisfies $(*)_{\beta}'$ condition.

\begin{rem}
For $1\leq \beta_1\leq \beta_2$,  if $\{K_i\}$ satisfies $(*)'_{\beta_2}$ for some sequences of constants $\{L(\beta_2,i)\}$ in some regions $\om_i$,
then $\{K_i\}$ satisfies $(*)'_{\beta_1}$ for $\{L(\beta_1,i)\}$, where
\[
L(\beta_1,i)=
\begin{cases}
L(\beta_2,i)\max\left(\max\limits_{2\leq s\leq [\beta_1]}\|\nabla K_i\|^{\frac{\beta_2-s}{\beta_2-1}-\frac{\beta_1-s}{\beta_1-1}}_{L^\infty(\om_i)},\mathrm{diam}(\Omega_i)^{\beta_2-\beta_1}\right), \mbox{ if }[\beta_2]=[\beta_1]\\
L(\beta_2,i)\max\left(\max\limits_{2\leq s\leq [\beta_1]}\|\nabla K_i\|^{\frac{\beta_2-s}{\beta_2-1}-\frac{\beta_1-s}{\beta_1-1}}_{L^\infty(\om_i)},\|\nabla K_i\|^{\frac{\beta_2-[\beta_1]-1}{\beta_2-1}}_{L^\infty(\om_i)}\mathrm{diam}(\Omega_i)^{1+[\beta_1]-\beta_1}\right), \\
\hspace{9cm}\mbox{if }[\beta_2]>[\beta_1]
\end{cases}
\] in the corresponding regions.
\end{rem}

The following theorem gives a priori bounds of solutions in $L^{\frac{2n}{n-2\sigma}}$ norm.

\begin{thm}
 \label{main thm B} Let $n \geq 2$, and $K\in C^{1,1}(\Sn)$ be a positive function.
If there exists some constant $d>0$ such that $K$ satisfies $(*)'_{(n-2\sigma)}$ for some constant $L>0$ in $\om_{d}:=\{\xi\in \Sn:|\nabla_{g_0}K(\xi)|<d\}$, then for any positive solution $v\in C^2(\Sn)$
of \eqref{main equ},
\be \label{finite energy}
 \|v\|_{L^{\frac{2n}{n-2\sigma}}(\Sn)}\leq C,
\ee
where $C$ depends only on $n,\sigma, \inf_{\Sn}K>0, \|K\|_{C^{1,1}(\Sn)}, L$, and $d$.
\end{thm}

The above theorem was proved by Schoen and Zhang \cite{SZ} for $n=3$ and $\sigma=1$, and by Li \cite{Li95} for $n\ge 4$ and $\sigma=1$.

Denote $H^\sigma(\Sn)$ by the closure of $C^\infty(\Sn)$ under the norm
\[
 \int_{\Sn} vP_\sigma(v)\,\ud vol_{g_0}.
\]
The estimate \eqref{finite energy} for the solution $v$ is equivalent to
\[
 \|v\|_{H^\sigma(\Sn)}\leq C.
\]

However, the estimate \eqref{finite energy} is not sufficient to imply $L^\infty$ bound for $v$ on $\Sn$. For instance,
\[
 \int_{\Sn}v_{\xi_0,\lda}^{\frac{2n}{n-2\sigma}}(\xi)\,\ud vol_{g_0}=\int_{\Sn}\,\ud vol_{g_0},
\]
but $v_{\xi_0,\lda}(\xi_0)=\lda^{\frac{n-2\sigma}{2}} \to \infty$ as $\lda \to \infty$.
Furthermore, a sequence of solutions $v_i$ may blow up at more than one point, and it is the case when $\sigma=1$ (see \cite{Li96}).
The following theorem shows that the latter situation does not happen when $K$ satisfies a little stronger condition.

\begin{thm}
 \label{main thm C}  Let $n \geq 2$. Suppose that $\{K_i\}\in C^{1,1}(\Sn)$ is a sequence of positive functions with uniform $C^{1,1}$ norm and $1/A_1\leq K_i\leq A_1$ on $\Sn$ for some $A_1>0$ independent of $i$. Suppose also that $\{K_i\}$ satisfying
$(*)'_{\beta}$ condition for some constants $\beta>n-2\sigma$,  $L, d>0$ in $\om_{d}$.
Let $\{v_i\}\in C^2(\Sn)$ be a sequence of corresponding positive solutions of \eqref{main equ} and $v_i(\xi_i)=\max_{\Sn}v_i$ for some $\xi_i$.
Then, after passing to a subsequence, $\{v_i\}$ is either bounded
in $L^\infty(\Sn)$ or blows up at exactly one point in the strong sense: There exists a sequence of M\"obius diffeomorphisms $\{\varphi_i\}$ from
$\Sn$ to $\Sn$ satisfying $\varphi_i(\xi_i)=\xi_i$ and $|\det \ud \varphi_i(\xi_i)|^{\frac{n-2\sigma}{2n}}=v_i^{-1}(\xi_i)$ such that
\[
 \|T_{\varphi_i}v_i -1\|_{C^0(\Sn)}\to 0,\quad \text{as }i \to \infty,
\] where $T_{\varphi_i}v_i:=(v\circ\varphi_i)|\det \ud \varphi_i|^{\frac{n-2\sigma}{2n}}$.
\end{thm}

For $n=3, \sigma=1$, the above theorem was established by Chang, Gursky and Yang in \cite{CGY} and by Schoen and Zhang in \cite{SZ}. For $n\ge 4, \sigma=1$, the above theorem was proved by Li in \cite{Li95}.

M\"obius diffeomorphisms $\varphi$ from $\Sn$ to $\Sn$ are those given by $\varphi=\phi\circ F$ where $\phi$ is a M\"obius transformation from $\R^n\cup\{\infty\}$ to $\R^n\cup\{\infty\}$  generated by translations, multiplications by nonzero constant and the inversion $x\to x/|x|^2$.

Our local analysis of solutions of \eqref{main equ} relies on a localization method introduced by Caffarelli and Silvestre in \cite{CS2} for the factional Laplacian $(-\Delta)^{\sigma}$ on the Euclidean space $\R^n$, through which \eqref{main equ} is connected to a degenerate elliptic differential equation in one dimension higher (see section \ref{LE}).

The proofs of Theorem \ref{main thm B} and Theorem \ref{main thm C} make use of blow up analysis of solutions of \eqref{main equ}, which is an adaptation of the analysis for $\sigma=1$ developed in \cite{SZ} and \cite{Li95}. Our blow up analysis requires a Liouville type theorem. For the definitions of weak solutions and the space $H_{loc}(t^{1-2\sigma},\overline {\R^{n+1}_+})$ in the following Liouville type theorem we refer to Definition \ref{weak solution} and the beginning of section \ref{section of liouville}.

\begin{thm}\label{thm3.1} Let  $U\in H_{loc}(t^{1-2\sigma},\overline {\R^{n+1}_+})$, $U(X)\geq 0$ in $\R^{n+1}_+$ and $U\not\equiv 0$, be a weak solution of
\begin{equation}\label{eq}
\begin{cases}
\mathrm{div}(t^{1-2\sigma}\nabla U(x,t))&=0, \quad \mbox{in }\R^{n+1}_+,\\
-\lim\limits_{t\to 0}t^{1-2\sigma}\partial_t U(x,t)&=U^{\frac{n+2\sigma}{n-2\sigma}}(x,0), \quad x\in\R^n.
\end{cases}
\end{equation}
Then $U(x,0)$ takes the form
\[
\big(N_{\sigma}c_{n,\sigma}2^{2\sigma}\big)^{\frac{n-2\sigma}{4\sigma}}\left(\frac{\lda}{1+\lda^2|x-x_0|^2}\right)^{\frac{n-2\sigma}{2}}
\]
where $\lda>0$, $x_0\in \R^n$, $c_{n,\sigma}$ is the constant in \eqref{main equ} and $N_{\sigma}$ is the constant in \eqref{norms in two spaces are same}.
Moreover,
\[
 U(x,t)=\int_{\mathbb{R}^n} \mathcal{P}_\sigma(x-y,t)U(y,0)\,\ud y
\]for $(x,t)\in \R^{n+1}_+$, where $\mathcal{P}_\sigma(x)$ is the kernel given in \eqref{poisson involution}.
\end{thm}

\begin{rem}\label{subcritical nonexistence}
If we replace $U^{\frac{n+2\sigma}{n-2\sigma}}(x,0)$ by $U^{p}(x,0)$ for $0\leq p<\frac{n+2\sigma}{n-2\sigma}$ in \eqref{eq}, then the only nonnegative solution of \eqref{eq} is $U\equiv 0$. Moreover, for $p<0$, \eqref{eq} has no positive solution. These can be seen from the proof of Theorem \ref{thm3.1} with a standard modification (see, e.g., the proof of Theorem 1.2 in \cite{cl}). For $\sigma\in (1/2, 1)$ and $1< p<\frac{n+2\sigma}{n-2\sigma}$, this result has been proved in \cite{dS}.
\end{rem}

\begin{rem}
We do not make any assumption on the behavior of $U$ near $\infty$. If we assume that $U\in H(t^{1-2\sigma},\mathbb{R}^{n+1}_+)$, the theorem in the case of $p=\frac{n+2\sigma}{n-2\sigma}$ follows from \cite{CLO} and \cite{Li04}. When $\sigma=\frac{1}{2}$, the above theorem can be found in \cite{Hu}, \cite{HY}, \cite{lz1}, \cite{Ou} and \cite{lz}.
\end{rem}

Given the pages needed to present the proofs of all the results, we leave the proofs of Theorem \ref{K-M-E-S} and the existence part of Theorem \ref{main thm A} to the subsequent paper \cite{JLX2}. The needed ingredients for a proof of the existence part of Theorem \ref{main thm A} are all developed in this paper. With these ingredients, the existence part of Theorem \ref{main thm A} follows from a perturbation result and a degree argument which are given in \cite{JLX2}.

The present paper is organized as the following. In section \ref{LE} we derive some properties for solutions of fractional Laplacian equations. In particular we prove that local Schauder estimates hold for positive solutions. In section \ref{section of liouville}, using the method of moving spheres, we establish Theorem \ref{thm3.1}. This Liouville type theorem and the local Schauder estimates are used in the blow up analysis of solutions of \eqref{main equ}. In section \ref{Local analysis near isolated blow up points} we establish accurate blow up profiles of solutions of \eqref{main equ} near isolated blow up points. In fact most of the estimates hold also for subcritical approximations to such equations as well including in bounded domains of $\R^n$. In section \ref{Estimates on the sphere and proof of main theorems}, we provide $H^{\sigma}(\Sn)$ norm a priori estimates, at most one isolated simple blow up point, and $L^{\infty}(\Sn)$ norm a priori estimates for solutions of \eqref{main equ} under appropriate hypotheses on $K$. The proofs of Theorem \ref{main thm A}, \ref{main thm B} and \ref{main thm C} are given in this section. In the Appendix we provide a Kazdan-Warner identity, Lemma \ref{thm3.2} that is in the same spirit of the classical B\^ocher theorem, two lemmas on maximum principles and some complementarities.

\section{Preliminaries} 
\label{LE}

\subsection{A weighted Sobolev space}\label{weighted sobolev spaces}
Let $\sigma\in (0,1)$, $X=(x,t)\in \R^{n+1}$ where $x\in \R^n$ and $t\in \R$. Then $|t|^{1-2\sigma}$ belongs to the Muckenhoupt $A_2$ class in $\R^{n+1}$, namely, there exists a positive constant $C$, such that for any ball $B\subset \R^{n+1}$
\[
\left(\frac{1}{|B|}\int_{B}|t|^{1-2\sigma}\, \ud X \right)\left(\frac{1}{|B|}\int_{B}|t|^{2\sigma-1} \,\ud X \right)\le C.
\]
Let $D$ be an open set in $\R^{n+1}$.
Denote $L^2(|t|^{1-2\sigma}, D)$ as the Banach space of all measurable functions $U$, defined on $D$, for which
\[
 \| U\|_{L^{2}(|t|^{1-2\sigma},D)}:=\left(\int_{D}|t|^{1-2\sigma} U^{2}\,\ud X\right)^{\frac{1}{2}}<\infty.
\]
We say that $U\in H(|t|^{1-2\sigma},D)$ if $U\in L^2(|t|^{1-2\sigma}, D)$, and its weak derivatives $\nabla U$ exist and belong to $L^2(|t|^{1-2\sigma}, D)$. The norm of $U$ in $H(|t|^{1-2\sigma},D)$ is given by
\[
\|U\|_{H(|t|^{1-2\sigma},D)}:=\left(\int_{D}|t|^{1-2\sigma}U^2(X)\,\ud X+\int_{D}|t|^{1-2\sigma}|\nabla U(X)|^2\,\ud X\right)^{\frac{1}{2}}.
\]
It is clear that $H(|t|^{1-2\sigma},D)$ is a Hilbert space with the inner product
\[
 \langle U,V\rangle:=\int_{D}|t|^{1-2\sigma}(UV+\nabla U\nabla V)\,\ud X.
\]
Note that the set of smooth functions $C^{\infty}(D)$ is dense in $H(|t|^{1-2\sigma},D)$. Moreover, if $D$ is a domain, i.e. a bounded connected open set, with Lipschitz boundary $\pa D$, then there exists a linear, bounded extension operator from $H(|t|^{1-2\sigma},D)$ to $H(|t|^{1-2\sigma},\R^{n+1})$ (see, e.g., \cite{Chua}).

Let $\Omega$ be an open set in $\R^n$.
Recall that $H^{\sigma}(\Omega)$ is the fractional Sobolev space defined as
\[
H^{\sigma}(\Omega):=\left\{u\in L^2(\Omega):\frac{|u(x)-u(y)|}{|x-y|^{\frac{n}{2}+\sigma}}\in L^2(\Omega\times\Omega)\right\}
\]
with the norm
\[
\|u\|_{H^{\sigma}(\Omega)}:=\left(\int_{\Omega}u^2\,\ud x+\int_{\Omega}\int_{\Omega}\frac{|u(x)-u(y)|^2}{|x-y|^{n+2\sigma}}\,\ud x\,\ud y\right)^{1/2}.
\]
The set of smooth functions $C^{\infty} (\Omega)$ is dense in $H^{\sigma}(\Omega)$. If $\Omega$ is a domain with Lipschitz boundary, then there exists a linear, bounded extension operator from $H^{\sigma}(\Omega)$ to $H^{\sigma}(\R^n)$. Note that $H^{\sigma}(\R^n)$ with the norm $\|\cdot\|_{H^{\sigma}(\R^n)}$ is equivalent to the following space
\[
\left\{u\in L^2(\R^n): |\xi|^{\sigma}\mathscr{F}(u)(\xi)\in L^2(\R^n)\right\}
\]
with the norm
\[
\|\cdot\|_{L^2(\R^n)}+\||\xi|^{\sigma}\mathscr F(\cdot)(\xi)\|_{L^2(\R^n)}
\]
where $\mathscr F$ denotes the Fourier transform operator. It is known that (see, e.g., \cite{LM}) there exists $C>0$ depending only on $n$ and $\sigma$ such that for $U\in H(t^{1-2\sigma},\R^{n+1}_+)\cap C(\overline{\R^{n+1}_+})$, $\|U(\cdot,0)\|_{H^{\sigma}(\R^n)}\leq C\|U\|_{H(t^{1-2\sigma},\R^{n+1}_+)}$. Hence by a standard density argument, every $U\in H(t^{1-2\sigma},\R^{n+1}_+)$ has a well-defined trace $u:=U(\cdot, 0)\in H^{\sigma}(\R^n)$.

We define $\dot{H}^\sigma(\R^n)$ as the closure of the set $C_c^{\infty}(\R^n)$ of compact supported smooth functions under the norm
\[
\|u\|_{\dot{H}^\sigma(\R^n)}=\||\xi|^{\sigma}\mathscr F(u)(\xi)\|_{L^2(\R^n)}.
\]
Then there exists a constant $C$ depending only on $n$ and $\sigma$ such that
\be\label{fractional sobolev inequality}
\|u\|_{L^{\frac{2n}{n-2\sigma}}(\R^n)}\leq C \|u\|_{\dot H^{\sigma}(\R^n)}\quad\text{ for all } u\in C_c^{\infty}(\R^n).
\ee
For any $u\in \dot H^{\sigma}(\R^n)$, set
\be\label{poisson involution}
 U(x,t)=\mathcal P_{\sigma}[u]:=\int_{\mathbb{R}^n} \mathcal{P}_\sigma(x-\xi,t)u(\xi)\,\ud \xi,\quad (x,t)\in \R^{n+1}_+:=\R^n\times (0,+\infty),
\ee
where
\[
 \mathcal{P}_\sigma(x,t)=\beta(n,\sigma)\frac{t^{2\sigma}}{(|x|^2+t^2)^{\frac{n+2\sigma}{2}}}
\] with constant $\beta(n,\sigma)$ such that $\int_{\R^n}\mathcal{P}_\sigma(x,1)\,\ud x=1$.
Then $U\in C^{\infty}(\R^{n+1}_+)$, $U\in L^2(t^{1-2\sigma},K)$ for any compact set $K$ in $\overline{\R^{n+1}_+}$,  and $\nabla U\in L^2 (t^{1-2\sigma},\R^{n+1}_+)$. Moreover, $U$ satisfies (see \cite{CS2})
  \be\label{2.1}
 \mathrm{div}(t^{1-2\sigma}\nabla U)=0\quad \mbox{in }\mathbb{R}^{n+1}_+,
\ee
  \be\label{norms in two spaces are same}
 \|\nabla U\|_{L^2(t^{1-2\sigma},\R^{n+1}_+)}=N_{\sigma}\|u\|_{\dot H^{\sigma}(\R^n)},
  \ee
and
\be\label{connection of fractional and extension}
-\dlim_{t\to 0}t^{1-2\sigma}\pa_t U(x,t)=N_\sigma(-\Delta)^{\sigma} u(x), \quad \mbox{in }\R^n
\ee
in distribution sense, where $N_\sigma=2^{1-2\sigma}\Gamma(1-\sigma)/\Gamma(\sigma)$.
We refer $U=\mathcal P_{\sigma}[u]$ in \eqref{poisson involution} to be the \emph{extension} of $u$ for any $u\in \dot H^{\sigma}(\R^n)$.

For a domain $D\subset \R^{n+1}$ with boundary $\pa D$, we denote $\pa' D$ as the interior of $\overline D\cap \pa \R^{n+1}_+$ in $\R^n=\partial\R^{n+1}_+$ and $\pa''D=\pa D \setminus \pa' D$.
\begin{prop}\label{prop of trace weighted}
Let $D=\Omega\times(0,R)\subset \R^n\times \R_+$, $R>0$ and $\pa\Omega$ be Lipschitz.

(i) If $U\in H(t^{1-2\sigma},D)\cap C(D\cup \pa' D)$, then $u:=U(\cdot, 0)\in H^{\sigma}(\Omega)$, and
\[
\|u\|_{H^{\sigma}(\Omega)}\leq C \|U\|_{H(t^{1-2\sigma},D)}
\]
where $C$ is a positive constant depending only on $n,\sigma,R$ and $\Omega$.
Hence every $U\in H(t^{1-2\sigma},D)$ has a well-defined trace $U(\cdot, 0)\in H^{\sigma}(\Omega)$ on $\pa' D$. Furthermore, there exists $C_{n,\sigma}>0$ depending only on $n$ and $\sigma$ such that
\be\label{weighted trace embedding inequality}
\|U(\cdot,0)\|_{L^{\frac{2n}{n-2\sigma}}(\Omega)}\leq C_{n,\sigma} \|\nabla U\|_{L^2(t^{1-2\sigma}, D)}\quad\text{ for all } U\in C_c^{\infty}(D\cup\pa' D).
\ee

(ii) If $u\in H^{\sigma}(\Omega)$, then there exists $U\in H(t^{1-2\sigma},D)$ such that the trace of $U$ on $\Omega$ equals to $u$ and
\[
\|U\|_{H(t^{1-2\sigma},D)}\leq C \|u\|_{H^{\sigma}(\Omega)}
\]
where $C$ is a positive constant depending only on $n,\sigma,R$ and $\Omega$.
\end{prop}

\begin{proof}
The above results are well-known and here we just sketch the proofs. For (i), by the previously mentioned result on the extension operator, there exists $\tilde U\in H(t^{1-2\sigma},\R^{n+1})$ such that $\tilde U=U$ in $D$ and
\[
\|\tilde U\|_{H(t^{1-2\sigma},\R^{n+1})}\leq C \|U\|_{H(t^{1-2\sigma}, D)}.
\]
Hence by the previously mentioned result on the trace from $H(t^{1-2\sigma},\R^{n+1}_+)$ to $H^{\sigma}(\R^n)$, we have
 \[
 \|u\|_{H^{\sigma}(\Omega)}\leq \|\tilde U(\cdot,0)\|_{H^{\sigma}(\R^n)}\leq C \|\tilde U\|_{H(t^{1-2\sigma},\R^{n+1}_+)}\leq C \|U\|_{H(t^{1-2\sigma}, D)}.
 \]
For \eqref{weighted trace embedding inequality}, we extend $U$ to be zero in the outside of $\overline D$ and let $V$ be the extension of $U(\cdot,0)$ as in \eqref{poisson involution}. The inequality \eqref{weighted trace embedding inequality} follows from \eqref{fractional sobolev inequality}, \eqref{norms in two spaces are same} and
\[
\|\nabla V\|_{L^2(t^{1-2\sigma},\R^{n+1}_+)}\leq \|\nabla U\|_{L^2(t^{1-2\sigma},\R^{n+1}_+)}
\]
where Lemma \ref{lemma used in weighted trace embedding inequality} is used in the above inequality.

For (ii), since $\pa\Omega$ is Lipschitz, there exists $\tilde u\in H^{\sigma}(\R^n)$ such that $\tilde u=u$ in $\Omega$ and $\|\tilde u\|_{H^{\sigma}(\R^n)}\le C \|u\|_{H^{\sigma}(\Omega)}$.
Then $U=\mathcal P_{\sigma}[u]$, the extension of $\tilde u$, satisfies (ii).
\end{proof}

\subsection{Weak solutions of degenerate elliptic equations}
Let $D$ be a domain in $\R^{n+1}_+$ with $\pa' D\neq \emptyset$.
Let $a\in L^{\frac{2n}{n+2\sigma}}_{loc}(\pa' D)$ and $b\in L^1_{loc}(\pa' D)$. Consider
\begin{equation}\label{2.5}
\begin{cases}
\mathrm{div}(t^{1-2\sigma}\nabla U(X))= 0\quad &\mbox{in } D \\
-\dlim_{t\rightarrow 0^+}t^{1-2\sigma}\pa_tU(x,t)=a(x)U(x,0)+b(x)\quad  &\mbox{on } \pa' D.\\
\end{cases}
\end{equation}
\begin{defn}\label{weak solution}
We say that $U\in H(t^{1-2\sigma},D)$ is a weak solution (resp. supersolution, subsolution) of \eqref{2.5} in $D$, if for every nonnegative $\Phi\in C^\infty_c(D\cup \pa'D)$
\be\label{weak solution def parts}
\int_{D} t^{1-2\sigma} \nabla U \nabla \Phi =(resp.~\ge, \le) \int_{\pa' D} a U \Phi + b \Phi.
\ee
\end{defn}

We denote
$
Q_R=B_R\times (0,R)
$ where $B_R\subset \R^n$ is the ball with radius $R$ and centered at $0$.

\begin{prop}\label{energy estimates1-1-1}
Suppose that $a(x)\in L^{\frac{n}{2\sigma}}(B_1)$ and $b(x)\in L^{\frac{2n}{n+2\sigma}}(B_1)$. Let $U\in H(t^{1-2\sigma},Q_1)$ be a weak solution of \eqref{2.5} in $Q_1$. There exists $\delta>0$ depending only on $n$ and $\sigma$ such that if $\|a^+\|_{ L^{\frac{n}{2\sigma}}(B_1)}<\delta$, then there exists
a constant $C$ depending only on $n,\sigma$ and $\delta$ such that
\[
\|U\|_{H(t^{1-2\sigma},Q_{1/2})}\leq C(\|U\|_{L^2(t^{1-2\sigma},Q_1)}+\|b\|_{L^{\frac{2n}{n+2\sigma}}(B_1)}).
\]
Consequently, if $a\in L^p(B_1)$ for $p>\frac{n}{2\sigma}$, then $C$ depends only on $n,\sigma, \|a\|_{L^p(B_1)}$.
\end{prop}

\begin{proof}
Let $\eta\in C_c^{\infty}(Q_1\cup\pa' Q_1)$ be a cut-off function which equals to $1$ in $Q_{1/2}$ and supported in $Q_{3/4}$. By a density argument, we can choose $\eta^2 U$ as a test function in \eqref{weak solution def parts}. Then we have, by Cauchy-Schwarz inequality,
\[
\int_{Q_{1}}t^{1-2\sigma}\eta^2|\nabla U|^2\,\ud X\leq 4 \int_{Q_1}t^{1-2\sigma}|\nabla\eta|^2U^2\,\ud X + 2 \int_{\pa' Q_1} a^+ (\eta U)^2+b\eta^2 U\,\ud x.
\]
By H\"older inequality and Proposition \ref{prop of trace weighted},
\[
\begin{split}
\int_{\pa' Q_1} a^+ (\eta U)^2 \,\ud x &\leq \delta\|\eta U\|^2_{L^{\frac{2n}{n-2\sigma}}(\pa' Q_1)} \leq \delta C(n,\sigma)\|\nabla(\eta U)\|^2_{L^2(t^{1-2\sigma},Q_1)}\\
\end{split}
\]
By Young's inequality $\forall ~\va>0$,
\[
\begin{split}
\int_{\pa' Q_1} b\eta^2 U(\cdot,0 )\,\ud x &\leq \va\|\eta U\|^2_{L^{\frac{2n}{n-2\sigma}}(\pa' Q_1)}+C(\va) \|b\|^2_{L^{\frac{2n}{n+2\sigma}}(\pa' Q_1)}\\
& \leq \va C(n,\sigma)\|\nabla(\eta U)\|^2_{L^2(t^{1-2\sigma},Q_1)}+C(\va) \|b\|^2_{L^{\frac{2n}{n+2\sigma}}(\pa' Q_1)}.
\end{split}
\]
The first conclusion follows immediately if $\delta$ is sufficient small.

If $a\in L^p(B_1)$, we can choose $r$ small such that $\|a\|_{L^{\frac{n}{2\sigma}}(B_r(x_0))}<\delta$ for any ball $B_r(x_0)\subset B_1$. Then $\hat U(x,t)=r^{\frac{n-2\sigma}{2}}U(rx+x_0,rt)$ satisfies \eqref{2.5} with $\hat a(x)=r^{2\sigma}a(rx+x_0)$ and $\hat b(x,t)=r^{\frac{n+2\sigma}{2}}b(rx+x_0)$ in $Q_1$. Since $\|\hat a\|_{L^{\frac{n}{2\sigma}}(B_1)}<\delta$, applying the above result to $\hat U$, we have
\[
\|U\|_{H(t^{1-2\sigma},B_{1/2}\times (0,r/2))}\leq C(\|U\|_{L^2(t^{1-2\sigma},Q_1)}+\|b\|_{L^{\frac{2n}{n+2\sigma}}(B_1)})
\]
where $C$ depends only on $n,\sigma, \|a\|_{L^{\infty}(B_1)}$. This, together with the fact that \eqref{2.5} is uniformly elliptic in $B_1\times (r/4, 1)$, finishes the proof.
\end{proof}

\begin{prop}\label{prop of existence}
Suppose that $a(x)\in L^{\frac{n}{2\sigma}}(B_1)$. There exists $\delta>0$ which depends only on $n$ and $\sigma$ such that if $\|a^+\|_{L^{\frac{n}{2\sigma}}(B_1)}< \delta$, then for any $b(x)\in L^{\frac{2n}{n+2\sigma}}(B_1)$, there exists a unique solution in $H(t^{1-2\sigma}, Q_1)$ to \eqref{2.5} with $U|_{\pa'' Q_1}=0$.
\end{prop}
\begin{proof}
We consider the bilinear form
\[
B[U, V]:=\int_{Q_1} t^{1-2\sigma} \nabla U\nabla V\,\ud X-\int_{\pa' Q_1} a U V \,\ud x,\quad U, V\in \mathcal A
\]
where $\mathcal A:=\{U\in H(t^{1-2\sigma}, Q_1): U|_{\pa'' Q_1}=0 \text{ in trace sense} \}$.
By Proposition \ref{prop of trace weighted}, it is easy to verify that $B[\cdot,\cdot]$ is bounded and coercive provided $\delta$ is sufficiently small. Therefore the proposition follows from the Riesz representation theorem.
\end{proof}

\begin{lem}\label{wmp}
Suppose $U\in H(t^{1-2\sigma}, D)$ is a weak supersolution of \eqref{2.5} in $D$ with $a\equiv b\equiv 0$. If $U(X)\geq 0$ on $\pa'' D$ in trace sense,
then $U\geq 0$ in $D$.
\end{lem}
\begin{proof}
Use $U^-$ as a test function to conclude that $U^-\equiv 0$.
\end{proof}

The following result is a refined version of that in
\cite{TX}. Such De Giorgi-Nash-Moser type theorems for degenerated equations with Dirichlet boundary conditions have been established in \cite{FKS}.

\begin{prop}\label{lem harnack}
Suppose $a, b\in L^p(B_1)$ for some $p>\frac{n}{2\sigma}$.

(i) Let $U\in H(t^{1-2\sigma}, Q_1)$ be a weak subsolution of
\eqref{2.5} in $Q_1$. Then $\forall ~\nu>0$
\[
\sup_{Q_{1/2}} U^+\le C(\|U^+\|_{L^{\nu}(t^{1-2\sigma}, Q_1)}+\|b^+\|_{L^p(B_1)})
\]
where $U^+=\max(0,U)$, and $C>0$ depends only on $n,\sigma, p,\nu$ and $\|a^+\|_{L^p(B_1)}$.

 (ii) Let $U\in H(t^{1-2\sigma}, Q_1)$ be a nonnegative weak supersolution of
\eqref{2.5} in $Q_1$. Then for any $0<\mu<\tau<1$, $0<\nu\le\frac{n+1}{n}$ we have
\[
\inf_{Q_{\mu}} U +\|b^-\|_{L^p(B_1)}\ge C\|U\|_{L^{\nu}(t^{1-2\sigma}, Q_{\tau})}
\]
where $C>0$ depends only on $n,\sigma, p,\nu,\mu,\tau$ and $\|a^-\|_{L^p(B_1)}$.

(iii) Let $U\in H(t^{1-2\sigma}, Q_1)$ be a nonnegative weak solution of
\eqref{2.5} in $Q_1$. Then we have the following Harnack inequality
\be\label{Har ineq}
\sup_{Q_{1/2}} U\leq C (\inf_{Q_{1/2}} U+\|b\|_{L^p(B_1)}),
\ee
where $C>0$ depends only on $n,\sigma,p,\|a\|_{L^p(B_1)}$. Consequently,
there exists $\al\in (0,1)$ depending only on  $n,\sigma,p,\|a\|_{L^p(B_1)}$ such that any weak solution $U(X)$ of \eqref{2.5} is of $C^{\al}(\overline{Q_{1/2}})$.  Moreover,
\[
\|U\|_{C^{\al}(\overline{Q_{1/2}})}\leq C(\|U\|_{L^\infty(Q_1)}+\|b\|_{L^p(B_1)})
\]
where $C>0$ depends only on $n,\sigma,p,\|a\|_{L^p(B_1)}$.
\end{prop}

\begin{proof}
The proofs are modifications of those in \cite{TX}, where the method of Moser iteration is used. Here we only point out the changes.
Let $k=\|b^+\|_{L^p(B_1)}$ if $b^+\not\equiv 0$, otherwise let $k>0$ be any number which is eventually sent to $0$. Define $\overline U=U^++k$ and, for $m>0$, let
\[
  \overline U_m=
\begin{cases}
\overline{U}&\quad \mbox{if } U<m,\\
k+m &\quad \mbox{if } U\geq m.
\end{cases}
\]
Consider the test function
\[
 \phi=\eta^2(\overline U_m^\beta\overline U-k^{\beta+1})\in H(t^{1-2\sigma}, Q_1),
\]
for some $\beta\geq 0$ and some nonnegative function $\eta\in C^1_c(Q_1\cup \pa' Q_1)$. Direction calculations yield that, with setting $W=\overline U_m^{\frac{\beta}{2}}\overline U$,
\be\label{moser's key integration}
\frac{1}{1+\beta}\int_{Q_1}t^{1-2\sigma}|\nabla(\eta W)|^2\leq 16\int_{Q_1}t^{1-2\sigma}|\nabla \eta|^2 W^2 +4\int_{\pa' Q_1}(a^++\frac{b^+}{k})\eta^2W^2.
\ee
By H\"older's inequality and the choice of $k$, we have
\[
\int_{\pa' Q_1}(a^++\frac{b^+}{k})\eta^2W^2\leq (\|a^+\|_{L^p(B_1)}+1)\|\eta^2W^2\|_{L^{p'}(B_1)}
\]
where ${p'}=\frac{p}{p-1}<\frac{n}{n-2\sigma}$. Choose $0<\theta<1$ such that $\frac{1}{p'}=\theta+\frac{(1-\theta)(n-2\sigma)}{n}$. The interpolation inequality gives that, for any $\va>0$,
\[
\|\eta^2W^2\|_{L^{p'}(B_1)}\leq \va \|\eta W\|^2_{L^{\frac{2n}{n-2\sigma}}(B_1)}+ \va^{-\frac{1-\theta}{\theta}}\|\eta^2W^2\|_{L^1(B_1)}.
\]
By the trace embedding inequality in Proposition \ref{prop of trace weighted}, there exists $C>0$ depending only on $n,\sigma$ such that
\[
\|\eta W\|^2_{L^{\frac{2n}{n-2\sigma}}(B_1)}\le C\int_{Q_1}t^{1-2\sigma}|\nabla(\eta W)|^2.
\]
By Lemma 2.3 in \cite{TX}, there exist $\delta>0$ and $C>0$ both of which depend only on $n,\sigma$ such that
\[
\|\eta^2W^2\|_{L^1(B_1)}\le \va^{\frac{1}{\theta}}\int_{Q_1}t^{1-2\sigma}|\nabla(\eta W)|^2+\va^{-\frac{\delta}{\theta}}\int_{Q_1}t^{1-2\sigma}\eta^2W^2.
\]
By choosing $\va$ small, the above inequalities give that
\[
\int_{Q_1}t^{1-2\sigma}|\nabla(\eta W)|^2\le C(1+\beta)^{\delta/\theta}\int_{Q_1}t^{1-2\sigma}(\eta^2+|\nabla\eta|^2)W^2
\]
where $C$ depends only on $n,\sigma$ and $\|a^+\|_{L^p(B_1)}$. Then the proof of Proposition 3.1 in \cite{TX} goes through without any change. This finishes the proof of (i) for $\nu=2$. Then (i) also holds for any $\nu>0$ which follows from standard arguments. For part (ii)  we choose $k=\|b^-\|_{L^p(B_1)}$ if $b^-\not\equiv 0$, otherwise let $k>0$ be any number which is eventually sent to $0$.  Then we can show that there exists some $\nu_0>0$ for which (ii) holds, by exactly the same proof of Proposition 3.2 in \cite{TX}. Finally use the test function $\phi=\overline U^{-\beta}\eta^2$ with $\beta\in (0,1)$ to repeat the proof in (i) to conclude (ii) for $0<\nu\le\frac{n+1}{n}$. Part (iii) follows from (i), (ii) and standard elliptic equation theory.
\end{proof}

\begin{rem}
Harnack inequality \eqref{Har ineq}, without lower order term $b$, has been obtained earlier in \cite{CaS} using a different method.
\end{rem}

The above proofs can be improved to yield the following result.
\begin{lem}\label{b-k}
Suppose
$a\in L^{\frac{n}{2\sigma}}(B_1), b\in L^p(B_1)$ with $p>\frac{n}{2\sigma}$ and $U\in H(t^{1-2\sigma}, Q_1)$ is a weak subsolution of \eqref{2.5} in $Q_1$. There exists $\delta>0$ which depends only on $n$ and $\sigma$ such that if $\|a^+\|_{L^{\frac{n}{2\sigma}}(B_1)}<\delta$, then
\[
\|U^+(\cdot, 0)\|_{L^q(\partial'Q_{1/2})}\leq C(\|U^+\|_{H(t^{1-2\sigma},Q_1)}+\|b^+\|_{L^p(B_1)}).
\]
where  $C>0$ depends only on $n,p,\sigma,\delta$, and $q=\min\left(\frac{2(n+1)}{n-2\sigma},\ \frac{n(p-1)}{(n-2\sigma)p}\cdot \frac{2n}{n-2\sigma}\right)$.
\end{lem}

\begin{rem}
Analogues estimates were established for $-\Delta u=a(x)u$ in \cite{BK} (see Theorem 2.3 there) and for $-\mathrm{div}(|\nabla u|^{p-2}\nabla u)=a(x)|u|^{p-2}u$ in \cite{al} (see Lemma 3.1 there).
\end{rem}

\begin{proof}[Proof of Lemma \ref{b-k}]
We start from \eqref{moser's key integration}, where we choose $\beta=\min\left(\frac{2}{n}, \frac{2(2\sigma p-n)}{(n-2\sigma)p}\right)$. By H\"older inequality and Proposition \ref{prop of trace weighted},
\[
\begin{split}
\int_{\pa' Q_1}(a^++\frac{b^+}{k})\eta^2W^2 & \leq \delta \|\eta^2W^2\|_{L^{\frac{n}{n-2\sigma}}(B_1)} +\|\eta^2W^2\|_{L^{p'}(B_1)}\\
&\leq C(n,\sigma)\delta \int_{Q_1}t^{1-2\sigma}|\nabla(\eta W)|^2+C_{n,\sigma,p}\|\overline U\|_{H(t^{1-2\sigma},Q_1)}.
\end{split}
\]
By Poincare's inequality in \cite{FKS}, we have
\[
\int_{Q_1}t^{1-2\sigma}|\nabla \eta|^2 W^2\leq C_{n,\sigma,p}\|\overline U\|_{H(t^{1-2\sigma},Q_1)}.
\]
If $\delta$ is sufficiently small, the the above together with \eqref{moser's key integration} imply that
\[
\int_{Q_1}t^{1-2\sigma}|\nabla(\eta W)|^2\leq C_{n,\sigma,p}\|\overline U\|_{H(t^{1-2\sigma},Q_1)}.
\]
Hence it follows from H\"older inequality and Proposition \ref{prop of trace weighted} that, by sending $m\to\infty$,
\[
\|\overline U(\cdot, 0)\|_{L^q(\partial'Q_{1/2})}\leq C_{n,\sigma,p}\int_{Q_1}t^{1-2\sigma}|\nabla(\eta W)|^2\leq C_{n,\sigma,p}\|\overline U\|_{H(t^{1-2\sigma},Q_1)}.
\]
This finishes the proof.
\end{proof}

\begin{cor}\label{thm2.3}
Suppose that $K\in L^\infty(B_1)$,  $U\in H(t^{1-2\sigma},Q_1)$ and $U\geq 0$ in $Q_1$ satisfies, for some $1\leq p\leq (n+2\sigma)(n-2\sigma)$,
\[
\begin{cases}
\mathrm{div}(t^{1-2\sigma}\nabla U(X))= 0\quad &\mbox{in } Q_1 \\
-\dlim_{t\rightarrow 0^+}t^{1-2\sigma}\pa_tU(x,t)=K(x)U(x,0)^p\quad  &\mbox{on } \pa' Q_1.\\
\end{cases}
\]
Then
(i) $U \in L^{\infty}_{loc}(Q_1\cup\pa' Q_1)$, and hence $U(\cdot, 0) \in L^{\infty}_{loc}(B_1)$.

(ii) There exist $C>0$ and $\al\in(0,1)$ depending only on $n,\sigma,  p,\|u\|_{L^\infty(B_{3/4})}$, $\|K\|_{L^\infty(B_{3/4})}$ such that $U\in C^\al(\overline {Q_{1/2}})$
and
\[
\|U\|_{H(t^{1-2\sigma},Q_{1/2})}+\|U\|_{C^{\al}(\overline {Q_{1/2}})}\leq C.
\]
\end{cor}

Note that the regularity of solution of $-\Delta u=u^{\frac{n+2}{n-2}}$ was proved by Trudinger in \cite{Tu}.

\begin{proof}[Proof of Corollary \ref{thm2.3}]
By Proposition \ref{prop of trace weighted}, $U(\cdot,0)\in H^{\sigma}(B_1)\subset L^{\frac{2n}{n-2\sigma}}(B_1)$. Thus $U(\cdot,0)^{p-1}\in L^{\frac{n}{2\sigma}}(B_1)$. Then part (i) follows from Lemma \ref{b-k} and Proposition \ref{lem harnack}. Part (ii) follows from  Proposition \ref{energy estimates1-1-1} and Proposition \ref{lem harnack}.
\end{proof}

\subsection{Local Schauder estimates}

Let $\Omega$ be a domain in $\R^n$, $a\in L^{\frac{2n}{n+2\sigma}}_{loc}(\Omega)$ and $b\in L^1_{loc}(\Omega)$. We say $u\in \dot H^{\sigma}(\R^n)$ is a weak solution of
\[
(-\Delta)^{\sigma}u=a(x)u+b(x)\quad \mbox{in } \Omega
\]
if for any $\phi\in C^{\infty}(\R^n)$ supported in $\Omega$,
\[
\int_{\R^n}(-\Delta)^{\frac{\sigma}{2}}u(-\Delta)^{\frac{\sigma}{2}}\phi=\int_{\Omega}a(x)u\phi+b(x)\phi.
\]
Then by \eqref{connection of fractional and extension}, $u\in \dot H^{\sigma}(\R^n)$ is a weak solution of
\[
(-\Delta)^{\sigma}u=\frac{1}{N_{\sigma}}\Big(a(x)u+b(x)\Big)\quad \mbox{in } B_1
\]
if and only if $U=\mathcal P_{\sigma}[u]$, the extension of $u$ defined in \eqref{poisson involution}, is a weak solution of \eqref{2.5} in $Q_1$.

For $\al\in (0,1)$,  $C^\al(\om)$ denotes the standard H\"older space over domain $\om$.
For simplicity, we use $C^{\al}(\om)$ to denote $C^{[\al],\al-[\al]}(\om)$ when $1<\al\notin \mathbb{N}$ (the set of positive integers).

In this part, we shall prove the following local Schauder estimates for
nonnegative solutions of fractional Laplace equation.

\begin{thm}\label{thm2.2} Suppose $a(x), b(x)\in C^{\al}(B_1)$ with $0<\al\not\in \mathbb{N}$.
Let $u\in \D$ and $u\geq 0$ in $\R^n$ be a weak solution of
\[
(-\Delta)^{\sigma}u=a(x)u+b(x), \quad \mbox{in } B_1.
\]
Suppose that $2\sigma+\alpha$ is \emph{not} an integer. Then $u\in C^{2\sigma+\al}(B_{1/2})$. Moreover,
\be\label{schauder estimate}
\|u\|_{C^{2\sigma+\al}(B_{1/2})}\leq C(\dinf_{B_{3/4}}u+\|b\|_{C^{\al}(B_{3/4})})
\ee
where $C>0$ depends only on $n,\sigma,\al,\|a\|_{C^{\al}(B_{3/4})}$.
\end{thm}

\begin{rem}
Replacing the assumption $u\geq 0$ in $\R^n$ by  $u\geq 0$ in $B_1$, estimate \eqref{schauder estimate} may fail (see \cite{Ka}).
Without the sign assumption of $u$, \eqref{schauder estimate} with $\inf_{B_{3/4}}u$ substituted by $\|u\|_{L^\infty(\R^n)}$ holds,
which is proved in \cite{CS}, \cite{CS10} and \cite{CSpre} in a much more general setting of fully nonlinear nonlocal equations.
\end{rem}

The following proposition will be used in the proof of Theorem \ref{thm2.2}.
\begin{prop}\label{lem2.2} Let $a(x), b(x)\in C^{k}(B_1), \ U(X)\in H(t^{1-2\sigma}, Q_1)$ be a weak solution of \eqref{2.5} in $Q_1$,
where $k$ is a positive integer. Then we have
\[
\dsum_{i=0}^{k}\|\nabla_x^iU\|_{L^\infty(Q_{1/2})}\leq C(\|U\|_{L^2(t^{1-2\sigma},Q_{1})}+\|b\|_{C^{k}(B_1)}),
\]
where $C>0$ depends only on $n,\sigma, k, \|a\|_{C^k(B_1)}$.
\end{prop}

\begin{proof}
We know from Proposition \ref{lem harnack} that $U$ is H\"older continuous in $\overline{Q_{8/9}}$. Let $h\in \R^n$ with $|h|$ sufficiently small. Denote $U^h(x,t)=\frac{U(x+h,t)-U(x,t)}{|h|}$. Then $U^h$ is a weak solution of
\begin{equation}\label{eq:difference quotient1-1}
\begin{cases}
\mathrm{div}(t^{1-2\sigma}\nabla U^h(X))= 0\quad &\mbox{in } Q_{8/9} \\
-\dlim_{t\rightarrow 0^+}t^{1-2\sigma}\pa_tU^h(x,t)=a(x+h) U^h+a^h U+ b^h\quad  &\mbox{on } \pa' Q_{8/9}.\\
\end{cases}
\end{equation}
By Proposition \ref{energy estimates1-1-1} and Proposition \ref{lem harnack},
\[
\begin{split}
\|U^h\|_{H(t^{1-2\sigma}, Q_{2/3})}+\|U^h\|_{C^{\al}(\overline{Q_{2/3}})}
&\leq C(\|U^h\|_{L^2(t^{1-2\sigma},Q_{3/4})}+\|b\|_{C^{1}(B_1)})\\
&\leq C (\|\nabla U\|_{L^2(t^{1-2\sigma},Q_{4/5})}+\|b\|_{C^{1}(B_1)})\\
&\leq C (\|U\|_{L^2(t^{1-2\sigma},Q_{1})}+\|b\|_{C^{1}(B_1)})\\
\end{split}
\]
for some $\al\in (0,1)$ and positive constant $C>0$ depending only on  $n,\sigma, \|a\|_{C^1(B_1)}$. Hence $\nabla_x U\in H(t^{1-2\sigma}, Q_{2/3})\cap C^{\al}(\overline {Q_{2/3}})$, and it is a weak solution of
\begin{equation*}
\begin{cases}
\mathrm{div}(t^{1-2\sigma}\nabla (\nabla_x U)= 0\quad &\mbox{in } Q_{2/3} \\
-\dlim_{t\rightarrow 0^+}t^{1-2\sigma}\pa_t (\nabla_x U)=a \nabla_x U+U \nabla_x a + \nabla_x b\quad  &\mbox{on } \pa' Q_{2/3}.\\
\end{cases}
\end{equation*}
Then this Proposition follows immediately from Proposition \ref{energy estimates1-1-1} and Proposition \ref{lem harnack} for $k=1$. We can continue this procedure for $k=2,3,\cdots$ (by induction).
\end{proof}

To prove Theorem \ref{thm2.2} we first obtain Schauder estimates for solutions of the equation
\begin{equation}\label{ge 2.5}
\begin{cases}
\mathrm{div}(t^{1-2\sigma}\nabla U(X))= 0\quad &\mbox{in } Q_R \\
-\dlim_{t\rightarrow 0^+}t^{1-2\sigma}\pa_tU(x,t)=g(x)\quad  &\mbox{on } \pa' Q_R.\\
\end{cases}
\end{equation}

\begin{thm}\label{thm2.1-poisson}
Let $U(X)\in H(t^{1-2\sigma},Q_2)$ be a weak solution of (\ref{ge 2.5}) with $R=2$
and $g(x)\in C^\al(B_2)$ for some $0<\al\not\in\mathbb{N}$.
If $2\sigma+\alpha$ is not an integer, then $U(\cdot,0)$ is of $C^{2\sigma+\al}(B_{1/2})$. Moreover, we have
\[
 \|U(\cdot,0)\|_{C^{2\sigma+\al}(B_{1/2})}\leq C(\|U\|_{L^\infty(Q_2)}+\|g\|_{C^\al(B_2)}),
\]
where $C>0$ depends only on $n,\sigma,\al$.
\end{thm}

This theorem together with Proposition \ref{lem harnack} implies the following

\begin{thm}\label{thm2.1}
Let $U(X)\in H(t^{1-2\sigma},Q_1)$ be a weak solution of (\ref{2.5}) with $D=Q_1$
and $a(x),b(x)\in C^\al(B_1)$ for some $0<\al\not\in\mathbb{N}$.
If $2\sigma+\alpha$ is not an integer, then $U(\cdot,0)$ is of $C^{2\sigma+\al}(B_{1/2})$. Moreover, we have
\[
 \|U(\cdot,0)\|_{C^{2\sigma+\al}(B_{1/2})}\leq C(\|U\|_{L^\infty(Q_1)}+\|b\|_{C^\al(B_1)}),
\]
where $C>0$ depends only on $n,\sigma,\al, \|a\|_{C^{\al}(B_1)}$.
\end{thm}
\begin{proof}
From Proposition \ref{lem harnack}, $U$ is H\"older continuous in $\overline{Q_{3/4}}$. Theorem \ref{thm2.1} follows from bootstrap arguments by applying Theorem \ref{thm2.1-poisson} with $g(x):=a(x)U(x,0)+b(x)$.
\end{proof}

\begin{proof}[Proof of Theorem \ref{thm2.1-poisson}]
Our arguments are in the spirit of those in \cite{Ca} and \cite{LN}. Denote $C$ as various constants that depend only on $n$ and $\sigma$.
Let $\rho=\frac{1}{2}$, $Q_k=Q_{\rho^k}(0),\pa'Q_k=B_k, k=0,1,2,\cdots$. (Note that we have abused notations a little bit. Only in this proof we refer $Q_k, B_k$ as $Q_{\rho^k},B_{\rho^k}$.)
We also denote $M=\|g\|_{C^\al(B_2)}$. From Proposition \ref{lem harnack} we have already known that $U$ is H\"older continuous in $\overline{Q_0}$. First we assume that $\al\in (0,1)$

\emph{Step 1:} We consider the case of $2\sigma+\alpha<1$. Let $W_k$ be the unique weak solution of (which is guaranteed by Proposition \ref{prop of existence})
\begin{equation}
\begin{cases}
\mathrm{div}(t^{1-2\sigma}\nabla W_k(X))= 0\quad &\mbox{in } Q_k \\
-\dlim_{t\rightarrow 0^+}t^{1-2\sigma}\pa_tW_k(x,t)=g(0)-g(x)\quad  &\mbox{on } \pa' Q_k\\
W_k(X)=0\quad &\mbox{on }\pa'' Q_k
\end{cases}
\end{equation}
Let $U_k=W_k+U$ in $Q_k$ and $h_{k+1}=U_{k+1}-U_{k}$ in $Q_{k+1}$, then
\be\label{bound of W}
\|W_k\|_{L^{\infty}(Q_k)}\leq C M \rho^{(2\sigma+\al) k}.
\ee
Indeed \eqref{bound of W} follows by applying Lemma \ref{wmp} to the equation of $\rho^{-2\sigma k}W_k(\rho^k x)\pm (t^{2\sigma}-3)M\rho^{\al k}$ in $Q_0$. Hence by weak maximum principle again we have
\[
\|h_{k+1}\|_{L^{\infty}(Q_k)}\leq C M \rho^{(2\sigma+\al) k}.
\]
By Proposition \ref{lem2.2}, we have, for $i=0,1,2,3$
\begin{equation}\label{higher bound of h}
\|\nabla_x^i h_{k+1}\|_{L^{\infty}(Q_{k+2})}\leq CM\rho^{(2\sigma+\al-i)k}.
\end{equation}
Similarly apply Proposition \ref{lem2.2} to $U_0$, we have
\begin{equation}\label{higher bound of u0}
\begin{split}
\|\nabla_x^i U_0\|_{L^{\infty}(Q_{2})}&\leq C(\|U_0\|_{L^{\infty}(Q_1)}+M)\leq  C(\|U\|_{L^{\infty}(Q_0)}+M)\\
\end{split}
\end{equation}
For any given point $z$ near $0$, we have
\[
\begin{split}
&|U(z,0)-U(0,0)|\\
&\leq|U_k(0,0)-U(0,0)|+|U(z,0)-U_k(z,0)| + |U_{k}(z,0)-U_k(0,0)|\\
&=I_1+I_2+I_3
\end{split}
\]
Let $k$ be such that $\rho^{k+4}\leq |z|\leq \rho^{k+3}$. By \eqref{bound of W},
\[
I_1+I_2\leq C M \rho^{(2\sigma+\al) k}\leq C M |z|^{2\sigma+\alpha}.
\]
For $I_3$, by \eqref{higher bound of h} and \eqref{higher bound of u0},
\[
\begin{split}
I_3&\leq |U_0(z,0)-U_0(0,0)|+\sum_{j=1}^k|h_j(z,0)-h_j(0,0)|\\
&\leq C|z|\Big(\|\nabla_x U_0\|_{L^{\infty}(Q_{k+3})}+\sum_{j=1}^k\|\nabla_x h_j\|_{L^{\infty}(Q_{k+3})}\Big)\\
&\leq C|z|\Big(\|U\|_{L^{\infty}(Q_{0})}+M+M\sum_{j=1}^k\rho^{(2\sigma+\al-1)j}\Big)\\
&\leq C|z|\Big(\|U\|_{L^{\infty}(Q_{0})}+M(1+|z|^{2\sigma+\alpha-1})\Big).
\end{split}
\]
Thus, for $2\sigma+\al<1$, we have
\[
|U(z,0)-U(0,0)|\leq C \big(M+\|U\|_{L^{\infty}(Q_{0})}\big) |z|^{2\sigma+\alpha}.
\]
which finishes the proof of Step 1.

\medskip

\emph{Step 2:} For $1<2\sigma+\alpha<2$, the arguments in Step 1 imply that
\be\label{bound of first}
\|\nabla_x U(\cdot,0)\|_{L^{\infty}(B_1)}\leq C\Big(\|U\|_{L^{\infty}(Q_0)}+M\Big).
\ee
Apply \eqref{bound of first} to the equation of $W_k$ we have, together with \eqref{bound of W},
\[
\|\nabla_x W_k(\cdot, 0)\|_{L^{\infty}(B_{k+1})}\leq CM\rho^{(2\sigma+\alpha-1)k}
\]
By \eqref{higher bound of h} and \eqref{higher bound of u0},
\[
\begin{split}
& |\nabla_x U_{k}(z,0)-\nabla_x U_k(0,0)|\\
&\leq |\nabla_x U_{0}(z,0)-\nabla_x U_0(0,0)|+\sum_{j=1}^k|\nabla_x h_{j}(z,0)-\nabla_x h_j(0,0)|\\
&\leq C|z|\Big(\|\nabla_x^2 U_0\|_{L^{\infty}(Q_{k+3})}+\sum_{j=1}^k\|\nabla_x ^2h_j\|_{L^{\infty}(Q_{k+3})}\Big)\\
&\leq C|z|\Big(\|U\|_{L^{\infty}(Q_{0})}+M+M\sum_{j=1}^k\rho^{(2\sigma-2+\al)j}\Big)\\
&\leq C|z|\Big(\|U\|_{L^{\infty}(Q_{0})}+M(1+|z|^{2\sigma+\alpha-2})\Big).
\end{split}
\]
Hence
\[
\begin{split}
&|\nabla_x U(z,0)-\nabla_x U(0,0)|\\
&\leq|\nabla_x W_k(0,0)|+|\nabla_x W_k(z,0)| + |\nabla_x U_{k}(z,0)-\nabla_x U_k(0,0)|\\
&\leq  CM\rho^{(2\sigma+\alpha-1)k} + C|z|\Big(\|U\|_{L^{\infty}(Q_{0})}+M(1+|z|^{2\sigma+\alpha-2})\Big)\\
&\leq C \big(M+\|U\|_{L^{\infty}(Q_{0})}\big) |z|^{2\sigma+\alpha-1}.
\end{split}
\]
which finishes the proof of Step 2.

\medskip

\emph{Step 3: } For $2\sigma+\al>2$, the arguments in Step 2 imply that
\be\label{bound of second}
\|\nabla_x^2 U(\cdot,0)\|_{L^{\infty}(B_1)}\leq C\Big(\|U\|_{L^{\infty}(Q_0)}+M\Big),
\ee
Apply \eqref{bound of second} to the equation of $W_k$ we have, together with \eqref{bound of W},
\[
\|\nabla_x^2 W_k(\cdot, 0)\|_{L^{\infty}(B_{k+1})}\leq CM\rho^{(2\sigma+\alpha-2)k}
\]
By \eqref{higher bound of h} and \eqref{higher bound of u0},
\[
\begin{split}
& |\nabla_x^2 U_{k}(z,0)-\nabla_x^2 U_k(0,0)|\\
&\leq |\nabla_x^2 U_{0}(z,0)-\nabla_x^2 U_0(0,0)|+\sum_{j=1}^k|\nabla_x^2 h_{j}(z,0)-\nabla_x^2 h_j(0,0)|\\
&\leq C|z|\Big(\|\nabla_x^3 U_0\|_{L^{\infty}(Q_{k+3})}+\sum_{j=1}^k\|\nabla_x ^3h_j\|_{L^{\infty}(Q_{k+3})}\Big)\\
&\leq C|z|\Big(\|U\|_{L^{\infty}(Q_{0})}+M+M\sum_{j=1}^k\rho^{(2\sigma+\alpha-3)k}\Big)\\
&\leq C|z|\Big(\|U\|_{L^{\infty}(Q_{0})}+M(1+|z|^{2\sigma+\al-3})\Big).
\end{split}
\]
Hence
\[
\begin{split}
&|\nabla_x^2 U(z,0)-\nabla_x^2 U(0,0)|\\
&\leq|\nabla_x^2 W_k(0,0)|+|\nabla_x^2 W_k(z,0)| + |\nabla_x^2 U_{k}(z,0)-\nabla_x U_k(0,0)|\\
&\leq CM\rho^{(2\sigma+\alpha-2)k} + C|z|\Big(\|U\|_{L^{\infty}(Q_{0})}+M(1+|z|^{2\sigma+\al-3})\Big)\\
&\leq C \big(M+\|U\|_{L^{\infty}(Q_{0})}\big) |z|^{2\sigma+\alpha-2}.
\end{split}
\]
which finishes the proof of Step 3. This finishes the proof of Theorem \ref{thm2.1-poisson} for $\al\in (0,1)$.

For the case that $\al>1$, we may apply $\nabla_x$ to \eqref{ge 2.5} $[\al]$ times, as in the proof of Proposition \ref{lem2.2}, and repeat the above three steps. Theorem \ref{thm2.1-poisson} is proved.
\end{proof}

\begin{proof}[Proof of Theorem \ref{thm2.2}]
Since $u\in \dot{H}^\sigma(\R^n)$ is nonnegative, its extension $U\geq 0$ in $\R^{n+1}_+$ and $U\in H(t^{1-2\sigma},Q_1)$ is a weak solution of \eqref{2.5} in $Q_1$.
The theorem follows immediately from Theorem \ref{thm2.1} and Proposition \ref{lem harnack}.
\end{proof}

\begin{rem}
Another way to show Theorem \ref{thm2.2} is the following. Let $u\in\dot {H}^{\sigma}(\R^n)$ and $u\geq 0$ in $\R^n$ be a solution of
\[
(-\Delta)^{\sigma} u=g(x),\quad\mbox{in }B_1
\]
where $g\in C^{\al}(B_1)$.
 Let $\eta$ be a nonnegative smooth cut-off function supported in $B_{1}$ and equal to $1$ in $B_{7/8}$. Let $v\in \dot H^{\sigma}(\R^n)$ be the solution of
\[
(-\Delta)^{\sigma} v=\eta (x)g(x),\quad\mbox{in }\R^n
\]
where $\eta g$ is considered as a function defined in $\R^n$ and supported in $B_1$, i.e., $v$ is a Riesz potential of $\eta g$
\[
v(x)=\frac{\Gamma(\frac{n-2\sigma}{2})}{2^{2\sigma}\pi^{n/2}\Gamma (\sigma)}\int_{\R^n}\frac{\eta(y)g(y)}{|x-y|^{n-2\sigma}}\ud y.
\]
 Then if $2\sigma+\al$ and $\al$ are not integers, we have (see, e.g., \cite{S})
\[
\|v\|_{C^{2\sigma+\al}(B_{1/2})}\leq C(\|v\|_{L^{\infty}(\R^n)}+\|\eta g\|_{C^{\al}(\R^n)})\leq C\|g\|_{C^{\al}(B_1)}.
\]
Let $w=u-v$ which belongs to $\dot H^{\sigma}(\R^n)$ and satisfies
\[
(-\Delta)^{\sigma} w=0,\quad\mbox{in }B_{7/8}.
\]
Let $W=\mathcal P_{\sigma}[w]$ be the extension of $w$, and $\tilde W=W+\|v\|_{L^{\infty}(\R^n)}\geq 0$ in $\R^{n+1}_+$. Notice that $\tilde W$ is a nonnegative weak solution of \eqref{2.5} with $a\equiv b\equiv 0$ and $D=Q_1$. By Proposition \ref{lem2.2} and Proposition \ref{lem harnack}, we have
\[
\begin{split}
&\|w+\|v\|_{L^{\infty}(\R^n)}\|_{C^{2\sigma+\al}(B_{1/2})}\\
&\leq  C\|\tilde W\|_{L^2(t^{1-2\sigma},Q_{7/8})}\leq C\inf_{Q_{3/4}}\tilde W\leq C(\inf_{Q_{3/4}}u +\|v\|_{L^{\infty}(\R^n)}).
\end{split}
\]
Hence
\[
\begin{split}
\|u\|_{C^{2\sigma+\al}(B_{1/2})}&\leq \|v\|_{C^{2\sigma+\al}(B_{1/2})}+\|w\|_{C^{2\sigma+\al}(B_{1/2})}\\
&\leq C(\inf_{B_{3/4}}u+\|g\|_{C^{\al}(B_1)}).
\end{split}
\]
Using bootstrap arguments as that in the proof of Theorem \ref{thm2.1}, we conclude Theorem \ref{thm2.2}.
\end{rem}

\begin{rem}
Indeed, our proofs also lead to the following. If we only assume that $a(x), b(x), g(x)\in L^{\infty}(B_1)$, and let $U$, $u$ be those in Theorem \ref{thm2.1-poisson} and in Theorem \ref{thm2.2} respectively, then the estimates
\[
 \|U(\cdot,0)\|_{C^{2\sigma}(B_{1/2})}\leq C_1(\|U\|_{L^\infty(Q_1)}+\|g\|_{L^{\infty}(B_1)})
\]
\[
\|u\|_{C^{2\sigma}(B_{1/2})}\leq C_2(\dinf_{B_{3/4}}u+\|b\|_{L^{\infty}(B_{3/4})})
\]
hold provided $\sigma\neq 1/2$ , where $C_1>0$ depends only on $n,\sigma,\al$ and $C_2>0$ depends only on $n,\sigma,\al,\|a\|_{L^{\infty}(B_{3/4})}$.
For $\sigma=\frac{1}{2}$,  we have the following log-Lipschitz property: for any $y_1,y_2\in B_{1/4}, y_1\neq y_2$,
\[
\frac{|U(y_1,0)-U(y_2,0)|}{|y_1-y_2|}\leq C_1(\|U\|_{L^\infty(Q_1)}-\|g\|_{L^{\infty}(B_1)}\log |y_1-y_2|),
\]
\[
\frac{|u(y_1)-u(y_2)|}{|y_1-y_2|}\leq -C_2\log |y_1-y_2|(\dinf_{B_{3/4}}u+\|b\|_{L^{\infty}(B_{3/4})})
\]
where $C_1>0$ depends only on $n,\sigma$ and $C_2>0$ depends only on $n,\sigma, \|a\|_{L^{\infty}(B_{3/4})}$.

\end{rem}

Next we have

\begin{lem} (Lemma 4.5 in \cite{CaS})\label{vertical weighted holder}
Let $g\in C^{\alpha}(B_1)$ for some $\al\in (0,1)$ and $U\in L^{\infty}(Q_1)\cap H(t^{1-2\sigma}, Q_1)$ be a weak solution of \eqref{ge 2.5}. Then there exists $\beta\in (0,1)$ depending only on $n,\sigma,\al$ such that $t^{1-2\sigma}\pa_t U\in C^{\beta}(\overline{Q_{1/2}})$. Moreover, there exists a positive constant $C>0$ depending only on $n,\sigma$ and $\beta$ such that
\[
\|t^{1-2\sigma}\pa_t U\|_{C^{\beta}(\overline{Q_{1/2}})}\leq C(\|U\|_{L^{\infty}(Q_1)}+\|g\|_{C^{\al}(B_1)}).
\]
\end{lem}

\begin{prop}\label{thm2.4-1}
Suppose that $K\in C^1(B_1)$, $ U\in H(t^{1-2\sigma},Q_1)$ and $U\geq 0$ in $Q_1$ is a weak solution of
\be\label{eq:fn1-1-1}
\begin{cases}
\mathrm{div}(t^{1-2\sigma}\nabla U)=0,&\quad \mbox{in }Q_1\\
-\dlim_{t\rightarrow 0}t^{1-2\sigma}\pa_t U(x,t)=K(x)U^p(x,0),&\quad \mbox{on }\pa' Q_1,
\end{cases}
\ee
where $1\leq p\leq \frac{n+2\sigma}{n-2\sigma}$.
Then there exist $C>0$ and $\al\in(0,1)$ both of which depend only on $n,\sigma,  p,\|U\|_{L^\infty(Q_1)}, \|K\|_{C^1(Q_1)}$ such that
\[
\nabla_x U \quad\text{and}\quad t^{1-2\sigma}\pa_t U \quad\text{are of}\quad  C^\al(\overline {Q_{1/2}})
\]
and
\[
\|\nabla_x U\|_{C^\al(\overline {Q_{1/2}})}+\|t^{1-2\sigma}\pa_t U\|_{C^{\al}(\overline {Q_{1/2}})}\leq C.
\]
\end{prop}

\begin{proof}
We use $C$ and $\al$ to denote various positive constants with dependence specified as in the proposition, which may vary from line to line. By Corollary \ref{thm2.3}, $U\in L^{\infty}_{loc}(Q_1\cup \partial' Q_1)$ and
\[
\|U\|_{C^{\al}(\overline{Q_{8/9}})}\leq C.
\]
With the above, we may apply Theorem \ref{thm2.1} to obtain $U(\cdot,0)\in C^{1,\sigma}(\overline {B_{7/8}})$ and
\[
\|U(\cdot, 0)\|_{C^{1,\sigma}(\overline {B_{7/8}})}\leq C.
\]
Hence we may differentiate \eqref{eq:fn1-1-1} with respect to $x$ (which can be justified from the proof of Proposition \ref{lem2.2}) and apply Proposition \ref{lem harnack} to $\nabla_x U$ to obtain
\[
\|\nabla_x U\|_{C^{\al}(\overline{Q_{1/2}})}\leq C.
\]
Finally we can apply Lemma \ref{vertical weighted holder} to obtain
\[
\|t^{1-2\sigma}\pa_t U\|_{C^{\al}(\overline {Q_{1/2}})}\leq C.
\]
\end{proof}

\section{Proof of Theorem \ref{thm3.1}}
\label{section of liouville}

We first introduce some notations. We say that $U\in L^\infty_{loc}(\overline {\R^{n+1}_+})$
if $U\in L^\infty(\overline {Q_R})$ for any $R>0$. Similarly we say
$U\in H_{loc}(t^{1-2\sigma},\overline {\R^{n+1}_+})$ if $U\in H(t^{1-2\sigma},\overline {Q_R})$ for any $R>0$. 


%

In the following $\B_R(X)$ is denoted as the ball in $\R^{n+1}$ with radius $R$ and center $X$, and $\B^+_R(X)$ as $\B_R(X)\cap \R^{n+1}_+$. We also write $\B_R(0), \B^+_R(0)$ as $\B_R, \B_R^+$ for short respectively. We start with a Lemma, which is a version of the strong maximum principle.

\begin{prop}\label{liminf}
Suppose $U(X)\in H(t^{1-2\sigma},D_{\va})\cap C(\B_1^+\cup B_1 \setminus \{0\})$ and $U>0$ in $\B_1^+\cup B_1 \setminus \{0\}$ is a weak supersolution of \eqref{2.5} with $a\equiv b\equiv 0$ and $D=D_{\va}:=\B_1^+\setminus\overline{\B_{\va}^+}$ for any $0<\va<1$,
then
\[
\liminf_{(x,t)\to 0}U(x,t)>0.
\]
\end{prop}

\begin{proof}
For any $\delta>0$, let
\[
V_{\delta}=U+\frac{\delta}{|(x,t)|^{n-2\sigma}}- \dmin_{\pa''\B^+_{0.8}}U.
\]
Then $V$ is also a weak supersolution in $D_{\delta^{\frac{2}{n-2\sigma}}}$. Applying Lemma \ref{wmp} to $V_{\delta}$ in $D_{\delta^{\frac{2}{n-2\sigma}}}$ for sufficiently small $\delta$, we have $V_{\delta}\geq 0$ in $D_{\delta^{\frac{2}{n-2\sigma}}}$.  For any $(x,t)\in \B^+_{0.8}\backslash\{0\}$, we have $\lim_{\delta\to 0}V_{\delta}(x,t)\geq 0$, i.e., $U(x,t)\geq \min_{\pa''\B^+_{0.8}}U$ .
\end{proof}

The proof of Theorem \ref{thm3.1} uses the method of moving spheres and is inspired by \cite{lz1}, \cite{lz} and \cite{cl}. For each $x\in\R^n$ and $\lda>0$, we define, $\X=(x,0)$, and
\be\label{kelvin}
U_{\X, \lda}(\xi):=\left(\frac{\lda}{|\xi-\X|}\right)^{n-2\sigma}U\left(\X+\frac{\lda^2(\xi-\X)}{|\xi-\X|^2}\right),
\quad \xi\in\overline{\R^{n+1}_+}\backslash\{\X\},
\ee
the Kelvin transformation of $U$ with respect to the ball $\B_{\lda}(\X)$. We point out that if $U$ is a solution of \eqref{eq}, then $U_{\bar x,\lda}$ is a solution of \eqref{eq} in $\R^{n+1}_+\setminus\overline{\B_{\va}^+}$, for every $\bar x\in \pa\R^{n+1}_+$, $\lda>0$, and $\va>0$.

By Corollary \ref{thm2.3} any nonnegative weak solution $U$ of \eqref{eq} belongs to $L^\infty_{loc}(\overline {\R^{n+1}_+})$, and hence by Proposition \ref{lem harnack}, $U$ is H\"older continuous and positive in $\overline{\R^{n+1}_+}$. By Theorem \ref{thm2.1-poisson}, $U(\cdot,0)$ is smooth in $\R^n$. From classical elliptic equations theory, $U$ is smooth in $\R^{n+1}_+$.

\begin{lem}\label{get started}
For any $x\in\R^n$, there exists a positive constant $\lda_0(x)$ such that for any $0<\lda<\lda_0(x)$,
\be
U_{\X, \lda}(\xi)\leq U(\xi), \quad \text{in } \R^{n+1}_+\backslash \B^+_{\lda}(\X).
\ee
\end{lem}

\begin{proof}
Without loss of generality we may assume that $x=0$ and write $U_{\lda}=U_{0,\lda}$.
\medskip

\textit{Step 1.} We show that there exist $0<\lda_1<\lda_2$ which may depend on $x$, such that
\[
U_{\lda}(\xi)\leq U(\xi), ~\forall~0<\lda<\lda_1,~\lda<|\xi|<\lda_2.
\]
For every $0<\lda<\lda_1<\lda_2$, $\xi\in\pa'' \B_{\lda_2}$, we have $\frac{\lda^2\xi}{|\xi|^2}\in \B^+_{\lda_2}$. Thus we can choose $\lda_1=\lda_1(\lda_2)$ small such that
\begin{equation*}
\begin{split}
U_{\lda}(\xi)&=\left(\frac{\lambda}{|\xi|}\right)^{n-2\sigma}U\left(\frac{\lda^2\xi}{|\xi|^2}\right)\\
&\leq\left(\frac{\lda_1}{\lda_2}\right)^{n-2\sigma}\sup\limits_{\overline{\B_{\lda_2}^+}}U\leq \inf_{\partial ''{\B_{\lda_2}^+}}U\leq U(\xi)
\end{split}
\end{equation*}
Hence
\[
U_{\lda}\leq U\quad \mbox{on } \partial ''(\B^+_{\lda_2}\backslash \B^+_{\lda})
\]
  for all $\lda_2>0$ and $0<\lda<\lda_1(\lda_2)$.

We will show that $U_{\lda}\leq U$ on $(\B^+_{\lda_2}\backslash \B^+_{\lda})$ if $\lda_2$ is small and $0<\lda<\lda_1(\lda_2)$.
Since $U_{\lda}$ satisfies \eqref{eq} in $\B_{\lda_2}^+\setminus\overline{\B_{\lda_1}^+}$, we have
\begin{equation}\label{diff}
\begin{cases}
\mathrm{div}(t^{1-2\sigma}\nabla (U_{\lda}-U))&=0, \quad \text{in}\quad \B_{\lda_2}^+\backslash \overline{\B_{\lda}^+};\\
\lim\limits_{t\to 0}t^{1-2\sigma}\partial_t (U_{\lda}-U)&=
U^{\frac{n+2\sigma}{n-2\sigma}}(x,0)-U_{\lda}^{\frac{n+2\sigma}{n-2\sigma}}(x,0),
\quad \text{on}\quad \partial '(\B_{\lda_2}^+\backslash \overline{\B_{\lda}^+}).
\end{cases}
\end{equation}
Let $(U_{\lda}-U)^+:=\max(0,U_{\lda}-U)$ which equals to $0$ on $\pa''(\B^+_{\lda_2}\backslash \B^+_{\lda})$. Hence, by a density argument, we can use $(U_{\lda}-U)^+$ as a test function in the definition of weak solution of \eqref{diff}. We will make use of the narrow domain technique
from \cite{BN}. With the help of the mean value theorem, we have
\begin{equation*}
\begin{split}
&\int_{\B_{\lda_2}^+\backslash \B_{\lda}^+} t^{1-2\sigma}|\nabla(U_{\lda}-U)^+|^2\\
&=\int_{B_{\lda_2}\backslash B_{\lda}}(U_{\lda}^{\frac{n+2\sigma}{n-2\sigma}}(x,0)-U^{\frac{n+2\sigma}{n-2\sigma}}(x,0))(U_{\lda}-U)^+\\
&\leq C \int_{B_{\lda_2}\backslash B_{\lda}}((U_{\lda}-U)^+)^2 U_{\lda}^{\frac{4\sigma}{n-2\sigma}}\\
&\leq C\left(\int_{B_{\lda_2}\backslash B_{\lda}}((U_{\lda}-U)^+)^\frac{2n}{n-2\sigma}\right)^{\frac{n-2\sigma}{n}}\left(\int_{B_{\lda_2}\backslash B_{\lda}} U_{\lda}^{\frac{2n}{n-2\sigma}}\right)^{\frac{2\sigma}{n}}\\
& \leq C\left(\int_{\B_{\lda_2}^+\backslash \B_{\lda}^+} t^{1-2\sigma}|\nabla(U_{\lda}-U)^+|^2\right)
\left(\int_{ B_{\lda_2}} U^{\frac{2n}{n-2\sigma}}\right)^{\frac{2\sigma}{n}}
\end{split}
\end{equation*}
where Proposition \ref{prop of trace weighted} is used in the last inequality and $C$ is a positive constant depending only on $n$ and $\sigma$.
We fix $\lda_2$ small such that
\[
C\left(\int_{B_{\lda_2}} U^{\frac{2n}{n-2\sigma}}\right)^{\frac{2\sigma}{n}}<1/2.
\]
Then $\nabla(U_{\lda}-U)^+=0$ in $\B_{\lda_2}^+\backslash \B_{\lda}^+$. Since $(U_{\lda}-U)^+=0$ on
 $\partial ''(\B^+_{\lda_2}\backslash \B^+_{\lda})$, $(U_{\lda}-U)^+=0$ in $\B_{\lda_2}^+\backslash \B_{\lda}^+$.
 We conclude that $U_{\lda}\leq U$ on $(\B^+_{\lda_2}\backslash \B^+_{\lda})$ for $0<\lda<\lda_1:=\lda_1(\lda_2)$.
\medskip

\textit{Step 2.}  We show that there exists $\lda_0\in (0,\lda_1)$ such that $\forall~0<\lda<\lda_0$
\[
U_{\lda}(\xi)\leq U(\xi),~|\xi|>\lda_2,~\xi\in \R^{n+1}_+.
\]
Let $\phi(\xi)=\left(\frac{\lda_2}{|\xi|}\right)^{n-2\sigma}\inf\limits_{\partial'' \B_{\lda_2}} U$,
which satisfies
\[
\begin{cases}
  \mathrm{div}(t^{1-2\sigma}\nabla \phi)=0,&\quad \mbox{in } \R^{n+1}_+\setminus \B_{\lda_2}^+\\
   -\lim_{t\to 0}t^{1-2\sigma}\pa_t \phi(x,t)=0,&\quad x\in \R^n\setminus \overline{B_{\lda_2}},
\end{cases}
\]
 and $\phi(\xi)\leq U(\xi)$ on $\partial'' \B_{\lda_2}$. By the weak maximum principle Lemma \ref{wmp},
\[
U(\xi)\geq \left(\frac{\lda_2}{|\xi|}\right)^{n-2\sigma}\inf\limits_{\partial'' \B_{\lda_2}} U,~\forall~|\xi|>\lda_2,~\xi\in \R^{n+1}_+.
\]
Let $\lda_0=\min(\lda_1, \lda_2(\inf\limits_{\partial'' \B_{\lda_2}} U/\sup\limits_{ \B_{\lda_2}} U)^{\frac{1}{n-2\sigma}})$.
Then for any $0<\lda<\lda_0,~|\xi|\geq \lda_2$, we have
\[
U_{\lda}(\xi)\leq (\frac{\lda}{|\xi|})^{n-2\sigma}U(\frac{\lda^2\xi}{|\xi|^2})
\leq (\frac{\lda_0}{|\xi|})^{n-2\sigma}\sup\limits_{\B_{\lda_2}}U\leq (\frac{\lda_2}{|\xi|})^{n-2\sigma}
\inf\limits_{\partial ''\B_{\lda_2}}U\leq U(\xi).
\]
Lemma \ref{get started} is proved.
\end{proof}

With Lemma \ref{get started}, we can define for all $x\in R^n$,
\[
\bar\lda(x)=\sup\{\mu>0: U_{\X,\lda}\leq U \text{ in }\R^{n+1}_+\backslash \B_{\lda}^+,~\forall~ 0<\lda<\mu\}.
\]
By Lemma \ref{get started}, $\bar\lda(x)\geq \lda_0(x).$

\begin{lem}\label{equal}
If $\bar\lda(x)<\infty$ for some $x\in\R^n$, then
\[
U_{\X,\bar\lda(x)}\equiv U.
\]
\end{lem}

\begin{proof}
Without loss of generality we assume that $x=0$ and write $U_{\lda}=U_{0,\lda}$ and $\bar\lda=\bar\lda(0)$. By the definition of $\bar\lda$,
\[
U_{\bar\lda}\geq U \mbox{ in } \overline{\B_{\bar\lda}^+}\backslash\{0\},
\]
and therefore, for all $0<\va<\bar\lda$,
\begin{equation}\label{eq:diff-super}
\begin{cases}
\mathrm{div}(t^{1-2\sigma}\nabla (U_{\lda}-U))&=0, \quad \text{in}\quad \B_{\lda}^+\backslash \overline{\B_{\va}^+};\\
-\lim\limits_{t\to 0}t^{1-2\sigma}\partial_t (U_{\lda}-U)&\geq 0
\quad \text{on}\quad \partial '(\B_{\lda}^+\backslash \overline{\B_{\va}^+}).
\end{cases}
\end{equation}
We argue by contradiction. If $U_{\bar\lda}$ is not identically equal to $U$,
applying the Harnack inequality Proposition \ref{lem harnack} to \eqref{eq:diff-super}, we have
\[
U_{\bar\lda}>U \text{ in }\overline{\B_{\bar\lda}}\backslash\{\{0\}\cup\partial '' \B_{\bar\lda}\},
\]
and in view of Proposition \ref{liminf},
\[
\liminf\limits_{\xi\to 0}(U_{\bar\lda}(\xi)-U(\xi))>0.
\]
So there exist $\va_1>0$ and $c>0$ such that $U_{\bar\lda}(\xi)>U(0)+c, ~\forall ~0<|\xi|<\va_1$. Choose $\va_2$ small such that
\[
\left(\frac{\bar\lda}{\bar\lda+\va_2}\right)^{n-2\sigma}(U(0)+c)>U(0)+\frac{c}{2}.
\]
Thus for all $0<|\xi|<\va_1$ and $\bar\lda<\lda<\bar\lda+\va_2$,
\[
U_{\lda}(\xi)=\left(\frac{\bar\lda}{\lda}\right)^{n-2\sigma}U_{\bar\lda}\left(\frac{\bar\lda^2 \xi}{\lda^2}\right)\geq
\left(\frac{\bar\lda}{\bar\lda+\va_2}\right)^{n-2\sigma}(U(0)+c)\geq U(0)+c/2.
\]
Choose $\va_3$ small such that for all $0<|\xi|<\va_3$, $U(0)>U(\xi)-c/4$. Hence for all $0<|\xi|<\va_3$ and $\bar\lda<\lda<\bar\lda+\va_2$,
\[
U_{\lda}(\xi)>U(\xi)+c/4.
\]
For $\delta$ small, which will be fixed later, denote $K_{\delta}=\{\xi\in\R^{n+1}_+: \va_3\leq|\xi|\leq\bar\lda-\delta\}$.
Then there exists $c_2=c_2(\delta)$ such that
\[
U_{\bar\lda}(X)-U(X)>c_2  \ \text{ in }\  K_{\delta}.
\]
By the uniform continuous of $U$ on compact sets, there exists $\va_4\leq\va_2$ such that for all $\bar\lda<\lda<\bar\lda+\va_4$
\[
U_{\lda}-U_{\bar\lda}>-c_2/2 \ \text{ in }\  K_{\delta}.
\]
Hence
\[
U_{\lda}-U>c_2/2 \ \text{ in } \ K_{\delta}.
\]
Now let us focus on the region $\{\xi\in\R^{n+1}_+: \bar\lda-\delta\leq|\xi|\leq\lda\}$.
Using the narrow domain technique as that in Lemma \ref{get started}, we can choose $\delta$
small (notice that we can choose $\va_4$ as small as we want) such that
\[
U_{\lda}\geq U \ \text{ in } \ \{\xi\in\R^{n+1}_+: \bar\lda-\delta\leq|\xi|\leq\lda\}.
\]
In conclusion there exists $\va_4$ such that for all $\bar\lda<\lda<\bar\lda+\va_4$
\[
U_{\lda}\geq U \ \text{ in }\ \{\xi\in\R^{n+1}_+: 0<|\xi|\leq\lda\}
\]
which contradicts with the definition of $\bar\lda$.
\end{proof}

\begin{proof}[Proof of Theorem \ref{thm3.1}] It follows from the same arguments in \cite{lz}, with the help of Lemma \ref{equal}, that:\\
\noindent (i)  Either $\bar\lda(x)=\infty$  for all $x\in\R^n$  or $\bar\lda(x)<\infty$  for all $x\in\R^n$; (Lemma 2.3 in \cite{lz})\\
\noindent (ii) If for all $x\in\R^n$ , $\bar\lda(x)=\infty$ then $U(x,t)=U(0,t)$, $~\forall~(x,t)\in\R^{n+1}_+$; (Lemma 11.3 in \cite{lz})\\
\noindent (iii) If  $\bar\lda(x)<\infty$ for all $x\in\R^n$, then by Lemma 11.1 in \cite{lz}
\begin{equation}\label{eq:u-bubble}
u(x):=U(x,0)=a\left(\frac{\lda}{1+\lda^2|x-x_0|^2}\right)^{\frac{n-2\sigma}{2}}
\end{equation}
where $\lda>0$, $a>0$ and $x_0\in \R^n$.

We claim that (ii) never happens, since this would imply, using \eqref{eq}, that
\[
U(x,t)=U(0)-U(0)^{\frac{n+2\sigma}{n-2\sigma}}\frac{t^{2\sigma}}{2\sigma}
\]
which contradicts to the positivity of $U$. Then (iii) holds.

We are only left to show that $V:=U-\mathcal P_{\sigma}[u]\equiv 0$ where $u(x)$ is given in \eqref{eq:u-bubble} and belongs to $\dot H^{\sigma}(\R^n)$. Hence, $V$ satisfies
\[
\begin{cases}
\mathrm{div}(t^{1-2\sigma}\nabla V)&=0,\quad\mbox{in }\R^{n+1}_+\\
\hspace{1.7cm}V&=0\quad\mbox{on }\partial\R^{n+1}_+.
\end{cases}
\]
By Lemma \ref{equal}, we know that $V_{\bar\lda}$ can be extended to a smooth function near $0$. Multiplying the above equation by $V$ and integrating by parts, it leads to
$
\int_{\R^{n+1}_+}t^{1-2\sigma}|\nabla V|^2=0.
$
Hence we have $V\equiv 0$.

Finally $a=\big(N_{\sigma}c_{n,\sigma}2^{2\sigma}\big)^{\frac{n-2\sigma}{4\sigma}}$ follows from \eqref{eq:equ of projection} with $\phi=1$ and \eqref{connection of fractional and extension}.

\end{proof}

\section{Local analysis near isolated blow up points}
\label{Local analysis near isolated blow up points}

The analysis in this and next section adapts the blow up analysis developed in \cite{SZ} and \cite{Li95} to give accurate blow up profiles for solutions of degenerate elliptic equations. For $\sigma=\frac 12$, similar results have been proved in \cite{HL99} and \cite{EG}, where equations are elliptic.

Let $\om\subset \mathbb{R}^n$ $(n\geq 2)$ be a domain, $\tau_i\geq 0$ satisfy
$\lim_{i\rightarrow \infty}\tau_i=0$, $p_i=(n+2\sigma)/(n-2\sigma)-\tau_i$, and $K_i\in C^{1,1}(\om)$ satisfy, for
some constants $A_1,A_2>0$, that
 \be\label{4.1}
\begin{split}
  1/A_1\leq K_i(x)&\leq A_1 \quad \mbox{for all } x\in \om,\\
  \|K_i\|_{C^{1,1}(\om)}&\leq A_2 .
\end{split}
 \ee
Let $u_i\geq 0$ in $\R^n$ and $u_i\in L^\infty(\om)\cap \dot{H}^\sigma(\R^n)$ satisfying
\be\label{4.2}
(-\Delta)^{\sigma}u_i=c(n,\sigma)K_i(x)u_i^{p_i}, \quad \mbox{in } \om.
\ee

We say that $\{u_i\}$ blows up if $\|u_i\|_{L^\infty(\om)}\to \infty$ as $i\to \infty$.

\begin{defn}\label{def4.1}
Suppose that $\{K_i\}$ satisfies (\ref{4.1}) and $\{u_i\}$ satisfies (\ref{4.2}).
We say a point $\overline y\in \om$ is an isolated blow up point of $\{u_i\}$ if there exist
$0<\overline r<\mbox{dist}(\overline y,\om)$, $\overline C>0$, and a sequence $y_i$ tending to $\overline y$, such that,
$y_i$ is a local maximum of $u_i$, $u_i(y_i)\rightarrow \infty$ and
\[
u_i(y)\leq \overline C |\overline y-y_i|^{-2\sigma/(p_i-1)} \quad \mbox{for all } y\in B_{\overline r}(y_i).
\]
\end{defn}

Let $y_i\rightarrow \overline y$ be an isolated blow up of $u_i$, define
\be\label{def:average}
\overline u_i(r)=\frac{1}{|\pa B_r|} \int_{\pa B_r(y_i)}u_i,\quad r>0,
\ee
and
\[
\overline w_j(r)=r^{2\sigma/(p_i-1)}\overline u_i(r), \quad r>0.
\]

\begin{defn}\label{def4.2}
We say $y_i \to \overline y\in \om$ is an isolated simple blow up point, if $y_i \to \overline y$ is an isolated blow up point, such that, for some
$\rho>0$ (independent of $i$) $\overline w_i$ has precisely one critical point in $(0,\rho)$ for large $i$.
\end{defn}

In this section, we are mainly concerned with the
profile of blow up of $\{u_i\}$. And under certain conditions, we can show that
isolated blow up points have to be isolated simple blow up points.

Let $u_i\in C^2(\om)\cap \D$ and $u_i\geq 0$ in $\R^n$ satisfy \eqref{4.2} with $K_i$ satisfying \eqref{4.1}.
Without loss of generality, we assume throughout this section that $B_2\subset \om$ and
$y_i\to 0$ as $i\to \infty$ is an isolated blow up point of
$\{u_i\}$ in $\om$. Let $U_i=\mathcal P_{\sigma}[u_i]$ be the extension of $u_i$ (see \eqref{poisson involution}). Then we have
\be\label{4.4}
\begin{cases}
\mathrm{div}(t^{1-2\sigma}\nabla U_i)=0,&\quad \mbox{in } \R^{n+1}_+,\\
-\dlim_{t\rightarrow 0}t^{1-2\sigma}\frac{\pa U_i(x,t)}{\pa t}=c_0 K_i(x)U_i(x,0)^{p_i},&\quad \mbox{for any } x\in \om,
\end{cases}
\ee
where $c_0=N_{\sigma}c(n,\sigma)$ with $N_\sigma=2^{1-2\sigma}\Gamma(1-\sigma)/\Gamma(\sigma)$.

\begin{lem}\label{lem4.1} Suppose that $u_i\in C^2(\om)\cap\D$
 and $u_i\geq 0$ in $\R^n$ satisfies (\ref{4.2}) with $\{K_i\}$ satisfying (\ref{4.1}),
and $y_i\rightarrow 0$
is an isolated blow up point of $\{u_i\}$, i.e., for some positive constants $A_3$ and $\bar r$ independent of $i$,
\be\label{4.7}
|y-y_i|^{2\sigma/(p_i-1)}u_i(y)\leq A_3,\quad \mbox{for all } y\in B_{\bar r}\subset\om.
\ee Denote $U_i=\mathcal P_{\sigma}[u_i]$, and  $Y_i=(y_i,0)$.
Then for any $0<r<\frac13 \overline r$, we have the following Harnack inequality
\[
\sup_{\mathcal{B}^+_{2r}(Y_i)\setminus\overline{\mathcal{B}^+_{r/2}(Y_i)}} U_i\leq C \inf_{\mathcal{B}^+_{2r}(Y_i)\setminus\overline{\mathcal{B}^+_{r/2}(Y_i)}} U_i,
\]
where $C$ is a positive constant depending only on $n, \sigma, A_3, \bar r$ and $\dsup_i\|K_i\|_{L^\infty(B_{\overline r}(y_i))}$.
\end{lem}

\begin{proof}
For $0<r<\frac{\bar r}{3}$, set
\[
V_i(Y)=r^{2\sigma/(p_i-1)}U_i(Y_i+rY),\quad \mbox{in } Y\in \mathcal{B}_3^+.
\]
It is easy to see that
\[
\mathrm{div}(s^{1-2\sigma}\nabla V_i)=0,\quad \mbox{in } \mathcal{B}^+_3,
\]
and
\[
-\dlim_{s\rightarrow 0} s^{1-2\sigma} \pa_s V_i(y,s)=c_0K(y_i+ry)V_i(y,0)^{p_i},\quad \mbox{on } \pa'\mathcal{B}^+_3.
\]
Since $y_i\to 0$ is an isolated blow up point of $u_i$,
\[
V_i(y,0)\leq A_3 |y|^{-2\sigma/(p_i-1)},\quad \mbox{for all } y\in B_3.
\]
Lemma \ref{lem4.1} follows after applying Proposition \ref{lem harnack} and the standard Harnack inequality for uniform elliptic equation together to $V_i$ in the domain
$Q_2\setminus \overline Q_{1/2}$.
\end{proof}

\begin{prop}\label{prop4.1}
Suppose  that $u_i\in C^2(\om)\cap\D$ and $u_i\geq 0$ in $\R^n$
satisfies (\ref{4.2}) with $K_i\in C^{1,1}(\om)$ satisfying (\ref{4.1}).
Suppose also that $y_i\to 0$ be an isolated blow up point of $\{u_i\}$ with \eqref{4.7}.
Then for any $R_i\rightarrow \infty$, $\va_i\rightarrow 0^+$, we have,
after passing to a subsequence (still denoted as $\{u_i\}$,
$\{y_i\}$, etc. ...), that
\[
\|m_i^{-1}u_i(m_i^{-(p_i-1)/2\sigma}\cdot+y_i)-(1+k_i|\cdot|^2)^{(2\sigma-n)/2}\|_{C^2(B_{2R_i}(0))}\leq \va_i,
\]
\[
R_im_i^{-(p_i-1)/2\sigma}\rightarrow 0\quad \mbox{as}\quad i\rightarrow \infty,
\]
where $m_i=u_i(y_i)$ and $k_i=K_i(y_i)^{1/\sigma}/4$.
\end{prop}

\begin{proof}
Let
\[
\phi_i(x)=m_i^{-1}u_i(m_i^{-(p_i-1)/2\sigma}x+y_i),\quad \mbox{for }x\in\mathbb{R}^n.
\]
It follows that
\[
(-\Delta)^\sigma \phi_i(x)=c(n,\sigma)K_i(m_i^{-(p_i-1)/2\sigma}x+y_i)\phi_i^{p_i},
\]
\be\label{eq:prop4.1-123}
0<\phi_i\leq A_3 |x|^{-2\sigma/(p_i-1)},\quad |x|<\overline r m_i^{(p_i-1)/2\sigma},
\ee
and
\[
\phi_i(0)=1,\quad \nabla \phi_i(0)=0.
\]

Let $\Phi_i=\mathcal P_{\sigma}[\phi_i]$ be the extension of $\phi_i$ (see \eqref{poisson involution}). Then $\Phi$ satisfies
\[
\begin{cases}
\mathrm{div}(t^{1-2\sigma}\nabla \Phi_i(x,t))&=0,\quad |x,t|<\bar rm_i^{\frac{p_i-1}{2\sigma}},\\
-\lim\limits_{t\to 0}t^{1-2\sigma}\partial_t\Phi_i(x,t)&=N_{\sigma}c(n,\sigma)K_i(m_i^{-\frac{p_i-1}{2\sigma}}x+y_i)\Phi_i(x,0)^{p_i},\quad |x|<\bar rm_i^{\frac{p_i-1}{2\sigma}}.
\end{cases}
\]
 By the weak maximum principle we have, for any
$0<r<1$, $1=\phi_i(0)=\Phi_i(0,0)\geq \dmin_{\pa''\mathcal{B}_r}\Phi_i$.
It follows from Lemma \ref{lem4.1} that
\[
\dmax_{\pa B_r}\phi_i\leq\dmax_{\pa'' \B_r}\Phi_i\leq C\min_{\pa'' \B_r}\Phi_i\leq C.
\]
Namely,
\[
\dmax_{\overline{B_1}}\phi_i\leq C
\]
for some $C>0$ depending on $n,\sigma, A_1, A_2, A_3$. This and \eqref{eq:prop4.1-123} implies that for any $R>1$
\[
\dmax_{\overline{B_R}}\phi_i\leq C(R)
\]
for some $C(R)>0$ depending on $n,\sigma, A_1, A_2, A_3$ and $R$.
 Then by Corollary \ref{thm2.3} there exists some $\al\in (0,1)$ such that for every $R>1$,
\[
\|\Phi_i\|_{H(t^{1-2\sigma},Q_R)}+\|\Phi_i\|_{C^{\al}(\overline{Q_R})}\leq C_1(R),
\]
where $\al$ and $C_1(R)$ are independent of $i$.
Bootstrap using Theorem \ref{thm2.2}, we have, for every $0<\beta<2$ with $2\sigma+\beta\not\in \mathbb{N}$,
\[
\|\phi_i\|_{C^{2\sigma+\beta}(\overline{B_R})}\leq C_2(R,\beta)
\]
where $C_2(R,\beta)$ is independent of $i$.
Thus, after passing to a subsequence, we have, for some nonnegative functions
$\Phi(X)\in H_{loc}(t^{1-2\sigma},\overline{\mathbb{R}^{n+1}})\cap C^{\al}_{loc}(\overline{\mathbb{R}^{n+1}})$ and $\phi\in C^2(\R^n)$,
\[
\begin{cases}
\Phi_i&\rightharpoonup\Phi\quad\mbox{weakly in }H_{loc}(t^{1-2\sigma},\R^{n+1}_+),\\
\Phi_i&\rightarrow\Phi\quad\mbox{in }C^{\al/2}_{loc}(\overline{\R^{n+1}_+}),\\
\phi_i&\rightarrow\phi\quad\mbox{in }C^2_{loc}(\R^n).
\end{cases}
\]
It follows that
\[
\Phi(\cdot,0)\equiv\phi,\quad \phi(0)=1, \quad \nabla \phi(0)=0,
\]
and $\Phi$ satisfies
\[
\begin{cases}
\mathrm{div}(t^{1-2\sigma}\nabla \Phi)=0\quad &\mbox{in } \mathbb{R}^{n+1},\\
-\dlim_{t\rightarrow 0} t^{1-2\sigma}\pa_t \Phi(x,t)=c_0K\Phi(x,0)^{(n+2\sigma)/(n-2\sigma)}\quad &\mbox{on } \pa'\mathbb{R}^{n+1},
\end{cases}
\]
with $K=\lim\limits_{i\to\infty}K_i(y_i)$.
By Theorem \ref{thm3.1}, we have
\[
\phi(x)=(1+\lim_{i\rightarrow \infty}k_i|x|^2)^{(2\sigma-n)/2},
\]
where $k_i=K_i(y_i)^{1/\sigma}/4$. Proposition \ref{prop4.1} follows immediately.
\end{proof}

Note that since passing to subsequences does not affect our proofs, we will always choose $R_i\to\infty$ first, and then $\va_i\to 0^+$ as small as we wish (depending on $R_i$) and then choose our subsequence $\{u_i\}$ to work with.

\begin{prop}\label{prop4.2} Under the hypotheses of Proposition \ref{prop4.1}, there exists some positive constant $C=C(n,\sigma, A_1,A_2,A_3)$ such that,
\[
u_i(y)\geq C^{-1}m_i(1+k_im_i^{(p_i-1)/\sigma}|y-y_i|^2)^{(2\sigma-n)/2}, \quad |y-y_i|\leq 1.
\]
In particular, for any $e\in \mathbb{R}^n$, $|e|=1$, we have
\[
u_i(y_i+e)\geq C^{-1}m_i^{-1+((n-2\sigma)/2\sigma)\tau_i}.
\]
where $\tau_i=(n+2\sigma)/(n-2\sigma)-p_i$.
\end{prop}

\begin{proof}
Denote $r_i=R_i m_i^{-(p_i-1)/2\sigma}$.  It follows from Proposition \ref{prop4.1} that $r_i\to 0$ and
\[
u_i(y)\geq C^{-1}m_iR_i^{2\sigma-n}, \quad \mbox{for all } |y-y_i|=r_i.
\]
By the Harnack inequality Lemma \ref{lem4.1}, we have
\[
U_i(Y)\geq C^{-1}m_iR_i^{2\sigma-n}, \quad \mbox{for all } |Y-Y_i|=r_i,
\]
where $U_i=\mathcal P_{\sigma}[u_i]$ is the extension of $u_i$, $Y=(y,s)$ with $s\geq 0$, and $Y_i=(y_i,0)$.

Set
\[
\Psi_i(Y)=C^{-1}R_i^{2\sigma-n}r_i^{n-2\sigma}m_i(|Y-Y_i|^{2\sigma-n}-(\frac{3}{2})^{2\sigma-n}),\quad r_i\leq |Y-Y_i|\leq \frac{3}{2}.
\]
Clearly, $\Psi_i$ satisfies
\[
\mathrm{div}(s^{1-2\sigma}\nabla\Psi_i )=0=\mathrm{div}(s^{1-2\sigma}\nabla U_i), \quad r_i\leq |Y-Y_i|\leq \frac{3}{2},
\]
\[
\Psi_i(Y)\leq U_i(Y),\quad \mbox{on } \pa''\mathcal{B}_{r_i}\cup \pa''\mathcal{B}_{3/2},
\]
\[
-\dlim_{s\rightarrow 0^+}s^{1-2\sigma}\pa_s\Psi_i(y,s)
=0\leq -\dlim_{s\rightarrow 0^+}s^{1-2\sigma}\pa_sU_i(y,s),\quad \quad r_i\leq |y-y_i|\leq \frac{3}{2}.
\]
By the weak maximum principle Lemma \ref{wmp} applied to $U_i-\Psi_i$, we have
\[
U_i(Y)\geq \Psi_i(Y)\quad \mbox{for all } r_i\leq |Y-Y_i|\leq \frac32.
\]
Therefore, Proposition \ref{prop4.2} follows immediately from Proposition \ref{prop4.1}.
\end{proof}

\begin{lem}\label{lem4.3} Under the hypotheses of Proposition \ref{prop4.1},
and in addition that $y_i\to 0$ is also an isolated simple blow up point with the constant $\rho$, there exist $\delta_i>0$, $\delta_i=O(R_i^{-2\sigma+o(1)})$,
such that
\[
u_i(y)\leq C_1 u_i(y_i)^{-\lda_i}|y-y_i|^{2\sigma-n+\delta_i},\quad \mbox{for all }r_i\leq |y-y_i|\leq 1,
\]
where $\lda_i=(n-2\sigma-\delta_i)(p_i-1)/2\sigma-1$ and $C_1$ is some positive constant depending only on $n,\sigma, A_1,A_3$ and $\rho$.
\end{lem}

\begin{proof}
From Proposition \ref{prop4.1}, we see that
\be\label{4.8}
u_i(y)\leq C u_i(y_i)R_i^{2\sigma-n} \quad \mbox{for all } |y-y_i|=r_i.
\ee
Let $\overline u_i(r)$ be the average of $u_i$ over the sphere of radius $r$ centered at $y_i$.
It follows from the assumption of isolated simple
blow up and Proposition \ref{prop4.1} that
\be\label{4.9}
r^{2\sigma/(p_i-1)}\overline u_i(r) \quad \mbox{is strictly decreasing for $r_i<r<\rho$}.
\ee
By Lemma \ref{lem4.1},  \eqref{4.9} and \eqref{4.8},  we have, for all $r_i<|y-y_i|<\rho$,
\[
\begin{split}
|y-y_i|^{2\sigma/(p_i-1)}u_i(y)&\leq C|y-y_i|^{2\sigma/(p_i-1)}\overline u_i(|y-y_i|)\\&
\leq r_i^{2\sigma/(p_i-1)}\overline u_i(r_i)\\&
\leq CR_i^{\frac{2\sigma-n}{2}+o(1)},
\end{split}
\]
where $o(1)$ denotes some quantity tending to $0$ as $i\to \infty$.
Applying Lemma \ref{lem4.1} again, we obtain
\be\label{4.10}
U_i(Y)^{p_i-1}\leq O(R_i^{-2\sigma+o(1)})|Y-Y_i|^{-2\sigma}\quad \mbox{for all }  r_i\leq |Y-Y_i|\leq \rho.
\ee
Consider operators
\[
\begin{cases}\mathfrak{L}(\Phi)=\mathrm{div}(s^{1-2\sigma}\nabla \Phi(Y)), \quad &\mbox{in } \mathcal{B}^+_2,\\
L_i(\Phi)=-\dlim_{s\rightarrow 0^+}s^{1-2\sigma}\pa_s \Phi(y,s)-c_0K_iu_i^{p_i-1}(y)\Phi(y,0),\quad &\mbox{on } \pa'\mathcal{B}^+_2.
\end{cases}
\]
Clearly, $U_i>0$ satisfies $\mathfrak{L}(U_i)=0$ in $\mathcal{B}^+_2$ and $L_i(U_i)=0$ on $\pa'\mathcal{B}^+_2$.

For $0\leq \mu\leq n-2\sigma$, a direct computation yields
\[
\begin{split}
&\mathfrak{L}(|Y-Y_i|^{-\mu}-\va s^{2\sigma}|Y-Y_i|^{-(\mu+2\sigma)})\\&
=s^{1-2\sigma}|Y-Y_i|^{-(\mu+2)}\Big\{-\mu(n-2\sigma-\mu)+\frac{\va(\mu+2\sigma)(n-\mu)s^{2\sigma}}{|Y-Y_i|^{2\sigma}}\Big\}
\end{split}
\]
and
\[
L_i(|Y-Y_i|^{-\mu}-\va s^{2\sigma}|Y-Y_i|^{-(\mu+2\sigma)})=|Y-Y_i|^{-(u+2\sigma)}(2\va\sigma-c_0K_iu_i^{p_i-1}|Y-Y_i|^{2\sigma}).
\]

It follows from \eqref{4.10} that we can choose $\va_i=O(R_i^{-2\sigma +o(1)}) >0$, and then choose $\delta_i=O(R_i^{-2\sigma +o(1)})>0$
such that for $r_i<|y-y_i|<\rho$,
\[
\begin{split}
&L_i(|Y-Y_i|^{-\delta_i}-\va_i s^{2\sigma}|Y-Y_i|^{-(\delta_i+2\sigma)})\geq 0,\\
&L_i(|Y-Y_i|^{2\sigma-n+\delta_i}-\va_i s^{2\sigma}|Y-Y_i|^{-n+\delta_i})\geq 0
\end{split}
\]
and for $r_i<|Y-Y_i|<\rho$,
\[
\begin{split}
&\mathfrak{L}(|Y-Y_i|^{-\delta_i}-\va_i s^{2\sigma}|Y-Y_i|^{-(\delta_i+2\sigma)})\leq 0,\\
&\mathfrak{L}(|Y-Y_i|^{2\sigma-n+\delta_i}-\va_i s^{2\sigma}|Y-Y_i|^{-n+\delta_i})\leq 0.
\end{split}
\]

Set $M_i=2\max_{\pa'' \mathcal{B}^+_\rho} U_i$, $\lda_i=(n-2\sigma-\delta_i)(p_i-1)/2\sigma-1$ and
\[
\begin{split}
\Phi_i=&M_i\rho^{\delta_i}(|Y-Y_i|^{-\delta_i}-\va_i s^{2\sigma}|Y-Y_i|^{-(\delta_i+2\sigma)})\\
&+2Au_i(y_i)^{-\lda_i}(|Y-Y_i|^{2\sigma-n+\delta_i}-\va_i s^{2\sigma}|Y-Y_i|^{-n+\delta_i}),
\end{split}
\]
where $A>1$ will be chosen later.
By the choice of $M_i$ and  $\lda_i$, we immediately have
\[
\Phi_i(Y)\geq M_i\geq U_i(Y) \quad \mbox{for all } |Y-Y_i|=\rho.
\]
\[
\Phi_i\geq AU_i(Y_i)R_i^{2\sigma-n+\delta_i}\geq AU_i(Y_i)R_i^{2\sigma-n}\quad \mbox{for all }|Y-Y_i|=r_i.
\]
Due to (\ref{4.10}), we can choose $A$ to be sufficiently large such that
\[
\Phi_i\geq U_i\quad \mbox{for all }|Y-Y_i|=r_i.
\]
Therefore, applying maximum principles in section \ref{section of two maximum principle}  to $\Phi_i-U_i$ in $\B_{\rho}\backslash \overline{\B_{r_i}}$, it yields
\[
U_i\leq \Phi_i\quad \mbox{for all }r_i\leq |Y-Y_i|\leq \rho.
\]
For $r_i<\theta<\rho$, the same arguments as that in (\ref{4.10}) yield
\[
\begin{split}
\rho^{2\sigma/(p_i-1)}M_i&\leq C \rho^{2\sigma/(p_i-1)}\overline u_i(\rho)\\
&\leq C\theta^{2\sigma/(p_i-1)}\overline u_i(\theta)\\
&\leq C\theta^{2\sigma/(p_i-1)}\{M_i\rho^{\delta_i}\theta^{-\delta_i}+Au_i(y_i)^{-\lda_i}\theta^{2\sigma-n+\delta_i}\}.
\end{split}
\]
Choose $\theta=\theta(n,\sigma,\rho,A_1,A_2, A_3)$ sufficiently small so that
\[
C\theta^{2\sigma/(p_i-1)}\rho^{\delta_i}\theta^{-\delta_i}\leq \frac12 \rho^{2\sigma/(p_i-1)}.
\]
It follows that
\[
M_i\leq Cu_i(y_i)^{-\lda_i}.
\]
Then Lemma \ref{lem4.3} follows from the above and the Harnack inequality. 
\end{proof}

Below we are going to improve the estimate in Lemma \ref{lem4.3}. First, we prove a Pohozaev type identity.

\begin{prop}\label{thm3.3}
Suppose that $K\in C^{1}(B_{2R})$. Let $U\in H(t^{1-2\sigma},\B_{2R}^+)$ and $U\geq 0$ in $\B_{2R}^+$ be a weak solution of
\be\label{poh}
\begin{cases}
\mathrm{div}(t^{1-2\sigma}\nabla U)=0,&\quad \mbox{in }\mathcal{B}_{2R}^+\\
-\dlim_{t\rightarrow 0}t^{1-2\sigma}\pa_t U(x,t)=K(x)U^p(x,0),&\quad \mbox{on }\pa'\B_{2R}^+,
\end{cases}
\ee
where $p>0$. Then
\be\label{eq:poh-id}
\int_{\pa' \mathcal{B}^+_R} B'(X,U,\nabla U,R,\sigma)+\int_{\pa'' \mathcal{B}^+_R}t^{1-2\sigma} B''(X,U,\nabla U,R,\sigma)=0,
\ee
where
\[
B'(X,U,\nabla U,R,\sigma)=\frac{n-2\sigma}{2}KU^{p+1}+\langle X,\nabla U\rangle  KU^p
\]
and
\[
B''(X,U,\nabla U,R,\sigma)=\frac{n-2\sigma}{2}U\frac{\pa U}{\pa \nu}-\frac{R}{2}|\nabla U|^2+R|\frac{\pa U}{\pa \nu}|^2.
\]
\end{prop}
\begin{proof} Let $\Omega_{\va}=\B_R^+\cap\{t>\va\}$ for small $\va>0$. Multiplying \eqref{poh} by $\langle X,\nabla U\rangle$ and integrating by parts in $\Omega_{\va}$, we have, with notations $\pa'\om_{\va}=\mbox{interior of }\overline{\om_{\va}}\cap\{t=\va\}$, $\pa''\om_{\va}=\pa\om_{\va}\setminus\pa'\om_{\va}$ and $\nu=$ unit outer normal of $\pa\om_{\va}$,
\be\label{eq:ph1}
\begin{split}
&-\int_{\pa' \Omega_{\va}}t^{1-2\sigma}\pa_t U\langle X,\nabla U\rangle  + \int_{\pa'' \Omega_{\va}}t^{1-2\sigma}R|\frac{\pa U}{\pa \nu}|^2\\
&=\quad \int_{\Omega_{\va}} t^{1-2\sigma}|\nabla U|^2+\frac{1}{2}\int_{\Omega_{\va}} t^{1-2\sigma}X\cdot \nabla(|\nabla U|^2)\\
&=\quad -\frac{n-2\sigma}{2}\int_{\Omega_{\va}} t^{1-2\sigma}|\nabla U|^2+\frac{1}{2}\int_{\pa'' \Omega_{\va}}t^{1-2\sigma}R|\nabla U|^2\\
&\quad\quad -\frac{1}{2}\int_{\pa' \Omega_{\va}}t^{2-2\sigma}|\nabla U|^2.
\end{split}
\ee
Multiplying \eqref{poh} by $U$ and integrating by parts in $\Omega_{\va}$, we have
\be\label{eq:ph2}
\begin{split}
\int_{\Omega_{\va}} t^{1-2\sigma}|\nabla U|^2=-\int_{\pa' \Omega_{\va}} t^{1-2\sigma}U \pa_t U  + \int_{\pa'' \Omega_{\va}}t^{1-2\sigma}\frac{\pa U}{\pa \nu}U.
\end{split}
\ee
By Corollary \ref{thm2.3} and Proposition \ref{thm2.4-1}, there exists some $\al\in (0,1)$ such that $U$, $\nabla_x U$, and $t^{1-2\sigma}\pa_t U$ belong to $C^{\al}(\overline{\B_r^+})$ for all $r<2R$.
With this we can send $\va\to 0$ as follows. By \eqref{poh},
\[
-t^{1-2\sigma}\pa_t U(x,t)\to K(x)U^p(x,0)\quad \text{uniformly in }\ \B_{3R/2}\text{ as }\ t \to 0.
\]
Hence \eqref{eq:poh-id} follows by sending $\va\to 0$ in \eqref{eq:ph1} and \eqref{eq:ph2}.
\end{proof}

\begin{lem}\label{lem4.4}
Under the assumptions in Lemma \ref{lem4.3}, we have
\[
 \tau_i=O(u_{i}(y_i)^{-2/(n-2\sigma)+o(1)}),
\]
and thus
\[
 u_i(y_i)^{\tau_i}=1+o(1).
\]
\end{lem}

\begin{proof} Since $U_i$ satisfies \eqref{4.4}
and $\mathrm{div}(y-y_i)=n$, we have, using integration by part,
\[
\begin{split}
 &\frac{1}{c_0}\int_{\pa'\mathcal{B}^+_1(Y_i)} B'(Y,U_i,\nabla U_i,1,\sigma)\\&=\frac{n-2\sigma}{2n}\int_{\pa'\mathcal{B}^+_1(Y_i)}
\mathrm{div}(y-y_i)K_iU^{p_i+1}\\&\quad +\frac{1}{p_i+1}\int_{\pa'\mathcal{B}^+_1(Y_i)}\langle y-y_i,\nabla_y U^{p_i+1}_i\rangle K_i\\&
=-\frac{n-2\sigma}{2n}\int_{\pa'\mathcal{B}^+_1(Y_i)}\left[\langle y-y_i,\nabla_y K_i\rangle U_i^{p_i+1}+\langle y-y_i,\nabla_y U^{p_i+1}_i\rangle K_i\right]\\&
\quad +\frac{n-2\sigma}{2n}\int_{\pa B_1(y_i)}K_i U_i^{p_i+1} +\frac{1}{p_i+1}\int_{\pa'\mathcal{B}^+_1(Y_i)}\langle y-y_i,\nabla_y U^{p_i+1}_i\rangle K_i\\&
=\frac{\tau_i(n-2\sigma)^2}{2n(2n-\tau_i(n-2\sigma))}\int_{\pa'\mathcal{B}^+_1(Y_i)}\langle y-y_i,\nabla_y U^{p_i+1}_i\rangle K_i\\&
\quad -\frac{n-2\sigma}{2n}\int_{\pa'\mathcal{B}^+_1(Y_i)}\langle y-y_i,\nabla_y K_i\rangle U_i^{p_i+1}+\frac{n-2\sigma}{2n}\int_{\pa B_1(y_i)}K_i U_i^{p_i+1}
\end{split}
\]
and
\[
 \begin{split}
  &\int_{\pa'\mathcal{B}^+_1(Y_i)}\langle y-y_i,\nabla_y U^{p_i+1}_i\rangle K_i\\&=-n\int_{\pa'\mathcal{B}^+_1(Y_i)}K_iU^{p_i+1}_i
-\int_{\pa'\mathcal{B}^+_1(Y_i)}\langle y-y_i,\nabla_y K_i\rangle U_i^{p_i+1}+\int_{\pa B_1(y_i)}K_i U_i^{p_i+1}.
 \end{split}
\]
Combining the above two, together with Proposition \ref{thm3.3},  we conclude that
\be\label{4.11}
\begin{split}
 \tau_i\int_{\pa'\mathcal{B}^+_1(Y_i)}U_i^{p_i+1}\leq &C(n,\sigma,A_1,A_2)\Big\{\int_{\pa'\mathcal{B}^+_1(Y_i)}|y-y_i|U_i^{p_i+1}
\\&+\int_{\pa B_1(y_i)} U_i^{p_i+1}+ \int_{\pa''\mathcal{B}^+_1(Y_i)}t^{1-2\sigma}|B''(Y,U_i,\nabla U_i,1,\sigma)|\Big\}.
\end{split}
\ee
Since $U_i=u_i$ on $\pa'\B_1(Y_i)=B_1(y_i)\times \{0\}$, it follows from Proposition \ref{prop4.2} that
\be\label{4.12}
\begin{split}
\int_{\pa'\B_1(Y_i)}U_i^{p_i+1}&=\int_{B_1(y_i)}u_i^{p_i+1}\\&
\geq C^{-1}\int_{B_1(y_i)}\frac{m_i^{p_i+1}}{(1+|m_i^{(p_i-1)/2\sigma}(y-y_i)|^2)^{(n-2\sigma)(p_i+1)/2}}\\&
\geq C^{-1} m_i^{\tau_i(n/2\sigma-1)}\int_{B_{m_i^{(p_i-1)/2\sigma}}}\frac{1}{(1+|z|^2)^{(n-2\sigma)(p_i+1)/2}}\\&
\geq C^{-1}m_i^{\tau_i(n/2\sigma-1)},
\end{split}
\ee
where we used change of variables $z=m_i^{(p_i-1)/2\sigma}(y-y_i)$ in the second inequality.

By Proposition \ref{thm2.4-1} and Lemma \ref{lem4.3}, it is easy to see that the last two integral terms of right-handed side of \eqref{4.11} are in $O(m_i^{-2+o(1)})$.
By Proposition \ref{prop4.1}, we have
\be\label{4.13}
\begin{split}
\int_{\pa'\B_{r_i}(Y_i)}|Y-Y_i|U_i^{p_i+1}&=\int_{B_{r_i}(y_i)}|y-y_i|u_i^{p_i+1}\\&
\leq C\int_{B_{r_i}(y_i)}\frac{|y-y_i|m_i^{p_i+1}}{(1+|m_i^{(p_i-1)/2\sigma}(y-y_i)|^2)^{(n-2\sigma)(p_i+1)/2}}\\&
\leq C m_i^{-2/(n-2\sigma)+o(1)}\int_{B_{R_i}}\frac{|z|}{(1+|z|^2)^{n+o(1)}}\\&
\leq C m_i^{-2/(n-2\sigma)+o(1)}.
\end{split}
\ee
By Lemma \ref{lem4.3} and that $R_i\to\infty$, we have
\be\label{4.14}
\begin{split}
\int_{\pa'\B_1(Y_i)\setminus \pa'\B_{r_i}(Y_i)}|Y-Y_i|U_i^{p_i+1}&=\int_{B_1(y_i)\setminus B_{r_i}(y_i)}|y-y_i|u_i^{p_i+1}\\&
\leq m_i^{-\lda_i(p_i+1)}r_i^{n+1+(2\sigma -n+\delta_i)(p_i+1)}\\&
= o(m_i^{-2/(n-2\sigma)+o(1)}).
\end{split}
\ee
Combining \eqref{4.11}, \eqref{4.12}, \eqref{4.13}, \eqref{4.14} and that $\tau_i=o(1)$, we complete the proof.
\end{proof}

\begin{prop}\label{prop4.3} Under the assumptions in Lemma \ref{lem4.3}, we have
\[
u_i(y)\leq Cu_i^{-1}(y_i)|y-y_i|^{2\sigma-n},\quad \mbox{for all } |y-y_i|\leq 1.
\]
\end{prop}

Our proof of this Proposition makes use of the following

\begin{lem}\label{thm3.2}
Let $n\geq 2$. Suppose that for all $\va\in(0,1)$, $U\in H(t^{1-2\sigma}, \mathcal{B}_1^+\setminus\overline{\mathcal{B}_{\va}^+})$ and $U>0$ in $\mathcal{B}_1^+\setminus\overline{\mathcal{B}_{\va}^+}$ be a weak solution of
\be\label{eq:bocher}
\begin{cases}
\mathrm{div}(t^{1-2\sigma}\nabla U)=0\quad& \mbox{in }\mathcal{B}_1^+\setminus\overline{\mathcal{B}_{\va}^+},\\
-\dlim_{t\rightarrow 0}t^{1-2\sigma}\pa_t U(x,t)=0,\quad &\mbox{in } B_1\setminus \overline{B_{\va}^+}.
\end{cases}
\ee
Then
\[
U(X)=A |X|^{2\sigma -n}+H(X),
\]
where $A$ is a nonnegative constant and $H(X)\in H(t^{1-2\sigma}, \mathcal{B}_1^+)$ satisfies
\be\label{local fractional harmonic}
\begin{cases}
\mathrm{div}(t^{1-2\sigma}\nabla H)=0\quad& \mbox{in }\mathcal{B}_1^+,\\
-\dlim_{t\rightarrow 0}t^{1-2\sigma}\pa_t H(x,t)=0,\quad &\mbox{in } B_1.
\end{cases}
\ee
\end{lem}

The proof of Lemma \ref{thm3.2} is provided in Appendix \ref{A Bocher type theorem}.

\begin{proof}[Proof of Proposition \ref{prop4.3}] For $|y-y_i|<r_i$, it follows from Proposition \ref{prop4.1} that
\be\label{4.15}
\begin{split}
u_i(y)&\leq C m_i\left(\frac{1}{1+|m_i^{(p_i-1)/2\sigma}(y-y_i)|^2}\right)^{(n-2\sigma)/2}\\&
\leq Cm_i^{-1-\frac{n-2\sigma}{2\sigma}\tau_i}|y-y_i|^{2\sigma-n}\\&
\leq Cm_i^{-1}|y-y_i|^{2\sigma-n},
\end{split}
\ee
where Lemma \ref{lem4.4} is used the last inequality.

Suppose $|y-y_i|\geq r_i$. Let $e\in \overline{\R}^{n+1}_+$ with $|e|=1$, and set
$V_i(Y)=U_i(Y_i+e)^{-1}U_i(Y)$. Then $V_i$ satisfies
\[
\begin{cases}
\mathrm{div}(s^{1-2\sigma}\nabla V_i)=0,\quad &\mbox{in } \B_2^+,\\
-\dlim_{s\to 0}s^{1-2\sigma}\pa_sV_i(y,s)=c(n,\sigma)KU_i(Y_i+e)^{p_i-1}V_i^{p_i},\quad &\mbox{for }y\in B_2^+.
\end{cases}
\]
Note that $U_i(Y_i+e)\to 0$ by Lemma \ref{lem4.3}, and for any $r>0$
\be\label{uniform bound of Vi}
V_i(Y)\leq C(n,\sigma,A_1,r), \quad \mbox{for all } r<|y-y_i|\leq 1
\ee
which follows from Lemma \ref{lem4.1}.
It follows that $\{V_i\}$ converges to some positive function $V$ in
$ C_{loc}^{\infty}(\B^+_{3/2})\cap C_{loc}^{\al}(\overline \B^+_{3/2}\setminus \{0\})$ for some $\al\in (0,1)$,
and $V$ satisfies
\[
\begin{cases}
\mathrm{div}(s^{1-2\sigma}\nabla V)=0,\quad &\mbox{in } \B_1^+\\
-\dlim_{s\to 0}s^{1-2\sigma}\pa_sV(y,s)=0\quad &\mbox{for }y\in B_1^+\setminus\{0\}.
\end{cases}
\]
Hence  $\dlim_{i\to\infty}r^{2\sigma/(p_i+1)}\bar v_i(r)=r^{n-2\sigma}\bar v(r)$, where $v(y)=V(y,0)$.
Since $r_i\to 0$ and $y_i\to0$ is an isolated simple blow up point of $\{u_i\}$, it follows from Lemma \ref{lem4.1} that $r^{(n-2\sigma)/2}\overline V(r)$
is \emph{almost decreasing} for all $0<r<\rho$, i.e., there exists a positive constant $C$ (which comes from Harnack inequality in Lemma \ref{lem4.1}) such that for any $0<r_1\leq r_2<\rho$,
\[
r_1^{(n-2\sigma)/2}\overline V(r_1)\geq C r_2^{(n-2\sigma)/2}\overline V(r_2).
\]
 Therefore, $V$ has to have a singularity at $Y=0$.
Lemma \ref{thm3.2} implies
\be\label{4.16}
V(Y)=A|Y|^{2\sigma-n}+H(Y),
\ee
where $A>0$ is a constant and $H$ is as in Lemma \ref{thm3.2}.

We first establish
the inequality in Proposition \ref{prop4.3} for
$|Y-Y_i|=1$.  Namely, we prove that
\be\label{4.17}
U_i(Y_i+e)\leq CU_i^{-1}(Y_i)
\ee

Suppose that \eqref{4.17} does not hold, then along a subsequence we have
\be\label{4.18}
\dlim_{i\to\infty}U_i(Y_i+e)U_i(Y_i)=\infty.
\ee
By integration by parts (using $\Omega_{\va}$ and sending $\va\to 0$, as in the proof of Proposition \ref{thm3.3}), we obtain
\be\label{4.19}
\begin{split}
0&=-\int_{\B_1^+}\mathrm{div}(s^{1-2\sigma}\nabla V_i)\\&
=\int_{\pa''\B_1^+}s^{1-2\sigma}\frac{\pa V_i}{\pa \nu}+c(n,\sigma)U_i(Y_i+e)^{-1}\int_{\pa'\B_1^+}KU_i^{p_i}.
\end{split}
\ee
By Lemma \ref{lem4.4} and similar computation in \eqref{4.13} and \eqref{4.14}, we see that
\[
\int_{\pa'\B_1^+}KU_i^{p_i}\leq CU_i(Y_i)^{-1}.
\]
Due to \eqref{4.18},
\[
\dlim_{i\to \infty}U_i(Y_i+e)^{-1}\int_{\pa'\B_1^+}KU_i^{p_i}=0.
\] A direct computation yields with $\eqref{uniform bound of Vi}$ (again using $\Omega_{\va}$ and sending $\va\to 0$)
\[\begin{split}
\dlim_{i\to \infty}\int_{\pa''\B_1^+}s^{1-2\sigma}\frac{\pa V_i}{\pa \nu}&=
\dlim_{i\to \infty}\int_{\pa''\B_1^+}s^{1-2\sigma}\frac{\pa}{\pa \nu}(A|Y|^{2\sigma -n}+H(Y))\\&
=A(2\sigma-n)\int_{\pa''\B_1^+}s^{1-2\sigma}<0,
\end{split}
\]
which contradicts to \eqref{4.19}. Thus we proved \eqref{4.17}. By Lemma \ref{lem4.1}, we have established the inequality in Proposition \ref{prop4.3} for $\rho\leq |Y-Y_i|\leq 1$.

By a standard scaling argument, we can reduce the case of $r_i\leq|Y-Y_i|<\rho$ to $|Y-Y_i|=1$. We refer to \cite{Li95} (page 340) for details.
\end{proof}

Proposition \ref{prop4.2} and \ref{prop4.3} give a clear picture of $u_i$ near the isolated simple blow up point.
By the estimates there, it is easy to see the following result.

\begin{lem}\label{lem4.5} We have
\[
\begin{split}
\int_{|y-y_i|\leq r_i}&|y-y_i|^{s}u_i(y)^{p_i+1}\\
&=\begin{cases}
O(u_i(y_i)^{-2s/(n-2\sigma)}),\quad& -n< s<n,\\
O(u_i(y_i)^{-2n/(n-2\sigma)}\log u_i(y_i)),\quad& s=n,\\
o(u_i(y_i)^{-2n/(n-2\sigma)}),\quad & s>n,
\end{cases}
\end{split}
\]
and
\[
\begin{split}
\int_{ r_i<|y-y_i|\leq1}&|y-y_i|^{s}u_i(y)^{p_i+1}\\&
=\begin{cases}
o(u_i(y_i)^{-2s/(n-2\sigma)}),\quad& -n< s<n,\\
O(u_i(y_i)^{-2n/(n-2\sigma)}\log u_i(y_i)),\quad& s=n,\\
O(u_i(y_i)^{-2n/(n-2\sigma)}),\quad & s>n.
\end{cases}
\end{split}
\]
\end{lem}

\begin{proof}
The first estimate in the above Lemma follows from Proposition \ref{prop4.1}
and Lemma \ref{lem4.4}, and the second one follows from
Proposition \ref{prop4.3} and Lemma \ref{lem4.4}.
\end{proof}

For later application, we replace $K_i$ by $K_i(x)H_i(x)^{\tau_i}$ in \eqref{4.2} and consider
\be\label{4.20}
(-\Delta)^{\sigma}u_i(x)=c(n,\sigma)K_i(x)H_i(x)^{\tau_i}u^{p_i}_i(x),\quad \mbox{in } B_2,
\ee
where $\{H_i\}\in C^{1,1}(B_2)$ satisfies
\be\label{4.21}
A_4^{-1}\leq H_i(y)\leq A_4, \quad \mbox{for all }y\in B_2,\quad \text{and}\quad \|H_i\|_{C^{1,1}(B_2)}\leq A_5
\ee
 for some positive constants $A_4$ and $A_5$.

\begin{lem}\label{lem4.6}
Suppose that $\{K_i\}$ satisfies \eqref{4.1} and $(*)_{\beta}$ condition with $\beta<n$ for some positive constants $A_1,A_2$, $\{L(\beta,i)\}$,
and that $\{H_i\}$ satisfies \eqref{4.21} with $A_4, A_5$.
Let $u_i\in \D\cap C^2(B_2)$ and $u_i\geq 0$ in $\R^n$ be a solution of \eqref{4.20}. If $y_i\to 0$ is an
isolated simple blow up point of $\{u_i\}$ with \eqref{4.7} for some positive constant $A_3$,
then we have
\[\begin{split}
\tau_i\leq &Cu_i(y_i)^{-2}+C|\nabla K_i(y_i)|u_i(y_i)^{-2/(n-2\sigma)}\\&
\quad +C(L(\beta,i)+L(\beta,i)^{\beta-1})u_i(y_i)^{-2\beta/(n-2\sigma)},
\end{split}
\]
where $C>0$ depends only on $n,\sigma,A_1,A_2,A_3,A_4,A_5,\beta$ and $\rho$.
\end{lem}

\begin{proof}
Using Lemma \ref{lem4.4} and arguing the same as in the proof of Lemma \ref{lem4.4}, we have
\[
\begin{split}
\tau_i&\leq Cu_i(y_i)^{-2}+C\left|\int_{B_1(y_i)}\langle y-y_i,\nabla_y(K_iH_i^{\tau_i})\rangle u_i^{p_i+1}\right|\\&
\leq Cu_i(y_i)^{-2}+C\tau_i\left|\int_{B_1(y_i)}|y-y_i|u_i^{p_i+1}\right|\\&
\quad +C\left|\int_{B_1(y_i)}\langle y-y_i,\nabla K_i \rangle H_i^{\tau_i}u_i^{p_i+1}\right|.
\end{split}
\]
Making use of Lemma \ref{lem4.5}, we have
\[
\begin{split}
&\left|\int_{B_1(y_i)}\langle y-y_i,\nabla K_i \rangle H_i^{\tau_i}u_i^{p_i+1}\right|\\&
\leq C|\nabla K_i(y_i)|\int_{B_1(y_i)}|y-y_i|u_i^{p_i+1}\\&
\quad +C\int_{B_1(y_i)}|y-y_i||\nabla K_i(y)-\nabla K_i(y_i)|u_i^{p_i+1}\\&
\leq C|\nabla K_i(y_i)|u_i(y_i)^{-2/(n-2\sigma)}\\&
\quad  +C\int_{B_1(y_i)}|y-y_i||\nabla K_i(y)-\nabla K_i(y_i)|u_i^{p_i+1}.
\end{split}
\]
Recalling the definition of $(*)_\beta$, a directly computation yields
\begin{equation}\label{eq:lem4.5-(*)-1}
\begin{split}
&|\nabla K_i(y)-\nabla K_i(y_i)|\\&
\leq \Big\{\sum_{s=2}^{[\beta]}|\nabla ^s K_i(y_i)||y-y_i|^{s-1}+[\nabla ^{[\beta]}K_i]_{C^{\beta-[\beta]}(B_1(y_i))}|y-y_i|^{\beta-1}\Big\}\\&
\leq CL(\beta,i)\Big\{\sum_{s=2}^{[\beta]}|\nabla K_i(y_i)|^{(\beta-s)/(\beta-1)}|y-y_i|^{s-1}+|y-y_i|^{\beta-1}\Big\}.
\end{split}
\end{equation}
By Cauchy-Schwartz inequality, we have
\be\label{eq:lem4.5-(*)-2}
\begin{split}
&L(\beta,i)|\nabla K_i(y_i)|^{(\beta-s)/(\beta-1)}|y-y_i|^{s}\\&
\leq C(|\nabla K_i(y_i)||y-y_i|+(L(\beta,i)+L(\beta,i)^{\beta-1})|y-y_i|^{\beta}).
\end{split}
\ee
Hence, by Lemma \ref{lem4.5} we obtain
\be\label{eq:lem4.5-(*)-3}
\begin{split}
&\int_{B_1(y_i)}|y-y_i||\nabla K_i(y)-\nabla K_i(y_i)|u_i^{p_i+1}\\&
\leq   C|\nabla K_i(y_i)|u_i(y_i)^{-2/(n-2\sigma)}+C(L(\beta,i)+L(\beta,i)^{\beta-1})u_i(y_i)^{-2\beta/(n-2\sigma)}.
\end{split}
\ee
Lemma \ref{lem4.6} follows immediately.
\end{proof}

\begin{lem}\label{lem4.7} Under the hypotheses of Lemma \ref{lem4.6},
\[
|\nabla K_i(y_i)|\leq Cu_i(y_i)^{-2}+C(L(\beta,i)+L(\beta,i)^{\beta-1}) u_i(y_i)^{-2(\beta-1)/(n-2\sigma)},
\]
where $C>0$ depends only on $n,\sigma,A_1,A_2,A_3,A_4,A_5,\beta$ and $\rho$.
\end{lem}

\begin{proof}
Choose a cutoff function $\eta(Y)\in C^\infty_c(\B_{1/2})$ satisfying
\[
\eta(Y)=1,\quad |Y|\leq \frac14 \mbox{ and } \eta(Y)=0,\quad |Y|\geq \frac12.
\] Let $U_i(Y)$ be the extension of $u_i(y)$, namely,
\be\label{4.22}
\begin{cases}
\mathrm{div}(s^{1-2\sigma}\nabla U_i)=0,\quad &\mbox{in }\R^{n+1}_+\\
-\dlim_{s\to 0}s^{1-2\sigma}\pa_sU(y,s)=c_0K_i(y)H_i^{\tau_i}U_i^{p_i},\quad & y\in B_2.
\end{cases}
\ee
Multiplying \eqref{4.22} by $\eta(Y-Y_i)\pa y_jU_i(y,s)$, $j=1,\cdots,n$, and integrating by parts over $\B_1$,
we obtain
\[
\begin{split}
0&=\int_{\B^+_1}\mathrm{div}(s^{1-2\sigma}\nabla U_i)\eta \pa_{y_j}U_i\\&
=-\int_{\B_1^+}s^{1-2\sigma}\nabla U_i\nabla (\eta \pa_{y_j}U_i)+c_0\int_{\pa'\B_1^+(Y_i)}\eta K_iH_i^{\tau_i}\pa_{y_j}U_iU_i^{p_i}\\&
=\frac12 \int_{\B_{1/2}^+\setminus \B_{1/4}^+}s^{1-2\sigma}(|\nabla U_i|^2\pa_{y_j}\eta-2\nabla U_i\nabla \eta\pa_{y_j}U_i)\\&
\quad -\frac{c_0}{p_i+1}\int_{\pa'\B_1^+}\pa_{y_j}(K_iH_i^{\tau_i}\eta)U_i^{p_i+1}.
\end{split}
\]
By Proposition \ref{prop4.3}, we have
\[
U_i(Y)\leq CU_i(Y_i)^{-1},\quad \mbox{for all }1/2\geq  |Y|\geq 1/4
\]
and
\[
\int_{\B_{1/2}^+\setminus \B_{1/4}^+}s^{1-2\sigma}|\nabla U_i|^2\leq C U_i(Y_i)^{-2}.
\]
Therefore by Lemma \ref{lem4.5} we conclude that
\be\label{4.23}
\left|\int_{B_1}\pa_{y_j}K_i H_i^{\tau_i}u_i^{p_i+1} \right|\leq Cu_{i}(y_i)^{-2}+C\tau_i.
\ee
Hence
\[
\begin{split}
&\left|\partial_j K_i(y_i)\int_{B_{1}}H_i^{\tau_i}u_i^{p_i+1}\right|-Cu_{i}(y_i)^{-2}-C\tau_i\\
&\quad \leq \int_{B_1}|\partial_j K_i(y_i)-\partial_j K_i(y)|H_i^{\tau_i}u_i^{p_i+1}
\end{split}
\]
Summing over $j$, then making use of \eqref{eq:lem4.5-(*)-1}, \eqref{eq:lem4.5-(*)-2} and Lemma \ref{lem4.5}, we have
\[
\begin{split}
|\nabla K_i(y_i)|&\leq Cu_{i}(y_i)^{-2}+C\tau_i +\frac{1}{2}|\nabla K_i(y_i)|\\
&\quad+C(L(\beta,i)+L(\beta,i)^{\beta-1}) u_i(y_i)^{-2(\beta-1)/(n-2\sigma)}.
\end{split}
\]
Then Lemma \ref{lem4.7} follows from Lemma \ref{lem4.6}.
\end{proof}

\begin{lem}\label{lem4.8} Under the assumptions of Lemma \ref{lem4.6} we have
\[
\tau_i\leq Cu_i(y_i)^{-2}+C(L(\beta,i)+L(\beta,i)^{\beta-1})u_i(y_i)^{-2\beta/(n-2\sigma)}.
\]
\end{lem}

\begin{proof}
It follows immediately from Lemma \ref{lem4.6} and Lemma \ref{lem4.7}.
\end{proof}

\begin{cor}\label{cor4.1}
In addition to the assumptions of Lemma \ref{lem4.6}, we further assume that
one of the following two conditions holds:
(i)
 \[
\beta=n-2\sigma \mbox{ and }L(\beta,i)=o(1),
\]
and
(ii)\[
\beta>n-2\sigma \mbox{ and }L(\beta,i)=O(1).
\]
Then for any $0<\delta<1$ we have
\[
\lim_{i\to\infty} u_i(y_i)^{2}\int_{B_{\delta}(y_i)}(y-y_i)\cdot \nabla(K_iH_i^{\tau_i})u_i^{p_i+1}=0.
\]
\end{cor}

\begin{proof}
\[
\begin{split}
&\left|\int_{B_{\delta}(y_i)}(y-y_i)\cdot \nabla(K_iH_i^{\tau_i})u_i^{p_i+1}\right|\\
&\leq\left|\int_{B_{\delta}(y_i)}(y-y_i)\cdot \nabla K_i H_i^{\tau_i}u_i^{p_i+1}\right|+
\tau_i\left|\int_{B_{\delta}(y_i)}(y-y_i)\cdot \nabla H_iH_i^{\tau_i-1}K_i u_i^{p_i+1}\right|\\
&\leq C|\nabla K_i(y_i)|\int_{B_{\delta}(y_i)}|y-y_i|u_i^{p_i+1}\\
&\quad+C\int_{B_{\delta}(y_i)}|y-y_i||\nabla K_i(y)-\nabla K_i(y_i)|u_i^{p_i+1}+\tau_i\int_{B_{\delta}(y_i)}|y-y_i| u_i^{p_i+1}.\\
\end{split}
\]
The corollary follows immediately from Lemma \ref{lem4.7},  \eqref{eq:lem4.5-(*)-3} and  Lemma \ref{lem4.8}.
\end{proof}


\begin{prop}\label{prop5.1}
Suppose that $\{K_i\}$ satisfies \eqref{4.1} and $(*)_{n-2\sigma}$ condition for some positive constants $A_1,A_2$, $L$ independent of $i$,
and that $\{H_i\}$ satisfies \eqref{4.21} with $A_4, A_5$.
Let $u_i\in \D\cap C^2(B_2)$ be a solution of \eqref{4.20}. If $y_i\to 0$ is an
isolated blow up point of $\{u_i\}$ with \eqref{4.7} for some positive constant $A_3$, then $y_i\to 0$ is an
isolated simple blow up point.
\end{prop}

\begin{proof} Due to Proposition \ref{prop4.1}, $r^{2\sigma/(p_i-1)}\overline u_i(r)$ has precisely
one critical point in the interval $0<r<r_i$,
where $r_i=R_iu_i(y_i)^{-\frac{p_i-1}{2\sigma}}$ as before.
Suppose $y_i\to 0$ is not an isolated simple blow up point and let $\mu_i$ be the second critical point of $r^{2\sigma/(p_i-1)}\overline u_i(r)$.
Then we see that
\be\label{5.2}
\mu_i\geq r_i,\quad \dlim_{i\to \infty}\mu_i=0.
\ee

Without loss of generality, we assume that $y_i=0$. Set
\[
\phi_i(y)=\mu_i^{2\sigma/(p_i-1)}u_i(\mu_i y),\quad y\in \R^n.
\]
Clearly, $\phi_i$ satisfies
\[
\begin{split}
(-\Delta)^\sigma \phi_i(y)&=\tilde K_i(y)\tilde H_i^{\tau_i}(y)\phi_i^{p_i}(y),
\\
|y|^{2\sigma/(p_i-1)}\phi_i(y)&\leq A_3,\quad |y|<1/\mu_i,
\\
\lim_{i\to \infty}\phi_i(0)&=\infty,
\end{split}
\]
\[
r^{2\sigma/(p_i-1)}\overline \phi_i(r)\mbox{ has precisely one critical point in } 0<r<1,
\]
and
\[
\frac{\mathrm{d}}{\mathrm{d}r}\left\{ r^{2\sigma/(p_i-1)}\overline \phi_i(r)\right\}\Big|_{r=1}=0,
\]
where $\tilde K_i(y)=K_i(\mu_i y)$, $\tilde H_i(y)=H_i(\mu_i y)$ and $\overline \phi_i(r)=|\pa B_r|^{-1}\int_{\pa B_r}\phi_i$.

Therefore, $0$ is an isolated simple blow up point of $\phi_i$. Let $\Phi_i(Y)$ be the extension of $\phi_i(y)$ in the upper half space.
Then Lemma \ref{lem4.1}, Proposition \ref{prop4.3}, Lemma \ref{thm3.2} and elliptic equation theory together imply that
\[
\Phi_i(0)\Phi_i(Y)\to G(Y)=A|Y|^{2\sigma-n}+H(Y)
\quad \mbox{in } C^{\al}_{loc}(\overline{\R^{n+1}_+}\setminus \{0\})\cap C^2_{loc}(\R^{n+1}_+).
\]
and
\be\label{5.3}
\phi_i(0)\phi_i(y)\to G(y,0)=A|y|^{2\sigma-n}+H(y,0)
\quad \mbox{in } C^2_{loc}(\R^{n}\backslash\{0\})
\ee
as $i\to \infty$, where $A>0$, $H(Y)$ satisfies
\[
\begin{cases}
\mathrm{div}(s^{1-2\sigma}\nabla H)=0\quad &\mbox{in }\R^{n+1}_+\\
-\dlim_{s\to 0}s^{1-2\sigma}\pa_s H(y,s)=0\quad &\mbox{for } y\in \R^{n}.
\end{cases}
\]

Note that $G(Y)$ is nonnegative, we have $\liminf_{|Y|\to \infty}H(Y)\geq 0$. It follows from the weak maximum principle and the
Harnack inequality that $H(y)\equiv H\geq 0$ is a constant. Since
\[
\frac{\mathrm{d}}{\mathrm{d}r}\left\{ r^{2\sigma/(p_i-1)}\phi_i(0)\overline \phi_i(r)\right\}\Big|_{r=1}=
\phi_i(0)\frac{\mathrm{d}}{\mathrm{d}r}\left\{ r^{2\sigma/(p_i-1)}\overline \phi_i(r)\right\}\Big|_{r=1}=0,
\]
we have, by sending $i$ to $\infty$ and making use of \eqref{5.3}, that
\[A=H>0.\]

We are going to derive a contradiction to the Pohozaev identity  Proposition \ref{thm3.3}, by showing that for small positive $\delta$
\be\label{5.4}
\limsup_{i\to \infty}\Phi_i(0)^2\int_{\pa' \B_\delta^+}B'(Y,\Phi_i,\nabla \Phi_i,\delta,\sigma)\leq 0,
\ee
and
\be\label{5.5}
\limsup_{i\to \infty}\Phi_i(0)^2\int_{\pa'' \B_\delta^+}s^{1-2\sigma}B''(Y,\Phi_i,\nabla \Phi_i,\delta,\sigma)< 0.
\ee And thus Proposition \ref{prop5.1} will be established.

By Proposition \ref{thm2.4-1}, it is easy to verify \eqref{5.5} by that
\[
\begin{split}
&\limsup_{i\to \infty}\Phi_i(0)^2\int_{\pa'' \B_\delta^+}s^{1-2\sigma}B''(Y,\Phi_i,\nabla \Phi_i,\delta,\sigma)\\
&\quad=\int_{\pa'' \B_\delta^+}s^{1-2\sigma}B''(Y,G,\nabla G,\delta,\sigma)=-\frac{(n-2\sigma)^2}{2}A^2\int_{\pa''\B_1^+}t^{1-2\sigma}< 0,
\end{split}
\]
which shows \eqref{5.5}. On the other hand, via integration by parts, we have
\[
\begin{split}
&\int_{\pa' \B_\delta^+}B'(Y,\Phi_i,\nabla \Phi_i,\delta,\sigma)\\&
=\frac{n-2\sigma}{2} \int_{B_\delta}\tilde K_i\tilde H_i^{\tau_i}\phi_i^{p_i+1}
+\int_{B_\delta}\langle y,\nabla \phi_i\rangle \tilde K_i\tilde H_i^{\tau_i}\phi_i^{p_i}\\&
= \frac{n-2\sigma}{2}\int_{B_\delta}\tilde K_i\tilde H_i^{\tau_i}\phi_i^{p_i+1}-
\frac{n}{p_i+1} \int_{B_\delta}\tilde K_i\tilde H_i^{\tau_i}\phi_i^{p_i+1}\\&
\quad -\frac{1}{p_i+1} \int_{B_\delta}\langle y,\nabla(\tilde K_i\tilde H_i^{\tau_i})\rangle \phi_i^{p_i+1}
+\frac{\delta}{p_i+1} \int_{\pa B_\delta}\tilde K_i\tilde H_i^{\tau_i}\phi_i^{p_i+1}\\&
\leq -\frac{1}{p_i+1} \int_{B_\delta}\langle y,\nabla(\tilde K_i\tilde H_i^{\tau_i})\rangle \phi_i^{p_i+1}+C\phi_i(0)^{-(p_i+1)}.
\end{split}
\]
where Proposition \ref{prop4.3} is used in the last inequality. It is easy to see that $\{\tilde K_i\}$ satisfies $(*)_{n-2\sigma}$ with $L(\beta,i)=o(1)$.
Therefore, \eqref{5.4} follows from
Corollary \ref{cor4.1}.
\end{proof}

\begin{prop}\label{prop5.2}
Suppose the assumptions in Proposition \ref{prop5.1} except the $(*)_{n-2\sigma}$ condition for $K_i$. Then
\[
|\nabla K_i(y_i)|\to 0,\quad \mbox{as } i \to \infty.
\]
\end{prop}

\begin{proof}
Suppose that contrary that
\be\label{5.6}
|\nabla K_i(y_i)|\to d>0.
\ee
Without loss of generality, we assume $y_i=0$. There are two cases.

\medskip

\noindent \textit{Case 1.} $0$ is an isolated simple blow up point.
\medskip

In this case, we argue as in the proof of Lemma \ref{lem4.7} and obtain
\[
 \left|\int_{B_1}\nabla K_iH_i^{\tau_i}u_i^{p_i+1}\right|\leq Cu_i^{-2}(0)+C\tau_i.
\]
It follows from the mean value theorem,  Lemma \ref{lem4.4} and Lemma \ref{lem4.5} that
\[
 |\nabla K_i(0)|\leq C\int_{B_1}|\nabla K_i(y)-\nabla K_i(0)|H_i^{\tau_i}u_i^{p_i+1}+o(1)=o(1).
\]

\medskip

\noindent\textit{Case 2.} $0$ is not an isolated simple blow up point.

\medskip

In this case we argue as the proof of Proposition \ref{prop5.1}. The only difference is that we cannot derive \eqref{5.4} from Corollary \ref{cor4.1},
since $(*)_{n-2\sigma}$ condition for $K_i$ is not assumed. Instead, we will use the condition \eqref{5.6} to show \eqref{5.4}.

Let $\mu_i, \phi_i, \Phi_i$, $\tilde{K_i}$ and $\tilde{H_i}$  be as in the proof of Proposition \ref{prop5.1}.
The computation at the end of the proof of Proposition \ref{prop5.1} gives
\[
\begin{split}
&\int_{\pa' \B_\delta^+}B'(Y,\Phi_i,\nabla \Phi_i,\delta,\sigma)\\&
\leq -\frac{1}{p_i+1} \int_{B_\delta}\langle y,\nabla(\tilde K_i\tilde H_i^{\tau_i})\rangle \phi_i^{p_i+1}+C\phi_i(0)^{-(p_i+1)}.
\end{split}
\]

Now we estimate the integral term $\int_{B_\delta}\langle y,\nabla(\tilde K_i\tilde H_i^{\tau_i})\rangle \phi_i^{p_i+1}$.
Using Lemma \ref{lem4.4} and arguing the same as in the proof of Lemma \ref{lem4.4}, we have
\[
\begin{split}
 \tau_i&\leq C\phi_i(0)^{-2}+C\int_{B_\delta}|y||\nabla \tilde{K}_i(y)|H_i^{\tau_i} \phi_i^{p_i+1}\\&
\leq C\phi_i(0)^{-2}+C\mu_i\phi_i(0)^{-2/(n-2\sigma)}.
\end{split}
\]
By \eqref{4.23},
\[
 \left|\int_{B_\delta}\nabla \tilde K_i \tilde H_i^{\tau_i}\phi_i^{p_i+1} \right|\leq C\phi_{i}(y_i)^{-2}+C\tau_i.
\]
It follows that
\[
\begin{split}
 |\nabla \tilde K_i(0)|&\leq C\int_{B_\delta}|\nabla \tilde K_i(y)-\nabla \tilde K_i(0)|\phi_i^{p_i+1}+C\phi_i(0)^{-2}+C\tau_i\\
&\leq C\mu_i \phi_{i}(0)^{-2/(n-2\sigma)}+C\phi_i(0)^{-2}+C\tau_i.
\end{split}
\]
Since $|\nabla \tilde K_i(0)|=\mu_i |\nabla K_i(0)|\geq (d/2)\mu_i$, we have
\[
 \mu_i\leq C\phi_i(0)^{-2}+C\tau_i.
\]
It follows that
\[
 \tau_i\leq C\phi_i(0)^{-2}\quad \mbox{and}\quad \mu_i\leq C\phi_i(0)^{-2}.
\]
Therefore,
\[
 \left|\int_{B_\delta}\langle y,\nabla(\tilde K_i\tilde H_i^{\tau_i})\rangle \phi_i^{p_i+1}\right|\leq C\phi_i(0)^{-2-2/(n-2\sigma)}
\]
and \eqref{5.4} follows immediately.

\end{proof}

\medskip

\section{Estimates on the sphere and proofs of main theorems}
\label{Estimates on the sphere and proof of main theorems}

Consider
\be\label{6.1}
 P_\sigma(v)=c(n,\sigma)K v^{p},\quad \mbox{on } \mathbb{S}^n,
\ee where $p\in (1,\frac{n+2\sigma}{n-2\sigma}]$ and $K$ satisfies
\be\label{6.2}
A_1^{-1}\leq K\leq A_1, \quad \mbox{on } \mathbb{S}^n,
\ee
and
\be\label{6.3}
\|K\|_{C^{1,1}(\Sn)}\leq A_2.
\ee

\begin{prop}\label{prop6.1}
Let $v\in C^2(\Sn)$ be a positive solution to \eqref{6.1}. For any $0<\va<1$ and $R>1$, there exist large positive constants $C_1$, $C_2$
depending on
$n,\sigma, A_1,A_2,\va$ and $R$ such that, if
\[
\max_{\mathbb{S}^n} v\geq C_1,
\]
then $\frac{n+2\sigma}{n-2\sigma}-p<\va$, and there exists a finite set $\wp(v)\subset \Sn$ such that\\
(i). If $P\in \wp(v)$, then it is a local maximum of $v$ and
in the stereographic projection  coordinate system $\{y_1,\cdots,y_n\}$ with $P$ as the south pole,
\be\label{eq:isolated blowup-1}
\|v^{-1}(P)v(v^{-\frac{(p-1)}{2\sigma}}(P)y)-(1+k|y|^2)^{(2\sigma-n)/2}\|_{C^2(B_{2R})}\leq \va,
\ee
where $k=K(P)^{1/\sigma}/4$.\\
(ii). If $P_1,P_2$ belonging to $\wp(v)$ are two different points, then
\[
B_{Rv(P_1)^{-(p-1)/2\sigma}}(P_1)\cap B_{Rv(P_2)^{-(p-1)/2\sigma}}(P_2)=\emptyset.
\]
(iii). $v(P)\leq C_2\{\mbox{dist}(P,\wp(v))\}^{-2\sigma/(p-1)}$ for all $P\in \mathbb{S}^n$.
\end{prop}

\begin{proof}
Given Theorem \ref{thm3.1}, Remark \ref{subcritical nonexistence} and the proof of Proposition \ref{prop4.1}, the proof of Proposition \ref{prop6.1} is similar to that  of Proposition 4.1 in \cite{Li95}  and Lemma 3.1 in \cite{SZ}, and is omitted here. We refer to \cite{Li95} and \cite{SZ} for details.
\end{proof}

\begin{prop}\label{prop6.2} Assume the hypotheses in Proposition \ref{prop6.1}.
Suppose that there exists some constant
$d>0$ such that $K$ satisfies $(*)_{n-2\sigma}$ for some $L$ in $\om_d=\{P\in \mathbb{S}^n:|\nabla K(P)|<d\}$. Then, for
$\va>0$, $R>1$ and any solution $v$ of \eqref{6.1} with $\max_{\mathbb{S}^n}v>C_1$, we have
\[
 |P_1-P_2|\geq \delta^*>0,\quad \mbox{for any }P_1,P_2\in \wp(v)\mbox{ and } P_1\neq P_2,
\]
where $\delta^*$ depends only on $n,\sigma, \delta,\va, R,A_1,A_2,L_2,d$.
\end{prop}

\begin{proof}

Suppose the contrary, then there exists sequences of $\{p_i\}$ and $\{K_i\}$ satisfying the above assumptions,
and a sequence of corresponding solutions $\{v_i\}$ such that
\be\label{6.7}
\lim_{i\to \infty}|P_{1i}-P_{2i}|=0,
\ee
where $P_{1i},P_{2i}\in \wp(v_i)$, and $|P_{1i}- P_{2i}|=\dmin_{\substack{P_1,P_2\in \wp(v_i)\\ P_1\neq P_2}}|P_1-P_2|$.

Since $B_{Rv_i(P_{1i})^{-(p_i-1)/2\sigma}}(P_{1i})$ and $B_{Rv_i(P_{2i})^{-(p_i-1)/2\sigma}}(P_{2i})$ have to be disjoint,
we have, because of \eqref{6.7}, that
$v_i(P_{1i})\to \infty$ and $v_i(P_{2i})\to \infty$. Therefore, we can pass to a subsequence (still denoted as $v_i$)
with $R_i\to \infty$, $\va_i\to 0$
as in Proposition \ref{prop4.1} ($\va_i$ depends on $R_i$ and can be chosen as small as we need in the following arguments) such that,
for $y$ being the stereographic projection coordinate with south pole at $P_{ji},~j=1,2$, we have
\be\label{6.8}
 \|m_i^{-1}v_i(m_i^{-(p_i-1)/2\sigma}y)-(1+k_{ji}|y|^2)^{(2\sigma-n)/2}\|_{C^2(B_{2R_i}(0))}\leq \va_i,
\ee
where $m_i=v_i(0)$, $k_{ji}=K_i(q_{ji})^{1/\sigma}$, $j=1,2; i=1,2,\cdots$

In the stereographic coordinates with $P_{1i}$ being the south pole,
the equation \eqref{6.1} is transformed into
\be\label{6.6}
(-\Delta)^{\sigma}u_i(y)=c(n,\sigma)K_i(y)H_i^{\tau_i}(y)u_i^{p_i}(y), \quad y\in \mathbb{R}^n,
\ee
where
\be\label{eq:projection}
\begin{split}
 u_i(y)&=\left(\frac{2}{1+|y|^2}\right)^{(n-2\sigma)/2}v_i(F(y)),\\
 H_i(y)&=\left(\frac{2}{1+|y|^2}\right)^{(n-2\sigma)/2},
 \end{split}
\ee
and $F$ is the inverse of the stereographic projection. Let us still use $P_{2i}\in \R^n$ to denote the stereographic coordinates of $P_{2i}\in \mathbb{S}^n$ and
set $\vartheta_i=|P_{2i}|\to 0$. For simplicity, we
assume $P_{2i}$ is a local maximum point of $u_i$. Since we can always reselect a sequence of points as in the proof of Proposition \ref{prop6.1} to substitute for $P_{2i}$.

From (ii) in Proposition \ref{prop6.1}, there exists some constant $C$ depending only on $n,\sigma,$ such that
\be\label{6.9}
 \vartheta_i>\frac{1}{C}\max\{R_iu_i(0)^{-(p_i-1)/2\sigma}, R_iu_i(P_{2i})^{-(p_i-1)/2\sigma}\}.
\ee
Set
\[
 w_i(y)=\vartheta_i^{2\sigma/(p_i-1)}u_i(\vartheta_i y),\quad \mbox{in } \R^n.
\]
It is easy to see that $w_i$ which is positive in $\R^n$, satisfies
\be\label{6.10}
(-\Delta)^{\sigma}w_i(y)=c(n,\sigma)\tilde K_i(y)\tilde H_i^{\tau_i}(y)w_i(y)^{p_i},\quad \mbox{in }\R^n
\ee
and
\[
 w_i(y)\in C^2(\R^n),\quad \liminf_{|y|\to \infty}w_i(y)<\infty,
\]
where $\tilde K_i(y)=K_i(\vartheta_i y)$, $\tilde H_i(y)=H_i(\vartheta_i y)$.

By Proposition \ref{prop6.1}, $u_i$ satisfies
\[
\begin{split}
 u_i(y)&\leq C_2|y|^{-2\sigma/(p_i-1)}\quad \mbox{for all }|y|\leq \vartheta_i/2 \\
 u_i(y)&\leq C_2|y-P_{2i}|^{-2\sigma/(p_i-1)}\quad \mbox{for all }|y-P_{2i}|\leq \vartheta_i/2.
\end{split}
\]
In view of \eqref{6.9}, we therefore have
\[
\begin{split}
 \dlim_{i\to \infty}w_i(0)=\infty,\quad &\dlim_{i\to \infty}w_i(|P_{2i}|^{-1}P_{2i})=\infty\\
|y|^{2\sigma/(p_i-1)}w_i(y)\leq C_2,\quad &|y|\leq 1/2,\\
|y- |P_{2i}|^{-1}P_{2i}|^{2\sigma/(p_i-1)}w_i(y)\leq C_2,\quad &|y-|P_{2i}|^{-1}P_{2i}|\leq 1/2.
\end{split}
\]
After passing a subsequence, if necessary,
there exists a point $\overline P\in \R^n$ with $|\overline P|=1$ such that
$|P_{2i}|^{-1}P_{2i}\to \overline P$ as $i\to \infty$. Hence $0$ and $\overline P$ are both isolated blow up points of
$w_i$.

If $|\nabla K_i(0)|\leq d/2$, then $0$ is an isolated simple blow up point of $w_i$ because of the $(*)_{n-2\sigma}$ condition and Proposition \ref{prop5.1}.
If $|\nabla K_i(0)|\geq d/2$, arguing as in the proof of Proposition \ref{prop5.2} we can conclude that $0$ is an isolated simple blow up point of $w_i$.
Similarly, $\overline P$ is also an isolated simple blow up point of $w_i$.

By Proposition \ref{prop4.3},
\[
 w_i(0)w_i(y)\leq C_\va, \mbox{for all } \va\leq |y|\leq 1/2,
\]
where $C_\va$ is independent of $i$. Let $W_i$ be the extension of $w_i$.
Due to Proposition \ref{prop6.1}, Harnack inequality Lemma \ref{lem4.1}, and the choice of $P_{1i}, P_{2i}$,
there exists an at most countable set $\wp\subset \R^n$ such that
\[
 \inf\{|x-y|:x,y\in \wp,~ x\neq y\}\geq 1,
\]
and
\[
\begin{split}
  \dlim_{i\to \infty}W_i(0)W_i(Y)=G(Y),\quad&\mbox{in } C^0_{loc}(\overline{\R^{n+1}_+}\setminus \wp)\\
 G(Y)>0,\quad &Y\in \overline{\R^{n+1}_+}\setminus \wp.
\end{split}
\]
Let $\wp_1\subset \wp$ contain those points near which $G$ is singular. Clearly, $0,\overline P\in \wp_1$. Since $p_i>1$,
it follows from \eqref{6.10} that
\[
 \begin{cases}
  \mathrm{div}(s^{1-2\sigma}\nabla G)=0,\quad &\mbox{in }\R^{n+1}_+,\\
  -\dlim_{s\to 0}s^{1-2\sigma} \pa_s G(y,s)=0,\quad &\mbox{for all }y\in \R^n\setminus \wp_1.
 \end{cases}
\]
By Lemma \ref{thm3.2} and maximum principle, there exist positive constants $N_1,N_2$ and some nonnegative function $H$ satisfying
\[
\begin{cases}
 \mathrm{div}(s^{1-2\sigma}\nabla H)=0,\quad &\mbox{in }\R^{n+1}_+,\\
  -\dlim_{s\to 0}s^{1-2\sigma} \pa_s H(y,s)=0,\quad &\mbox{for all }y\in \R^n\setminus \{\wp_1\setminus\{0,\overline P\}\}
\end{cases}
\]
such that
\[
 G(Y)=N_1|Y|^{2\sigma-n}+N_2|Y-\overline P|^{2\sigma-n}+H(Y),\quad Y\in \overline{\R^{n+1}_+}\setminus \{\wp_1\}.
\]
Applying Proposition \ref{thm2.4-1} to $H$, it is not difficult to verify \eqref{5.5} with $\Phi_i$ replaced by $W_i$.
On the other hand, we can establish \eqref{5.4} with $\Phi_i$ replaced by $W_i$
if $|\nabla K_i(0)|\leq d/2$, because $(*)_{n-2\sigma}$ condition with $L=o(1)$ holds for $\tilde K_i$
and thus Corollary \ref{cor4.1} holds. If $|\nabla K_i(0)|\geq d/2$, we can apply the
argument in the proof of Proposition \ref{prop5.2} to conclude that
$\vartheta_i, \tau_i\leq w_i(0)^{-2}$, and hence \eqref{5.4} also holds for $W_i$.

Proposition \ref{prop6.2} is established.
\end{proof}

Consider
\be\label{6.11}
\begin{split}
P_\sigma(v)=c(n,\sigma)K_i v_i^{p_i}\quad &\mbox{on } \mathbb{S}^n,\\
v_i>0,\quad &\mbox{on } \mathbb{S}^n,\\
p_i=\frac{n+2\sigma}{n-2\sigma}-\tau_i, \quad &\tau_i\geq 0,\tau_i\to 0.
\end{split}
\ee

\begin{thm}\label{thm6.1}
 Suppose $K_i$ satisfies the assumption of $K$ in Proposition \ref{prop6.2}. Let $v_i$ be solutions of
\eqref{6.11}, we have
\be\label{6.12}
\|v_i\|_{H^\sigma(\Sn)}\leq C,
\ee
where $C>0$ depends only on $n,\sigma,A_1,A_2,L,d$. Furthermore,
after passing to a subsequence, either $\{v_i\}$ stays bounded in $L^\infty(\mathbb{S}^n)$
or $\{v_i\}$ has only isolated simple blow up points and the distance between any two blow up points is bounded blow by some
positive constant depending only on $n,\sigma,A_1,A_2,L,d$.
\end{thm}

\begin{proof}
The theorem follows immediately from Proposition \ref{prop6.2}, Proposition \ref{prop5.2}, Proposition \ref{prop5.1}, Proposition \ref{prop4.1}
and Lemma \ref{lem4.5}.
\end{proof}

\begin{proof}[Proof of Theorem \ref{main thm B}] It follows immediately from Theorem \ref{thm6.1}.
\end{proof}

In the next theorem, we impose a stronger condition on $K_i$ such that $\{u_i\}$ has at most one blow up point.

\begin{thm}\label{thm6.2}
Suppose the assumptions in Theorem \ref{thm6.1}.
Suppose further that $\{K_i\}$ satisfies $(*)_{n-2\sigma}$
condition for some sequences $L(n-2\sigma,i)=o(1)$
in $\om_{d,i}=\{q\in \mathbb{S}^n: |\nabla_{g_0}K_i|<d\}$ or $\{K_i\}$ satisfies $(*)_{\beta}$
condition with $\beta>n-2\sigma$ in
$\om_{d,i}$. Then, after passing to a subsequence, either $\{v_i\}$ stays bounded in $L^\infty(\mathbb{S}^n)$
or $\{v_i\}$ has precisely one isolated simple blow up point.
\end{thm}

\begin{proof}
The strategy is the same as the proof of Proposition \ref{prop6.2}. We assume there
are two isolated blow up points. After some transformation,
we can assume that they are in the same half sphere.
The condition of $\{K_i\}$ guarantees that Corollary \ref{lem4.1} holds for $u_i$, where $u_i$ is as in \eqref{eq:projection}. Hence \eqref{5.4} holds for $U_i$, which is the extension of $u_i$.
Meanwhile \eqref{5.5} for $U_i$ is also valid, since the distance between these blow up points is uniformly lower bounded which is due to Proposition \ref{prop6.2}.
\end{proof}

\begin{proof}[Proof of Theorem \ref{main thm C}] By Theorem \ref{thm6.2}, we only
need to show the latter case of theorem.  After passing a subsequence, $\xi_i\to \overline \xi$ is the only isolated simple blow up point of
$v_i$. For simplicity, assume that $\xi_i$ is identical to the south pole and $K(\xi_i)=1$. Let $F:\R^n\to \Sn$ be the inverse of stereographic projection defined
at the beginning of the paper.
Define, for any $\lda >0$,
\[
 \psi_{\lda}: x\mapsto \lda x, \quad \forall x\in \R^n.
\]
Set $\varphi_i=F\circ \psi_{\lda_i}\circ F^{-1}$ with $\lda_i=v_i(\xi_i)^{-\frac{2}{n-2\sigma}}$. Then
$T_{\varphi_i}v_i$ satisfies
\[
P_\sigma(T_{\varphi_i}v_i)=c(n,\sigma)K\circ\varphi_i T_{\varphi_i}v_i^{\frac{n+2\sigma}{n-2\sigma}},\quad \text{on } \Sn.
\]
Let
\[
u_i(x)=\Big(\frac{2}{1+|x|^2}\Big)^{\frac{n-2\sigma}{2}}v_i\circ F(x), \quad x\in \R^n
\]
and
\[
\tilde u_i(x)=\Big(\frac{2}{1+|x|^2}\Big)^{\frac{n-2\sigma}{2}}T_{\varphi_i}v_i\circ F(x), \quad x\in \R^n.
\]
Note that
\[
|\det \ud \varphi_i(F(x))|^{\frac{n-2\sigma}{2n}}=\left(\Big(\frac{2}{1+|\lda_ix|^2}\Big)^n
\lda_i^{n}\Big(\frac{2}{1+|x|^2}\Big)^{-n}\right)^{\frac{n-2\sigma}{2n}}.
\]
Hence, $\tilde u_i(x)=\lda^{\frac{n-2\sigma}{2}}u_{i}(\lda_i x)$ for any $x\in \R^n$ and $0<u_i\leq 2^{\frac{n-2\sigma}{2}}$. Arguing as
before, we see that
\[
\tilde u_i(x)\to \Big(\frac{2}{1+|x|^{2}}\Big)^{\frac{n-2\sigma}{2}}, \quad \text{in }C^2_{loc}(\R^n).
\]
Therefore, $v_i\to 1$ in $C^2_{loc}(\Sn\setminus \{N\})$, where $N$ is the north pole of $\Sn$.

Since $T_{\varphi_i} v_i$ is uniformly bounded near the north pole, it follows from H\"older estimates that there exists a constant
$\al\in (0,1)$ such that $T_{\varphi_i} v_i \to f$ in $C^\al(B_\delta(N))$ for small constant $\delta>0$ and some function $f\in C^\al(B_\delta(N))$.
It is clear that $f=1$. Therefore, we complete the proof.
\end{proof}

\begin{thm}\label{thm6.3} Let $v_i$ be positive solutions of \eqref{6.11}.
Suppose that $\{K_i\}\subset C^\infty(\Sn)$ satisfies \eqref{6.3},
and for some point $P_0\in \Sn$, $\va_0>0$, $A_1>0$ independent of $i$ and $1<\beta<n$, that
\[
\{K_i\} \text{ is bounded in } C^{[\beta],\beta-[\beta]}(B_{\va_0}(q_0)),\quad \quad K_i(P_0)\geq A_1
\]
and
\[
K_i(y)=K_i(0)+Q_i^{(\beta)}(y)+R_i(y), \quad |y|\leq \va_0,
\]
where $y$ is the stereographic projection coordinate with $P_0$ as the south pole,
$Q_i^{(\beta)}(y)$ satisfies $Q_i^{(\beta)}(\lda y)=\lda^{\beta} Q_i^{(\beta)}(y)$, $\forall\lda>0,\ y\in\R^n$,  and $R_i(y)$ satisfies
\[\sum_{s=0}^{[\beta]} |\nabla^s R_i(y)||y|^{-\beta+s}\to 0\] uniformly for $i$ as $y\to 0$.

Suppose also that $Q_i^{(\beta)}\to Q^{(\beta)}$ in $C^1(\mathbb{S}^{n-1})$ and for some positive constant $A_6$ that
\be\label{6.13}
A_6|y|^{\beta-1}\leq |\nabla Q^{(\beta)}(y)|,\quad |y|\leq \va_0,
\ee
and
\be\label{6.14}
\left(
\begin{array}{l}
 \int_{\R^n}\nabla Q^{(\beta)}(y+y_0)(1+|y|^2)^{-n}\,\ud y\\[2mm]
\int_{\R^n}Q^{(\beta)}(y+y_0)(1+|y|^2)^{-n}\,\ud y
\end{array} \right)\neq 0, \quad \forall\  y_0\in \R^n.
\ee
If $P_0$ is an isolated simple blow up point of $v_i$, then
$v_i$ has to have at least another blow up point.
\end{thm}

\begin{proof}
Suppose the contrary, $P_0$ is the only blow up point of $v_i$.

We make a stereographic projection with $P_0$ being the south pole to the equatorial plane of $\Sn$, with its inverse $\pi$.
Then the Eq. \eqref{6.11} is transformed to
\be\label{6.15}
(-\Delta)^\sigma u_i=c(n,\sigma)K_i(y)u_i^{\frac{n+2\sigma}{n-2\sigma}},\quad \mbox{in } \R^n,
\ee
with
\[
u_i(y)=\left(\frac{2}{1+|y|^2}\right)^{(n-2\sigma)/2}v_i(\pi(y)).
\]
Let $y_i\to 0$ be the local maximum point of $u_i$. It follows from Lemma \ref{lem4.7} that
\[
|\nabla K_i(y_i)|=O(u_i(y_i)^{-2}+u_i(y_i)^{-2(\beta-1)/(n-2\sigma)}).
\]
First we establish
\be\label{6.16}
|y_i|=O(u_i(y_i)^{-2/(n-2\sigma)}).
\ee
Since we have assumed that $v_i$ has no other blow up point other than $P_0$, it follows from Proposition \ref{prop4.3} and Harnack inequality that for $|y|\ge\va>0,\ u_i(y)\le C(\va)|y|^{2\sigma-n}u_i(y_i)^{-1}$.

By Proposition \ref{K-W prop} we have
\be \label{eq:equal to 0-1}
 \int_{\R^n}\nabla K_iu_i^{\frac{2n}{n-2\sigma}}=0.
\ee
It follows  that for $\va>0$ small we have
\[
\left|\int_{B_{\va}}\nabla K_i(y+y_i)u_i(y+y_i)^{\frac{2n}{n-2\sigma}}\right|\le C(\va)u_i(y_i)^{-2n/(n-2\sigma)}.
\]
Using our hypotheses on $\nabla Q^{(\beta)}$ and $R_i$
we have
\[
\left|\int_{B_{\va}}(1+o_{\va}(1))\nabla Q_i^{(\beta)}(y+y_i)u_i(y+y_i)^{\frac{2n}{n-2\sigma}}\right|\le C(\va)u_i(y_i)^{-2n/(n-2\sigma)}.
\]
Multiplying the above by $m_i^{(2/(n-2\sigma))(\beta-1)}$, where $m_i=u_i(y_i)$, we have
\[
\begin{split}
&\left|\int_{B_{\va}}(1+o_{\va}(1))\nabla Q_i^{(\beta)}(m_i^{2/(n-2\sigma)}y+\tilde y_i)u_i(y+y_i)^{\frac{2n}{n-2\sigma}}\right|\\
&\le C(\va)u_i(y_i)^{(2/(n-2\sigma))(\beta-1-n)}
\end{split}
\]
where $\tilde y_i=m_i^{2/(n-2\sigma)}y_i$.
Suppose \eqref{6.16} is false, namely, $\tilde y_i\to+\infty$ along a subsequence. Then it follows from Proposition \ref{prop4.1} (we may choose $R_i\leq |\tilde y_i|/4$) that
\[
\begin{split}
&\left|\int_{|y|\le R_i m_i^{-2/(n-2\sigma)}}(1+o_{\va}(1))\nabla Q_i^{(\beta)}(m_i^{2/(n-2\sigma)}y+\tilde y_i)u_i(y+y_i)^{\frac{2n}{n-2\sigma}}\right|\\
&=\left|\int_{|z|\le R_i}(1+o_{\va}(1))\nabla Q_i^{(\beta)}(z+\tilde y_i)\big(m_i^{-1}u_i(m_i^{-2/(n-2\sigma)}z+y_i)\big)^{\frac{2n}{n-2\sigma}}\right|\sim |\tilde y_i|^{\beta-1}.
\end{split}
\]
On the hand, it follows from Lemma \ref{lem4.5} that
\[
\begin{split}
&\left|\int_{R_i m_i^{-2/(n-2\sigma)}\le |y|\le \va}(1+o_{\va}(1))\nabla Q_i^{(\beta)}(m_i^{2/(n-2\sigma)}y+\tilde y_i)u_i(y+y_i)^{\frac{2n}{n-2\sigma}}\right|\\
&\le C\left|\int_{R_i m_i^{-2/(n-2\sigma)}\le |y|\le \va}\big(|m_i^{2/(n-2\sigma)}y|^{\beta-1}+|\tilde y_i|^{\beta-1}\big)u_i(y+y_i)^{\frac{2n}{n-2\sigma}}\right|\\
&\leq o(1)|\tilde y_i|^{\beta-1}.
\end{split}
\]
It follows that
\[
|\tilde y_i|^{\beta-1}\leq C(\va)m_i^{(2/(n-2\sigma))(\beta-1-n)},
\]
which implies that
\[
|y_i|\leq C(\va)m_i^{-(2/(n-2\sigma))(n/(\beta-1))}=o(m_i^{-2/(n-2\sigma)}).
\]
This contradicts to that $\tilde y_i\to\infty$. Thus \eqref{6.16} holds.

We are going to find some $y_0$ such that \eqref{6.14} fails.

It follows from Kazdan-Warner condition Proposition \ref{K-W prop} that
\be\label{6.17}
\int_{\R^n}\langle y,\nabla K_i(y+y_i)\rangle u_i(y+y_i)^{2n/(n-2\sigma)}=0.
\ee
Since $P_0$ is an isolated simple blow up point and the only blow up point of $v_i$, we have for any $\va>0$,
\[
\left|\int_{B_\va}\langle y,\nabla K_i(y+y_i)\rangle u_i(y+y_i)^{2n/(n-\sigma)}\right|\leq C(\va)u_i(y_i)^{-2n/(n-2\sigma)}.
\]
It follows from Lemma \eqref{lem4.5} and expression of $K_i$ that
\[
\begin{split}
&\left|\int_{B_\va}\langle y,\nabla Q^{(\beta)}_i(y+y_i)\rangle u_i(y+y_i)^{2n/(n-2\sigma)}\right|\\&
\leq C(\va)u_i(y_i)^{-2n/(n-2\sigma)}\\&
\quad +o_\va(1)\int_{B_\va}|y||y+y_i|^{\beta-1}u_i(y+y_i)^{-2n/(n-2\sigma)}\\&
\leq C(\va)u_i(y_i)^{-2n/(n-2\sigma)}\\&
\quad +o_\va(1)\int_{B_\va}(|y|^\beta+|y||y_i|^{\beta-1})u_i(y+y_i)^{-2n/(n-2\sigma)}\\&
\leq C(\va)u_i(y_i)^{-2n/(n-2\sigma)}+o_\va(1)u_i(y_i)^{-2\beta/(n-2\sigma)},
\end{split}
\]
where we used \eqref{6.16} in the last inequality.

Multiplying the above by $u_i(y_i)^{2\beta/(n-2\sigma)}$, due to $\beta<n$ we obtain
\be\label{6.18}
\lim_{i\to \infty}u_i(y_i)^{2\beta/(n-2\sigma)}\left|\int_{B_\va}\langle y,\nabla Q^{(\beta)}_i(y+y_i)\rangle u_i(y+y_i)^{2n/(n-2\sigma)}\right|
=o_\va(1).
\ee
Let $R_i\to \infty$ as $i \to \infty$. We assume that
$r_i:=R_i u_i(y_i)^{-\frac{2}{n-2\sigma}}\to 0$ as we did in Proposition \ref{prop4.1}.
By Lemma \ref{lem4.5}, we have
\be\label{6.19}
\begin{split}
&u_i(y_i)^{2\beta/(n-2\sigma)}\left|\int_{r_i\leq |y|\leq\va}
\langle y,\nabla Q^{(\beta)}_i(y+y_i)\rangle u_i(y+y_i)^{2n/(n-2\sigma)}\right|\\&
\leq \lim_{i\to \infty}u_i(y_i)^{2\beta/(n-2\sigma)}\left|\int_{r_i\leq |y|\leq\va}
(|y|^\beta+|y||y_i|^{\beta-1}) u_i(y+y_i)^{2n/(n-2\sigma)}\right|\to 0
\end{split}
\ee
as $i\to \infty$.
Combining \eqref{6.18} and \eqref{6.19}, we conclude that
\[
\lim_{i\to \infty}u_i(y_i)^{2\beta/(n-2\sigma)}\left|\int_{B_{r_i}}\langle y,\nabla Q^{(\beta)}_i(y+y_i)\rangle u_i(y+y_i)^{2n/(n-2\sigma)}\right|
=o_\va(1).
\]
It follows from changing variable $z=u_i(y_i)^{\frac{2}{n-2\sigma}}y$, applying Proposition \ref{prop4.1} and then letting $\va\to0 $ that
\be\label{6.20}
\int_{\R^n}\langle z,\nabla Q^{(\beta)}(z+z_0)\rangle (1+k|z|^2)^{-n}=0,
\ee
where $z_0=\lim_{i\to \infty}u_i(y_i)^{2/(n-2\sigma)}y_i$ and $k=\lim_{i\to \infty}K_i(y_i)^{1/\sigma}$.

On the other hand, from \eqref{eq:equal to 0-1}
\be\label{6.21}
\int_{\R^n}\nabla K_i(y+y_i)u_i(y+y_i)^{2n/(n-2\sigma)}=0.
\ee
Arguing as above, we will have
\be\label{6.22}
\int_{\R^n}\nabla Q^{(\beta)}(z+z_0) (1+k|z|^2)^{-n}=0.
\ee
It follows from \eqref{6.20} and \eqref{6.22} that
\be\label{6.23}
\begin{split}
\int_{\R^n}& Q^{(\beta)}(z+z_0) (1+k|z|^2)^{-n}\,\ud z\\&
= \beta^{-1}\int_{\R^n}\langle z+z_0,\nabla Q^{(\beta)}(z+z_0)\rangle (1+k|z|^2)^{-n}\,\ud z\\&
=0.
\end{split}
\ee

Therefore, \eqref{6.14} does not hold for $y_0=\sqrt{k}z_0$.
\end{proof}

\begin{thm}\label{thm:maximum bound}
Let $\sigma \in (0,1)$ and $n\geq 3$. Suppose that $K\in C^{1,1}(\Sn)$, for some constant $A_1>0$,
\[
 1/A_1\leq K_i(\xi)\leq A_1\quad  \mbox{for all }\xi\in \Sn.
\]
Suppose also that for any critical point $\xi_0$ of $K$,
under the stereographic projection coordinate system $\{y_1, \cdots, y_n\}$ with $\xi_0$ as south pole, there exist
some small neighborhood $\mathscr{O}$ of $0$, a positive constant $L$, and $\beta=\beta(\xi_0)\in(n-2\sigma, n)$ such that
\[
 \|\nabla^{[\beta]}K\|_{C^{\beta-[\beta]}(\mathscr{O})}\leq L
\]
and
\[
 K(y)=K(0)+Q_{(\xi_0)}^{(\beta)}(y)+R_{(\xi_0)}(y)\quad \mbox{in }\mathscr{O},
\]
where $Q_{\xi_0}^{(\beta)}(y)\in C^{[\beta]-1,1}(\mathbb{S}^{n-1})$ satisfies
$Q_{\xi_0}^{(\beta)}(\lda y)=\lda^{\beta}Q_{\xi_0}^{(\beta)}(y)$, $\forall \lda >0$, $y\in \R^n$, and for some positive constant $A_6$
\[
A_6|y|^{\beta-1}\leq |\nabla Q^{(\beta)}(y)|,\quad y\in \mathscr{O},
\]
and
\[
\left(
\begin{array}{l}
 \int_{\R^n}\nabla Q^{(\beta)}(y+y_0)(1+|y|^2)^{-n}\,\ud y\\[2mm]
\int_{\R^n}Q^{(\beta)}(y+y_0)(1+|y|^2)^{-n}\,\ud y
\end{array} \right)\neq 0, \quad \forall\  y_0\in \R^n,
\]
and $R_{\xi_0}(y)\in C^{[\beta]-1,1}(\mathscr{O})$ satisfies $\lim_{y\to 0}\sum_{s=0}^{[\beta]}|\nabla^sR|{\xi_0}(y)|y|^{-\beta+s}=0$.

Then there exists a positive constant $C\geq 1$ depending on $n,\sigma, K$ such that for any solution $v$ of \eqref{main equ}
\[
 1/C\leq v\leq C,\quad \mbox{on }\Sn.
\]
\end{thm}

\begin{proof}
It follows directly from Theorem \ref{thm6.2} and Theorem \ref{thm6.3}.
\end{proof}

\begin{proof}[Proof of the compactness part of Theorem \ref{main thm A}]
It is easy to check that, if $K$ satisfies the condition in Theorem \ref{main thm A}, then it must satisfy the condition in the above theorem.
Therefore, we have the lower and upper bounds of $v$. The $C^2$ norm bound of $v$ follows immediately.
\end{proof}

\appendix
\section{Appendix}
\subsection{A Kazdan-Warner identity}\label{A Kazdan-Warner identity}
In this section we are going to show \eqref{K-W condition}, which is a consequence of the following

\begin{prop}\label{K-W prop}\  Let $K>0$ be a $C^{1}$ function on $S^n$, and let $v$ be
a positive function in $C^2(S^n)$ satisfying
\begin{equation}\label{eq:a1}
P_\sigma(v)=Kv^{ \frac{n+2\sigma}{n-2\sigma}}, \quad\mbox{on}\ S^n.
\end{equation}
Then, for any conformal Killing vector field $X$ on $S^n$, we have
\begin{equation}\label{eq:a2}
\int_{S^n} (\nabla_XK)v^{ \frac{2n}{n-2\sigma} }\,\ud V_{g_{\Sn}}=0.
\end{equation}
\end{prop}

 Let $\varphi_t: S^n\to S^n$ be a one parameter family of conformal
diffeomorphism (in this case they are M\"obius transformations),
depending on $t$ smoothly, $|t|<1$, and $\varphi_0=identity$. Then
\begin{equation}\label{one-parameter-mobius}
X:=\frac {d}{dt}(\varphi_t)^{-1}\bigg|_{t=0}\ \ \mbox{is a conformal
Killing vector field on}\ S^n.
\end{equation}

\begin{proof}
The proof is standard (see, e.g., \cite{BE} for a Kazdan-Warner identity for prescribed scalar curvature problems) and we include it here for completeness.
Since $P_{\sigma}$ is a self-adjoint operator, \eqref{eq:a1} has a variational formulation:
\[
I[v]:=  \frac 12 \int_{S^n} v P_{\sigma}(v) \,\ud V_{g_{\Sn}}-\frac {n-2\sigma}{2n} \int_{S^n} Kv^{ \frac{2n}{n-2\sigma}}\,\ud V_{g_{\Sn}}.
\]
Let $X$ be a conformal Killing vector field, then there exists
$\{\varphi_t\}$ satisfying \eqref{one-parameter-mobius}.
Let
$$
v_t:=(v\circ\varphi_t)w_t
$$
where $w_t$ is given by $$ g_t:= \varphi_t^*g_{\Sn}=w_t^{ \frac{4}{n-2\sigma}} g_{\Sn}.$$
Then
$$
I[v_t] =\frac 12 \int_{S^n} v P_{\sigma}(v) \,\ud V_{g_{\Sn}}-\frac {n-2\sigma}{2n} \int_{S^n} K(\varphi_t^{-1}(x))v^{\frac{2n}{n-2\sigma}}\,\ud V_{g_{\Sn}}.
$$
It follows from \eqref{eq:a1} that
$$
0= I'[v]\left(\frac d{dt}\bigg|_{t=0}v_t\right)=\frac d{dt}I[v_t]\bigg|_{t=0}= -\frac {n-2\sigma}{2n} \int_{S^n}
(\nabla_XK)v^{ \frac{2n}{n-2\sigma}}\,\ud V_{g_{\Sn}}.
$$
\end{proof}

\subsection{A proof of Lemma \ref{thm3.2}} \label{A Bocher type theorem}

The classical B\^ocher theorem in harmonic function theory states that a positive harmonic
function $u$ in a punctured ball $B_1\setminus \{0\}$ must be of the form
\[
u(x)=\begin{cases}
-a\log|x|+h(x),\quad & n=2,\\
a|x|^{2-n}+h(x), \quad & n\geq 3,
\end{cases}
\]
where $a$ is a nonnegative constant and $h$ is a harmonic function in $B_1$.

We are going to establish a similar result, Lemma \ref{thm3.2}, in our setting.
Denote $\mathcal{B}_R^+=\{X: |X|<R, t>0 \}$, $\pa'\mathcal{B}_R^+=\{(x,t):|x|<R\}$ and
$\pa''\mathcal{B}=\pa\mathcal{B}_R^+ \setminus  \pa'\mathcal{B}_R^+$.

\begin{proof}[Proof of Lemma \ref{thm3.2}] We adapt the proof of the B\^ocher theorem given in \cite{abr}.

Define
\[
A[U](r)=\frac{\int_{\partial''\B_{r}^+}t^{1-2\sigma}U(x,t)dS_r}{\int_{\partial''\B_{r}^+}t^{1-2\sigma}dS_r}
\]
where $r=|(x,t)|>0$ and $dS_r$ is the volume element of $\partial '' \B_r$.

By direct computation we have
\[
\frac{d}{dr}A[U](r)=\frac{\int_{\partial''\B_{r}^+}t^{1-2\sigma}\nabla U(x,t)\cdot\frac{(x,t)}{r}dS_r}{\int_{\partial''\B_{r}^+}t^{1-2\sigma}dS_r}.
\]
Let
\[
f(r)=\int_{\partial''\B_{r}^+}t^{1-2\sigma}\nabla U(x,t)\cdot\frac{(x,t)}{r}dS_r.
\]
Since $U$ satisfies \eqref{eq:bocher}, by integration by parts we have
\[
f(r_1)=f(r_2), \  \forall \ 0<r_1, r_2<1.
\]
Notice that
\[
\int_{\partial''\B_{r}^+}t^{1-2\sigma}dS_r=r^{n+1-2\sigma}\int_{\partial''\B_{1}^+}t^{1-2\sigma}dS_1.
\]
Thus there exists a constant $b$ such that
\[
\frac{d}{dr}A[U](r)=b r^{-n-1+2\sigma}.
\]
So there exist constants $a$ and $b$ such that
\[
A[U](r)=a+br^{2\sigma-n}.
\]

Since we have the Harnack inequalities for $U$ as in the proof of Lemma \ref{lem4.1}, the rest of the arguments are rather similar to those in \cite{abr} and are omitted here.
We refer to \cite{abr} for details.
\end{proof}

\subsection{Two lemmas on maximum principles}\label{section of two maximum principle}
\begin{lem}
Let $Q_1=B_1\times (0,1)\subset \R^{n+1}_+$, then there exists $\va=\va(n,\sigma)$ such that for all $|a(x)|\leq \va |x|^{-2\sigma}$, if $U\in H(t^{1-2\sigma}, Q_1)$, $U\geq 0$ on $\partial'' Q_1$, and
\[
\int_{Q_1}t^{1-2\sigma} \nabla U \nabla \varphi\geq \int_{B_1} aU(\cdot,0) \varphi\quad\mbox{for all }0\leq\varphi\in C^{\infty}_c(Q_1).
\]
Then
\[
U\geq 0\quad\mbox{in } Q_1.
\]
\end{lem}

\begin{proof}
By a density argument, we can use $U^-$ as a test function. Hence we have
\begin{equation}\label{1}
\int_{Q_1}t^{1-2\sigma} |\nabla U^-|^2\leq \int_{B_1}|a|(U^-(\cdot, 0))^2.
\end{equation}
We extend $U^-$ to be zero outside of $Q_1$ and still denote it as $U^-$. Then the trace
\[
U^-(\cdot, 0)\in \dot H^{\sigma}(\R^n).
\]
Since
\[
\|U^-(\cdot, 0)\|^2_{\dot H^{\sigma}(\R^n)}=\int_{\R^{n+1}_+}t^{1-2\sigma}|\nabla \mathcal P_{\sigma}*U^-(\cdot, 0)|^2 \leq \int_{\R^{n+1}_+}t^{1-2\sigma} |\nabla U^-|^2,
\]
we have
\[
\|U^-(\cdot, 0)\|^2_{\dot H^{\sigma}(\R^n)}\leq \int_{B_1}|a|(U^-(\cdot, 0))^2.
\]
By Hardy's inequality (see, e.g., \cite{Ya})
\[
C(n,\sigma)\int_{\R^n}|x|^{-2\sigma} (U^-(\cdot, 0))^2\leq \|U^-(\cdot, 0)\|^2_{\dot H^{\sigma}(\R^n)}
\]
where $C(n,\sigma)=2^{2\sigma}\frac{\Gamma((n+2\sigma)/4)}{\Gamma((n-2\sigma)/4)}$ is the best constant. Hence if $\va<C(n,\sigma)$, $U^-(\cdot,0)\equiv 0$ and hence by \eqref{1}, $U^-\equiv 0$ in $Q_1$.
\end{proof}

\begin{lem}
Let $a(x)\in L^{\infty}(B_1)$. Let $W\in C(\overline{Q_1})\cap C^2(Q_1)$ satisfying $\nabla_x W\in C(\overline{Q_1})$, $t^{1-2\sigma}\partial_t W\in C(\overline{Q_1})$, and
\begin{equation}
\begin{cases}
-\mathrm{div}(t^{1-2\sigma}\nabla W)&\geq 0\quad\mbox{in }Q_1\\
-\lim\limits_{t\to 0}t^{1-2\sigma}\partial_t W(x,t)&\geq a(x)W(x,0)\quad\mbox{on }\partial ' Q_1\\
\hspace{2cm} W&>0\quad\mbox{in }\overline{Q_1}.
\end{cases}
\end{equation}
If $U\in C(\overline{Q_1})\cap C^2(Q_1)$ satisfying $\nabla_x U \in C(\overline{Q_1})$,  $t^{1-2\sigma}\partial_t U\in C(\overline{Q_1})$, and
\begin{equation}
\begin{cases}
-\mathrm{div}(t^{1-2\sigma}\nabla U)&\geq 0\quad\mbox{in }Q_1\\
-\lim\limits_{t\to 0}t^{1-2\sigma}\partial_t U(x,t)&\geq a(x)U(x,0)\quad\mbox{on }\partial ' Q_1\\
\hspace{2cm} U&\geq 0\quad\mbox{in }\partial '' Q_1.
\end{cases}
\end{equation}
Then $U\geq 0$ in $Q_1$.
\end{lem}

\begin{proof}
Let $V=U/W$. Then
\begin{equation}\label{4}
\begin{cases}
-\mathrm{div}(t^{1-2\sigma}\nabla V)-2t^{1-2\sigma}\frac{\nabla V\nabla W}{W}-\frac{\mathrm{div}(t^{1-2\sigma}\nabla W) V}{W}&\geq 0\quad\mbox{in }Q_1\\
-\lim\limits_{t\to 0}t^{1-2\sigma}\partial_t V+\frac{V}{W}\big(-\lim\limits_{t\to 0}t^{1-2\sigma}\partial_t W(x,t)- a(x)W(x,0)\big)&\geq 0\quad\mbox{on }\partial ' Q_1\\
\hspace{8cm} V&\geq 0\quad\mbox{in }\partial '' Q_1.
\end{cases}
\end{equation}
We are going to show that $V\geq 0$ in $Q_1$. If not, then we choose $k$ such that $\inf_{Q_1}v <k\leq 0$. Let
\[
V_k=V-k\quad\mbox{and }V_k^-=\max(-V_k,0).
\]
Multiplying $V_k^-$ to \eqref{4}, we have
\begin{equation}\label{5}
\int_{Q_1}t^{1-2\sigma}|\nabla V_k^-|^2\leq 2\int_{Q_1}t^{1-2\sigma}W^{-1}V_k^-\nabla V_k^-\nabla W.
\end{equation}

\emph{Case 1:} Suppose $1-2\sigma\leq 0$. Denote $\Gamma_k=Supp(\nabla V_k^-)$. Then by the H\"older inequality and the bounds of $\nabla_x W$, $t^{1-2\sigma}\partial_t W$,
\[
2\int_{Q_1}t^{1-2\sigma}W^{-1}V_k^-\nabla V_k^-\nabla W\leq C \left(\int_{Q_1}t^{1-2\sigma}|\nabla V_k^-|^2\right)^{\frac{1}{2}}\left(\int_{\Gamma_k}t^{1-2\sigma}|V_k^-|^2\right)^{\frac{1}{2}}.
\]
Hence it follows from \eqref{5} that
\begin{equation}\label{6}
\int_{Q_1}t^{1-2\sigma}|\nabla V_k^-|^2\leq C\int_{\Gamma_k}t^{1-2\sigma}|V_k^-|^2.
\end{equation}
Since $V_k^-=0$ on $\partial '' Q_1$, by Lemma 2.1 in \cite{TX}, 
\begin{equation}\label{7}
\left(\int_{Q_1}t^{1-2\sigma}|V_k^-|^{2(n+1)/n}\right)^{\frac{n}{n+1}}\leq C \int_{Q_1}t^{1-2\sigma}|\nabla V_k^-|^2.
\end{equation}
By \eqref{6}, \eqref{7} and H\"older inequality,
\[
\int_{\Gamma_k}t^{1-2\sigma}\geq C.
\]
This yields a contradiction when $k\to\inf_{Q_1}v$, since $\nabla V=0$ on the set of $V\equiv \inf_{Q_1}V$.

\emph{Case 2:} Suppose $1-2\sigma>0$. Denote $\Gamma_k=Supp(V_k^-)$. Then by H\"older inequality and the bounds of $\nabla_x W$, $t^{1-2\sigma}\partial_t W$,
\[
\begin{split}
\int_{Q_1}t^{1-2\sigma}|\nabla V_k^-|^2&\leq 2\int_{Q_1}t^{1-2\sigma}W^{-1}V_k^-\nabla V_k^-\nabla W\\
&\leq C\int_{Q_1}V_k^-\nabla V_k^-\\
&\leq C(\int_{Q_1}t^{1-2\sigma}|\nabla V_k^-|^2)^{1/2}(\int_{Q_1}t^{2\sigma-1}|V_k^-|^2)^{1/2}.
\end{split}
\]
Hence
\[
\int_{Q_1}t^{1-2\sigma}|\nabla V_k^-|^2\int_{Q_1}t^{1-2\sigma}|\nabla V_k^-|^2\leq C\int_{Q_1}t^{1-2\sigma}|\nabla V_k^-|^2\int_{Q_1}t^{2\sigma-1}|V_k^-|^2.
\]
Since $V_k^-=0$ on $\partial'' Q_1$, by the proof of Lemma 2.3 in \cite{TX}, 
for any $\beta>-1$,
\[
\int_{Q_1}t^{\beta}|V_k^-|^2\leq C(\beta)\int_{Q_1}t^{1-2\sigma}|\nabla V_k^-|^2.
\]
In the following we choose $\beta=\sigma-1$. Hence,
\[
\int_{Q_1}t^{1-2\sigma}|\nabla V_k^-|^2\int_{Q_1}t^{\sigma-1}|V_k^-|^2\leq C\int_{Q_1}t^{1-2\sigma}|\nabla V_k^-|^2\int_{Q_1}t^{2\sigma-1}|V_k^-|^2,
\]
i.e.
\[
\int_{\Gamma_k}t^{1-2\sigma}|\nabla V_k^-|^2\int_{\Gamma_k}t^{\sigma-1}|V_k^-|^2\leq C\int_{\Gamma_k}t^{1-2\sigma}|\nabla V_k^-|^2\int_{\Gamma_k}t^{2\sigma-1}|V_k^-|^2.
\]
Fixed $\va>0$ sufficiently small which will be chosen later. By the strong maximum principle $\inf_{Q_1}V$ has to be attained only on $\partial ' Q_1$, then we can choose $k$ sufficiently closed to $\inf_{Q_1}V$ such that $\Gamma_k\subset B_1\times [0,\va]$. Then
\[
\va^{-\sigma}\int_{\Gamma_k}t^{2\sigma-1}|V_k^-|^2\leq C\int_{\Gamma_k}t^{\sigma-1}|V_k^-|^2.
\]
Choose $\va$ small enough such that $\va^{-\sigma}>C+1$. It follows that
\[
\int_{\Gamma_k}t^{1-2\sigma}|\nabla V_k^-|^2\int_{\Gamma_k}t^{2\sigma-1}|V_k^-|^2=0.
\]
Hence one of them has to be zero, which reaches a contradiction immediately.
\end{proof}

\subsection{Complementarities}
\begin{lem}\label{lemma used in weighted trace embedding inequality}
Let $u(x)\in C^{\infty}_c(\mathbb{R}^n)$ and $V(\cdot,t)=\mathcal{P}_\sigma(\cdot, t)*u(\cdot)$. 
For any $U\in C^\infty_c(\mathbb{R}^{n+1}_+ \cup \partial \mathbb{R}^{n+1}_+)$ with $U(x,0)=u(x)$,
\[
 \int_{\mathbb{R}^{n+1}_+} t^{1-2\sigma} |\nabla V|^2\leq  \int_{\mathbb{R}^{n+1}_+} t^{1-2\sigma} |\nabla U|^2.
\]
\end{lem}
\begin{proof}
Let $0\leq \eta(x,t)\leq 1$, $Supp(\eta)\subset \mathcal{B}_{2R}^+$, $\eta=1$ in $\mathcal{B}_R^+$ and $|\nabla \eta|\leq 2/R$. In the end we will let $R\to\infty$ and hence we may assume that $U$ is supported in $\overline{\mathcal{B}_{R/2}^+}$. Since $\mathrm{div}(t^{1-2\sigma}\nabla V)=0$, then
\[
 \begin{split}
  0&=\int_{\mathbb{R}^{n+1}_+}t^{1-2\sigma}\nabla V \nabla(\eta(U-V))\\&
=\int_{\mathbb{R}^{n+1}_+}t^{1-2\sigma}\eta\nabla U\nabla V-\int_{\mathbb{R}^{n+1}_+}t^{1-2\sigma}\eta|\nabla V|^2-\int_{\mathcal{B}_{2R}^+\backslash\mathcal{B}_R^+}t^{1-2\sigma}V \nabla\eta\nabla V
 \end{split}
\]
where we used $\eta(U-V)=0$ on the boundary of $\mathcal{B}_{2R}^+$ in the first equality.

Note that for $(x,t) \in \mathcal{B}_{2R}^+\backslash\mathcal{B}_R^+$
\[
\begin{split}
|V(x,t)|&=\beta(n,\sigma)\left|\int_{\mathbb{R}^{n}}\frac{t^{2\sigma}}{(|x-\xi|^2+t^2)^{\frac{n+2\sigma}{2}}} u(\xi)\,d\xi\right|\\&
\leq \beta(n,\sigma)\int_{\mathbb{R}^{n}}\frac{(|x|^2+t^2)^{\sigma}}{(|x|^2/4+t^2)^{\frac{n+2\sigma}{2}}} |u(\xi)|\,d\xi\\&
\leq C(n,\sigma)(|x|^2+t^2)^{-\frac{n}{2}} \|u\|_{L^1}
\end{split}
\]
where in the first inequality we have used that $U$ is supported in $\overline{\mathcal{B}_{R/2}^+}$.

Direct computations yield
\[
 \begin{split}
  &\left|\int_{\mathcal{B}_{2R}^+\backslash\mathcal{B}_R^+}t^{1-2\sigma}V\nabla\eta\nabla V\right|\\
  &\leq \left(\int_{\mathcal{B}_{2R}^+\backslash\mathcal{B}_R^+}t^{1-2\sigma}|\nabla V|^2\right)^{1/2}\left(\int_{\mathcal{B}_{2R}^+\backslash\mathcal{B}_R^+}t^{1-2\sigma}V^2|\nabla \eta|^2\right)^{1/2}\\
  &\leq\left(\int_{\mathcal{B}_{2R}^+\backslash\mathcal{B}_R^+}t^{1-2\sigma}|\nabla V|^2\right)^{1/2}\\
&\quad\cdot C(n,\sigma)|u|_{L^{1}(\mathbb{R}^n)}(R^{n+2-2\sigma-2-2n})^{1/2}\to 0 \mbox{ as }R\to\infty
 \end{split}
\]
where we used \eqref{norms in two spaces are same} that $\int_{\R^{n+1}_+}t^{1-2\sigma}|\nabla V|^2<\infty$.
Therefore, we have
\[
 \int_{\mathbb{R}^{n+1}_+}t^{1-2\sigma}|\nabla V|^2\leq\left|\int_{\mathbb{R}^{n+1}_+}t^{1-2\sigma}\nabla U\nabla V\right|.
\]
Finally, by H\"older inequality,
\[
\int_{\mathbb{R}^{n+1}_+}t^{1-2\sigma}|\nabla V|^2\leq\int_{\mathbb{R}^{n+1}_+}t^{1-2\sigma}|\nabla U|^2.
\]
\end{proof}

\noindent Tianling Jin

\noindent Department of Mathematics, Rutgers University\\
110 Frelinghuysen Road, Piscataway, NJ 08854, USA\\
Email: \textsf{kingbull@math.rutgers.edu}

\medskip

\noindent YanYan Li

\noindent Department of Mathematics, Rutgers University\\
110 Frelinghuysen Road, Piscataway, NJ 08854, USA\\
Email: \textsf{yyli@math.rutgers.edu}

\medskip

\noindent Jingang Xiong

\noindent School of Mathematical Sciences, Beijing Normal University\\
Beijing 100875, China\\[1mm]
\noindent and\\[1mm]
\noindent Department of Mathematics, Rutgers University\\
110 Frelinghuysen Road, Piscataway, NJ 08854, USA\\[0.8mm]
Email: \textsf{jxiong@mail.bnu.edu.cn/jxiong@math.rutgers.edu}

\end{document}